\documentclass[11pt]{article}
\usepackage{amsgen,amsmath,amstext,amsbsy,amsopn,amssymb, bbm, algorithm, algorithmic, amsthm, enumerate, color, graphicx, tikz, natbib, pdflscape, multirow}
\usepackage[titletoc]{appendix}
\usepackage[pdftex, colorlinks]{hyperref}\usepackage[left=2.54cm,top=2.54cm,right=2.54cm,bottom=2.54cm]{geometry}
\usetikzlibrary{intersections}

\newtheorem{Corollary}{Corollary}

\newtheorem{Lemma}{Lemma}

\newtheorem{Theorem}{Theorem}

\newtheorem{Remark}{Remark}
\newtheorem{Assumption}{Assumption}

\newcommand{\ba}{{\mathbf{a}}}

\newcommand{\e}{{\mathbf{e}}}
\newcommand{\bv}{{\mathbf{v}}}
\newcommand{\x}{{\mathbf{x}}}

\newcommand{\bA}{{\mathbf{A}}}
\newcommand{\B}{{\mathbf{B}}}
\newcommand{\R}{\mathbb{R}}

\newcommand{\I}{{\mathbf{I}}}
\newcommand{\M}{{\mathbf{M}}}

\newcommand{\bS}{{\mathbf{S}}}

\newcommand{\cU}{{\mathcal{U}}}
\newcommand{\X}{{\mathbf{X}}}
\newcommand{\Y}{{\mathbf{Y}}}
\newcommand{\Z}{{\mathbf{Z}}}

\newcommand{\bbeta}{{\boldsymbol{\beta}}}
\newcommand{\hbeta}{{\hat{\boldsymbol{\beta}}}}

\newcommand{\bmu}{{\boldsymbol{\mu}}}
\newcommand{\bomega}{{\boldsymbol{\omega}}}

\newcommand{\tomega}{{\tilde{\omega}}}
\newcommand{\hrho}{{\hat{\boldsymbol{\rho}}}}
\newcommand{\brho}{{\boldsymbol{\rho}}}
\newcommand{\boldr}{{\boldsymbol{r}}}

\newcommand{\hrhosij}{{\boldr^S_{ij}}}

\newcommand{\hsigma}{{\hat{\sigma}}}
\newcommand{\htau}{{\boldr^K}}
\newcommand{\htauij}{{\boldr^K_{ij}}}
\newcommand{\bOmega}{{\boldsymbol{\Omega}}}
\newcommand{\cOmega}{{\check{\boldsymbol{\Omega}}}}
\newcommand{\hOmega}{{\hat{\boldsymbol{\Omega}}}}
\newcommand{\tOmega}{{\tilde{\boldsymbol{\Omega}}}}
\newcommand{\bSigma}{{\boldsymbol{\Sigma}}}
\newcommand{\hSigma}{{\hat{\boldsymbol{\Sigma}}}}
\newcommand{\tSigma}{{\tilde{\boldsymbol{\Sigma}}}}
\newcommand{\signbar}{\overline{\sign}}
\newcommand{\defn}{\ensuremath{:=}}
\newcommand{\cSigma}{{\check{\boldsymbol{\Sigma}}}}

\newcommand{\supp}{\operatorname{supp}}
\newcommand{\bGamma}{{\boldsymbol{\Gamma}}}
\newcommand{\bB}{{\boldsymbol{B}}}

\newcommand{\0}{{\mathbf{0}}}
\newcommand{\1}{{\mathbbm{1}}}

\newcommand{\half}{\frac{1}{2}}
\newcommand{\median}{\operatornamewithlimits{median}}
\newcommand{\argmin}{\operatornamewithlimits{argmin}}
\newcommand{\argmax}{\operatornamewithlimits{argmax}}
\newcommand{\Corr}{{\rm Corr}}
\newcommand{\Cov}{{\rm Cov}}
\newcommand{\Var}{{\rm Var}}
\newcommand{\diag}{{\rm diag}}
\newcommand{\sign}{{\rm sign}}
\newcommand{\rank}{{\rm rank}}
\newcommand{\tr}{{\rm tr}}

\begin{document}
\title{High-dimensional robust precision matrix estimation: \\Cellwise corruption under $\epsilon$-contamination}
\author{Po-Ling Loh and Xin Lu Tan
\vspace{.3cm} \\
Department of Statistics \\
The Wharton School \\
University of Pennsylvania \\
Philadelphia, PA 19104
}
\maketitle

\begin{abstract}
We analyze the statistical consistency of robust estimators for precision matrices in high dimensions. We focus on a contamination mechanism acting cellwise on the data matrix. The estimators we analyze are formed by plugging appropriately chosen robust covariance matrix estimators into the graphical Lasso and CLIME.  Such estimators were recently proposed in the robust statistics literature, but only analyzed mathematically from the point of view of the breakdown point. This paper provides complementary high-dimensional error bounds for the precision matrix estimators that reveal the interplay between the dimensionality of the problem and the degree of contamination permitted in the observed distribution. We also show that although the graphical Lasso and CLIME estimators perform equally well from the point of view of statistical consistency, the breakdown property of the graphical Lasso is superior to that of CLIME. We discuss implications of our work for problems involving graphical model estimation when the uncontaminated data follow a multivariate normal distribution, and the goal is to estimate the support of the population-level precision matrix. Our error bounds do not make any assumptions about the the contaminating distribution and allow for a nonvanishing fraction of cellwise contamination.
\end{abstract}

%%%%%%%%%%%%%%%%%%%%%%%%%%%%%%%%%%%%%%%%%%%%%%%%%%%%%%%%%%%%%%%%%%%%%%%%%%%%%%

\section{Introduction}

Covariance matrix estimation has long taken center stage in multivariate analysis~\citep{And03}. The sample covariance estimator, which originates as the maximum likelihood estimator under a multivariate normal model, is optimal in many respects: It is unbiased, consistent, efficient under various distributional assumptions, and easy to compute. Despite its many positive traits,  however, the sample covariance matrix is also highly non-robust when data are observed subject to contamination. Hence, various procedures in robust statistics have been derived to obtain a covariance matrix estimator that behaves well even in the presence of contaminated data \citep{Huber1981, HamEtal11}.

In other areas of multivariate analysis, the precision matrix $\bOmega^* := (\bSigma^*)^{-1}$ is of significant interest. Examples include computing Mahalanobis distances, linear discriminant analysis, and Gaussian graphical models. In the setting of graphical models, a random vector $\X$ is associated with an undirected graph $G = (V, E)$ that encodes the conditional independence relations between components of $\X$ \citep{Lau96}. The vertex set $V$ contains $X_1, \ldots, X_p$, while the edge set $E$ consists of pairs $(i, j)$, where $(i, j)\in E$ if $X_i$ and $X_j$ are connected by an edge. For each non-edge $(i, j)\not\in E$, the variables $X_i$ and $X_j$ are conditionally independent given all other variables $X_{\backslash\{i, j\}} := V\backslash\{X_i, X_j\}$. When $\X \sim N(\bmu, \bSigma^*)$, pairwise conditional independence holds if and only if $\bOmega^*_{ij} = 0$. Thus, recovering the support of the precision matrix is equivalent to graphical model selection. The aforementioned observations have been used for network reconstruction in many scientific fields, including genetics and neuroscience (e.g., see \cite{WehEtal06, SmiEtal11}, and the references cited therein). When the dimensionality $p$ is small compared to the number of samples $n$, a reasonable method for robust precision matrix estimation could consist of computing a robust estimate of the covariance matrix and then taking a matrix inverse.

With the recent deluge of high-dimensional data, however, a need has arisen to obtain high-dimensional analogs of classical statistical procedures that are both computable and possess rigorous theoretical guarantees. Although several methods, notably the graphical Lasso (GLasso) \citep{YuanLin2007, FriEtal08} and the method of constrained $\ell_1$-minimization for inverse matrix estimation (CLIME) \citep{Cai2011}, have been proposed for high-dimensional precision matrix estimation, robust estimation of high-dimensional precision matrices has only recently emerged in the literature. The GLasso and CLIME estimators themselves tend to perform poorly under contaminated data, since they take as input the sample covariance matrix that is sensitive to even a single outlier.

Popular classical robust covariance estimators are applicable in settings where less than half of the observation vectors are contaminated.  Such assumption is closely connected to the Tukey-Huber contamination model that underlies much of the existing robustness theory \citep{Tukey1962, Huber1964}.  In the Tukey-Huber contamination model, a mixture distribution with a dominant nominal component (such as a multivariate normal distribution) and a minority unspecified component are posited, and each observation vector is either completely clean or completely spoiled.  Classical robust covariance estimators then involve downweighting contaminated observations in order to reduce their influence. When the dimension $p$ is large, however, the fraction of perfectly observed data vectors may be rather small: If all components of an observation vector had an independent chance of being contaminated, most observation vectors would be contaminated. Thus, downweighting an entire observation would waste the information contained in the clean components of the observation vector.  This describes the setting of the \emph{cellwise} contamination model, which was developed by \cite{Alqallaf2002}. It generalizes the classical Tukey-Huber contamination model, which may be viewed as a case of \emph{rowwise} contamination of the data matrix, and is fairly realistic for applications involving measurement error in DNA microarray analysis~\citep{TroEtal01} or dropout measurements in sensor arrays~\citep{Swa00}.

On the other hand, most existing approaches for robust covariance estimation focus on affine equivariance.  These include the $M$-estimators \citep{Maronna1976}, Minimum Volume Ellipsoid (MVE) and Minimum Covariance Determinant (MCD) estimators \citep{Rousseeuw1984, Rousseeuw1985}, and the Stahel-Donoho (SD) estimator \citep{Stahel1981, Donoho1982}.  Although affine equivariance may be a desirable property under rowwise contamination, it is less appropriate in the setting of cellwise contamination, since linear combinations of observation vectors lead to a propagation of outliers \citep{Alqallaf2009}. In addition, the MVE, MCD, and SD estimators all require heavy computational effort, rendering them impractical for high-dimensional datasets.  To deal with cellwise contamination, \cite{VanAelst2014} proposed a modified SD estimator that adapts winsorization \citep{Huber1981, Alqallaf2002} and a cellwise weighting scheme.  Similar to the original SD estimator, however, computation is only feasible for small $p$.  A recent approach by \cite{Agostinelli2014} is capable of dealing with both rowwise and cellwise outliers. The procedure consists of two steps: (1) flag cellwise outliers as missing values; and (2) apply a rowwise robust method to the incomplete data.  However, computation is again infeasible in high dimensions. Other recent proposals for robust high-dimensional covariance matrix estimation include \cite{Chen2015} and \cite{HanEtal15}, but both methods treat different contamination models and are not suitable to handle data with cellwise contamination.

In contrast, relatively few approaches exist for robust high-dimensional precision matrix estimation under any form of contamination. One method is supplied by the TLasso estimator of \cite{FinegoldDrton2011}, which builds upon the GLasso and models the data as coming from the multivariate $t$-distribution, a long-tailed surrogate for the multivariate normal distribution.  The ``alternative multivariate $t$-distribution" is used to model a case where different coordinates of the distribution are obtained from the latent multivariate normal distribution using different weights. Although the TLasso demonstrates a higher degree of robustness than the GLasso under both rowwise and cellwise contamination in simulations, however, a theoretical analysis from the point of view of robust statistics has not been derived. More recently, \cite{Croux2014} and \cite{Tarr2015} propose a promising new method for high-dimensional precision matrix estimation, designed specifically for cellwise contamination. The method consists of combining a robust covariance estimator that may be computed efficiently with a suitable high-dimensional precision matrix estimation procedure. Whereas \cite{Tarr2015} focus on developing new methodology and \cite{Croux2014} analyze breakdown behavior of the precision matrix estimators, however, a rigorous high-dimensional analysis from the point of view of statistical consistency has not been conducted.

In this paper, we focus on high-dimensional robust estimation of precision matrices under the cellwise contamination model, using the estimators proposed by \cite{Croux2014} and \cite{Tarr2015}. Formally, we derive statistical error bounds in elementwise $\ell_\infty$-norm for robust precision matrix estimation procedures under an $\epsilon$-contamination model, where at most an $\epsilon$ fraction of entries in the data matrix are corrupted by outliers. Our work fuses two threads of research involving classical robust procedures and high-dimensional statistical estimation in a novel and rigorous manner. The bounds we derive match standard high-dimensional bounds for uncontaminated precision matrix estimation, up to a constant multiple of $\epsilon$. Furthermore, they are of a complementary nature to \cite{Croux2014}, since we are primarily concerned with robustness as measured from the viewpoint of statistical consistency, rather than breakdown behavior.

More generally, our results reveal an interesting interplay between bounds for statistical error under $\epsilon$-contamination and classical measures of robustness such as the influence function \citep{Hampel1974} and breakdown point \citep{DonohoHuber1983}. Estimators with bounded influence have long been favored in classical robust statistics, as the rate of change in the statistical functional associated with the estimator is controlled when the nominal distribution is contaminated by an arbitrary point mass distribution. Our results show that a variety of bounded influence estimators, including Kendall's and Spearman's correlation coefficients, give rise to (inverse) covariance estimators with statistical error rates that depend linearly on the degree of contamination; the converse relationship may be seen to hold more generally as a result of our proof arguments. On the other hand, our discussion of the breakdown point of the precision matrix estimators, building upon the analysis of \cite{Croux2014}, emphasizes the significant differences between the notions of breakdown point and statistical consistency. Whereas our analysis shows that the robust CLIME and GLasso procedures have comparable behavior from the point of view of high-dimensional statistical consistency, the CLIME estimator has a substantially smaller breakdown point than the GLasso, due to its constrained feasibility region. Rather than advocating one measure of robustness over another, our discussion emphasizes the value of considering different quantitative measures of robustness in selecting an appropriate estimator.

The remainder of the paper is organized as follows: Section~\ref{SecBackground} furnishes the mathematical background for the cellwise contamination model and the robust covariance and precision matrix estimators to be considered in the paper. Section~\ref{SecMain} presents our main theoretical contributions, providing bounds on the statistical error of the covariance and precision matrix estimators under the cellwise contamination model, as well as concrete consequences in the presence of outliers and/or missing data. Section~\ref{SecBreakdown} provides a discussion of the breakdown point for the robust GLasso and CLIME estimators. In Section~\ref{SecProofs}, we discuss the main steps of the proofs of our theorems. Section~\ref{SecSimulations} contains simulation results that are used to validate the theoretical results of the paper.  We conclude with a discussion in Section~\ref{SecDiscussion}, including some avenues for future research.

\paragraph{\textbf{Notation:}} For a vector $\ba = (a_1, \ldots, a_p)^T \in\R^p$, we denote by $\|\ba\|_1 = \sum_{i=1}^p |a_i|$ and $\|\ba\|_2 = (\sum_{i=1}^pa_i^2)^{1/2}$ the $\ell_1$-norm and $\ell_2$-norm of $\ba$, respectively.  For a matrix $\bA = (a_{ij}) \in\R^{p\times q}$, we define the elementwise $\ell_\infty$-norm $\|\bA\|_\infty = \max_{1\leq i\leq p, 1\leq j\leq q}|a_{ij}|$, the elementwise $\ell_1$-norm $\|\bA\|_1 = \sum_{i=1}^p\sum_{j=1}^q|a_{ij}|$, the Frobenius norm  $\|\bA\|_F = (\sum_{i=1}^p\sum_{j=1}^q a_{ij}^2)^{1/2}$, the spectral norm $\|\bA\|_2 = \sup_{\|x\|\leq 1}\|\bA\x\|_2$, and the matrix $\ell_1$-norm $\|\bA\|_{L_1} = \max_{1\leq j\leq q}\sum_{i=1}^p |a_{ij}|$. We use $\lambda_1(\bA) \ge \lambda_2(\bA) \ge \cdots \geq \lambda_p(\bA)$ to denote the ordered eigenvalues of $\bA$, and we write $\bA\succ 0$ (respectively, $\bA\succeq 0$) to indicate that $\bA$ is positive definite (respectively, positive semidefinite).  We write $\I$ for the identity matrix and $\0$ for the vector of all zeros (the respective dimension of which will be clear from context). The binary operation $\otimes$ denotes the tensor product.

%%%%%%%%%%%%%%%%%%%%%%%%%%%%%%%%%%%%%%%%%%%%%%%%%%%%%%%%%%%%%%%%%%%%%%%%%%%%%%

\section{Background and Problem Setup}
\label{SecBackground}

We begin with a description of the cellwise contamination model, followed by a rigorous formulation of the robust covariance and precision matrix estimators to be studied in our paper.

Following the notation of \cite{Alqallaf2002, Alqallaf2009}, we write the cellwise contamination model in the following form:
\begin{equation}
\label{cellContam}
\X_k = (\I-\B_k)\Y_k + \B_k\Z_k, \qquad \forall k = 1, \ldots, n.
\end{equation}
Here, we observe the contaminated random vector $\X_k\in\R^p$.  The unobservable random vectors $\Y_k, \Z_k$, and $\B_k$ are independent, and $\Y_k\sim G$ (a nominal distribution) and $\Z_k\sim H^*$ (an unspecified outlier generating distribution). Furthermore, $\B_k = \diag(B_{k1}, \ldots, B_{kp})$ is a diagonal matrix, where $B_{k1}, \ldots, B_{kp}$ are independent Bernoulli random variables with $P(B_{ki} = 1) = \epsilon_i$, for all $1 \le i \le p$.

When $\epsilon_1 = \cdots = \epsilon_p = \epsilon$, the probability of an observation vector having no contamination in any component is $(1-\epsilon)^p$, a quantity that decreases exponentially as the dimension increases.  This probability goes below the critical value $1/2$ for $p\geq 14$ at $\epsilon = 0.05$, and for $p\geq 69$ at $\epsilon = 0.01$. Equation~\eqref{cellContam} is a special case of a more general model, where we allow other joint distributions for $B_{k1}, \ldots, B_{kp}$. For instance, if $B_{k1}, \ldots, B_{kp}$ were completely dependent (i.e., \mbox{$P(B_{k1} = \cdots = B_{kp}) = 1$}), we would obtain the rowwise contamination model.  In that case, the probability of an observation vector being totally free of contamination would be $1-\epsilon$, which is independent of the dimension.  \cite{Alqallaf2009} also uses the terms \emph{fully independent contamination model (FICM)} and \emph{fully dependent contamination model (FDCM)} to denote the cellwise and rowwise contamination settings, in order to distinguish the pattern of contamination across rows of the data matrix.

Throughout, we will work under the cellwise contamination model~\eqref{cellContam}, and assume that $G$ is a multivariate normal distribution $N(\bmu, \bSigma^*)$.  Our goal is to estimate the matrices $\bSigma^*$ and $\bOmega^* = (\bSigma^*)^{-1}$ from the (uncontaminated) normal component.

%%%%

\subsection{Covariance Matrix Estimation}

Note that when $\epsilon = 0$ (i.e., the data are uncontaminated), we may use the classical sample covariance matrix estimator $\tSigma$, defined pairwise as\begin{equation*}
\tSigma_{ij} = \frac{1}{n-1}\sum_{k=1}^n(X_{ki} - \bar{X}_i)(X_{kj} - \bar{X}_j), \qquad \forall 1 \le i, j \le p,
\end{equation*}
where $\bar{X}_i = (1/n)\sum_{k=1}^nX_{ki}$ and $\bar{X}_j = (1/n)\sum_{k=1}^nX_{kj}$. When $n \gg p$, the sample covariance is an efficient estimator for $\bSigma^*$. However, when $\epsilon>0$, the performance of $\tSigma$ may be compromised depending on the properties of $H^*$:  Under the cellwise contamination model, for $i \neq j$, we have
\begin{align*}
\left(\bSigma^*_X\right)_{ij} & = (1-\epsilon_i)(1-\epsilon_j) \left(\bSigma^*_Y\right)_{ij} + \epsilon_i \epsilon_j \left(\bSigma^*_Z\right)_{ij} \\
& = \left(\bSigma^*_Y\right)_{ij} - (\epsilon_i + \epsilon_j - \epsilon_i \epsilon_j) \left(\bSigma^*_Y\right)_{ij} + \epsilon_i \epsilon_j \left(\bSigma^*_Z\right)_{ij}.
\end{align*}
When no restrictions are placed on the covariance $\bSigma^*_Z$ of the contaminating distribution, the elementwise deviations between $\bSigma^*_X$ and $\bSigma^*_Y$ (and consequently, also the sample covariance $\tSigma_X := \tSigma$ and $\bSigma^*_Y$) will in general behave arbitrary badly. Furthermore, note that even when $\bSigma^*_Z$ is constrained to lie in a space where the deviations between $\bSigma^*_X$ and $\bSigma^*_Y$ are suitably bounded, we would require the contaminating distribution to have properties such as sub-Gaussian tails in order to ensure consistency of the sample covariance estimator on the order of $O\left(\sqrt{\frac{\log p}{n}}\right)$. When a procedure based on covariance estimation is used to estimate the precision matrix, the errors incurred during the covariance estimation step would propagate to the next step. For instance, this issue would arise in using the CLIME or GLasso estimator. In contrast, our theory for robust covariance estimators will not require any assumptions on either $\bSigma^*_Z$ or the tail behavior of the contaminating distribution.

To deal with cellwise contamination in the high-dimensional setting, we therefore take the pairwise approach suggested by \cite{Croux2014}, where a robust covariance or correlation estimate is computed for each pair of variables. Early proposals of robust procedures are of this type \citep{Bickel1964, Puri1971}, where a coordinatewise approach is taken for robust estimation of location. In addition to having relatively low computational complexity, the pairwise approach is appealing because a high breakdown point of the pairwise estimators translates into a high breakdown point of the overall covariance matrix.
%These estimators are fast to compute, but lack affine equivariance and (sometimes) positive definiteness. The lack of positive definiteness is due to the fact that a correction is often made to ensure consistency for normal data.  Such correction then destroys the positive definiteness of the original estimator. In \cite{Alqallaf2002}, the classical sample correlation is applied on the Winsorized data to obtain a robust correlation estimator.  Rather than transforming the underlying data, \cite{Tarr2015} constructs robust pairwise covariance estimators from robust scale estimators, using the fact that $\bSigma_{ij}^* = [\Var(\alpha X_{ki}+\beta X_{kj})-\Var(\alpha X_{ki}-\beta X_{kj})]/4\alpha\beta$.  In a similar spirit, observing that $\bSigma_{ij}^* = \sigma_i\sigma_j\brho_{ij}$, where $\sigma_i = [\Var(X_{ki})]^{1/2}, \sigma_j = [\Var(X_{kj})]^{1/2}$ and $\brho_{ij} = \Corr(X_{ki}, X_{kj})$, \cite{Croux2014} uses robust scale and correlation estimators as building blocks for their robust covariance estimator.
For $1 \le i, j \le p$, we write
\begin{equation}
\label{sigmaij}
\bSigma_{ij}^* = \sigma_i\sigma_j\brho_{ij},
\end{equation} 
where $\sigma_i = [\Var(X_{ki})]^{1/2}$, $\sigma_j = [\Var(X_{kj})]^{1/2}$, and $\brho_{ij} = \Corr(X_{ki}, X_{kj})$. We will take suitable robust estimators of $\hsigma_i$, $\hsigma_j$, and $\hrho_{ij}$, to obtain the covariance matrix estimator $\hSigma$, with $(i, j)$ entry $\hSigma_{ij} = \hsigma_i\hsigma_j\hrho_{ij}$.

To estimate $\sigma_i$, we consider the median absolute deviation from the median (MAD), a robust measure of scale.  The MAD estimator was popularized by \cite{Hampel1974}, who attributes the concept to Gauss.  It has a breakdown point of $50\%$.  Let $X_{(1), i}\leq\cdots\leq X_{(n), i}$ denote the ordered values of $X_{1i}, \ldots, X_{ni}$.  The sample median $\hat{m}_i$ and the sample MAD $\hat{d}_i$ are defined, respectively, as 
\[\hat{m}_i = X_{(k^*), i}, \quad \text{and} \quad \hat{d}_i = W_{(k^*), i},\]
where $W_{ki} = |X_{ki} - \hat{m}_i|$, for all $k = 1, \ldots, n$, and $k^* = \lceil n/2\rceil$. Expressed another way,
\begin{equation}
\label{EqnSmiley}
\hat{d}_i = \median_{1\leq k\leq n}\bigg(\Big|X_{ki}-\median_{1\leq \ell\leq n}(X_{\ell i})\Big|\bigg).
\end{equation}
We then estimate $\sigma_i$ by $\hsigma_i = [\Phi^{-1}(0.75)]^{-1}\hat{d}_i$, where the constant $[\Phi^{-1}(0.75)]^{-1}$ is chosen in order to make the estimator consistent for $\sigma_i$ at normal distribution. The population-level median of a distribution with cdf $F$ is defined to be $m(F) := F^{-1}\left(0.5\right)$, where  $F^{-1}(c) = \inf\{x: F(x) \ge c\}$, for $c \in [0,1]$. Similarly, we may define the population-level MAD $d(F)$ to be the median of the distribution of $|X - m(F)|$, where $X$ has cdf $F$.

To estimate $\brho_{ij}$, we consider the classical nonparametric correlation estimators, Kendall's tau and Spearman's rho:

\paragraph{\textbf{Kendall's tau:}} This statistic is given by
\begin{equation}
\label{tau}
\boldr^K_{ij} = \frac{2}{n(n-1)}\sum_{k<\ell}\sign(X_{ki}-X_{\ell i})\sign(X_{kj}-X_{\ell j}),
\end{equation}
where $\sign(X) = 1$ if $X>0$, $\sign(X) = -1$ if $X<0$, and $\sign(0) = 0$.

\paragraph{\textbf{Spearman's rho:}} This statistic is given by
\begin{equation}
\label{rhos}
\boldr^S_{ij} = \frac{\sum_{k=1}^n[\rank(X_{ki})-(n+1)/2][\rank(X_{kj})-(n+1)/2]}{\sqrt{\sum_{k=1}^n[\rank(X_{ki})-(n+1)/2]^2\sum_{k=1}^n[\rank(X_{kj})-(n+1)/2]^2}},
\end{equation}
where $\rank(X_{ki})$ denotes the rank of $X_{ki}$ among $X_{1i}, \ldots, X_{ni}$. \\

%\paragraph{\textbf{Quadrant correlation:}} This statistic is given by
%\begin{equation*}
%r_{ij}^Q = \frac{1}{n} \sum_{i=1}^n \sign\left((X_{ij} - \med_{\ell = 1, \dots, n} X_{\ell j}) (X_{ik} - \med_{\ell = 1, \dots, n} X_{\ell k})\right).
%\end{equation*}

The population versions of the estimators are given, respectively, by
\begin{subequations}
\begin{equation}
\label{EqnRhoKpop}
\brho_{ij}^K = E[\sign(X_{1i}-X_{2i})\sign(X_{1j}-X_{2j})],
\end{equation}
\begin{equation}
\label{EqnRhoSpop}
\brho_{ij}^S = 3E[\sign(X_{1i}-X_{2i})\sign(X_{1j}-X_{3j})].
%\rho_{ij}^Q & = \text{{\red{ fill in}}}.
\end{equation}
\end{subequations}

When $\epsilon_1 = \cdots = \epsilon_p = 0$, we have $\X_k\sim N(\bmu, \bSigma^*)$; in this case, it is known that \citep{Kendall1948, Kruskal1958}
\[\brho_{ij} = \sin\left(\frac{\pi}{2}\brho_{ij}^K\right) = 2\sin\left(\frac{\pi}{6}\brho_{ij}^S\right).\]
Hence, for asymptotic consistency at normal distribution, our estimator for $\brho_{ij}$ is the transformed version of Kendall's tau and Spearman's rho, given by $\sin(\frac{\pi}{2}\boldr^K_{ij})$ and $2\sin(\frac{\pi}{6}\boldr^S_{ij})$, respectively. We then define as $\hSigma$ our robust covariance matrix estimator, with
\begin{equation}
\hSigma_{ij}^K = \hsigma_i\hsigma_j\sin\Big(\frac{\pi}{2}\boldr^K_{ij}\Big), \qquad\text{and}\qquad\hSigma_{ij}^S = 2\hsigma_i\hsigma_j\sin\Big(\frac{\pi}{6}\boldr^S_{ij}\Big).
\end{equation}

%%%%%

\subsection{Precision Matrix Estimation}

A long line of literature exists for precision matrix estimation in the high-dimensional setting. We will focus our attention on sparse precision matrix estimation; i.e., $\bOmega^*$ contains many zero entries. In this section, we review two techniques, the GLasso and CLIME, which produce a sparse precision matrix estimator based on optimizing a function of the sample covariance matrix. As proposed by \cite{Croux2014} and \cite{Tarr2015}, these methods may easily be modified to obtained robust versions, where the sample covariance matrix estimator is simply replaced by a robust covariance estimator $\hSigma$ as described in the previous section.

The graphical lasso (GLasso) estimator \citep{YuanLin2007, FriEtal08} is defined as the maximizer of the following $\ell_1$-penalized log-likelihood function:
\begin{equation*}
\tOmega = \argmin_{\bOmega\succ 0} \big\{\tr(\tSigma\bOmega) - \log\det(\bOmega) + \lambda\|\bOmega\|_1\big\}.
\end{equation*}
Here, $\lambda>0$ is a tuning parameter that controls the sparsity of the resulting precision matrix estimator.

In this paper, we replace the sample covariance matrix $\tSigma$ by the robust alternative $\hSigma$, and consider a variant where only the off-diagonal entries of the estimator are penalized:
\begin{equation}
\label{GLasso}
\hOmega = \argmin_{\bOmega\succ 0} \big\{\tr(\hSigma\bOmega) - \log\det(\bOmega) + \lambda\|\bOmega\|_{1, \text{off}}\big\}.
\end{equation}
Note that although the program~\eqref{GLasso} is convex for any choice of $\hSigma \in \mathbb{R}^{p \times p}$, several state-of-the-art algorithms for optimizing the GLasso require the matrix $\hSigma$ to be positive semidefinite~\citep{FriEtal08, Zhao2014, HsiEtal11}. We will first derive statistical theory for the robust GLasso without a positive semidefinite projection step, and then discuss properties of the projected version in Section~\ref{SecBreakdown}.

A popular alternative to the GLasso is the method of constrained $\ell_1$-minimization for inverse matrix estimation (CLIME) proposed in \cite{Cai2011}.  The CLIME routine solves the following convex optimization problem by linear programming:
\[\tOmega = \argmin_{\bOmega\in\R^{p\times p}} \|\bOmega\|_1 \qquad\text{subject to}\qquad \|\tSigma\bOmega - \I\|_\infty\leq\lambda.\]
Note that here, no symmetry condition is imposed on $\bOmega$, and the solution is not symmetric in general. If a symmetric precision matrix estimate is desired, we may perform a post-symmetrization step on $\tOmega= (\tomega_{ij}^1)$ to obtain the symmetric matrix $\tOmega_{\text{sym}}$, defined by
\begin{align}
\tOmega_{\text{sym}} &= (\tomega_{ij}), \qquad\text{where} \nonumber\\
\tomega_{ij} &= \tomega_{ji} = \tomega_{ij}^1\1(|\tomega_{ij}^1|\leq |\tomega_{ji}^1|) + \tomega_{ji}^1\1(|\tomega_{ij}^1| > |\tomega_{ji}^1|). \label{postSymm}
\end{align}
In other words, between $\tomega_{ij}^1$ and $\tomega_{ji}^1$, we pick the entry with smaller magnitude.  Similar to the GLasso case, we will robustify the CLIME estimator by solving 
\begin{equation}
\label{CLIME}
\hOmega = \argmin_{\bOmega\in\R^{p\times p}} \|\bOmega\|_1 \qquad\text{subject to}\qquad \|\hSigma\bOmega - \I\|_\infty\leq\lambda,
\end{equation}
and then apply the post-symmetrization step~\eqref{postSymm} to $\hOmega$ to obtain the final robust CLIME estimator $\hOmega_{\text{sym}}$.

%%%%%%%%%%%%%%%%%%%%%%%%%%%%%%%%%%%%%%%%%%%%%%%%%%%%%%%%%%%%%%%%%%%%%%%%%%%%%%

\section{Main Results and Consequences}
\label{SecMain}

In this section, we provide rigorous statements of the main results of the paper. We begin by deriving bounds for robust covariance matrix estimation, which are used to obtain bounds on the error incurred by the precision matrix estimator. Note, however, that the statistical error bounds presented in Section~\ref{SecCovEst} are of independent interest, and we believe they are the first bounds appearing in the literature that quantify the robustness of covariance matrix estimators under a cellwise contamination model.

\subsection{Covariance Matrix Estimation}
\label{SecCovEst}

Throughout this section, we will assume that the standard deviations of the uncontaminated distributions are bounded as follows:
\begin{equation}
\label{EqnSigmaBd}
0<\min_{1\leq i\leq p}\sigma_i \leq \max_{1\leq i\leq p}\sigma_i \leq M_\sigma.
\end{equation}
We also define the expression
\begin{equation*}
c(\sigma_i) = \frac{15}{64\sqrt{2\pi} \sigma_i} \exp\left(- \frac{(1.1\sigma_i + 0.5)^2}{2\sigma_i^2}\right), \qquad \forall 1 \le i \le p.
\end{equation*}

Our first theorem provides a bound on the statistical error of the robust covariance estimator $\hSigma^K$ based on Kendall's tau correlations. Note that our result does \emph{not} involve any assumptions on the nature of the contaminating distribution $H$. Thus, the distribution $H$ may contain point masses, and we do not require a probability density function of $H$ to even exist.

\begin{Theorem}
\label{ThmRobustCovK}
Under the cellwise contamination model \eqref{cellContam}, suppose inequality~\eqref{EqnSigmaBd} is satisfied,
and $\epsilon = \max_{1\leq i\leq p}\epsilon_i \le 0.02$. Let $C > \pi\sqrt{2}$ and $C' > \frac{1}{\Phi^{-1}(0.75) \min_{1 \le i \le p} c(\sigma_i) \sqrt{2}}$, and suppose
\begin{equation}
\label{EqnContamScaling}
\max\left\{C \sqrt{\frac{\log p}{n}} + 26\pi \epsilon, \; C' \sqrt{\frac{\log p}{n}} + 7.2 M_\sigma \epsilon\right\} \le 1,
\end{equation}
and $\Phi^{-1}(0.75) C' \sqrt{\frac{\log p}{n}} < 1$. Then with probability at least
\begin{equation*}
1 - 2p^{-\left(\frac{C^2}{\pi^2}-2\right)} - 6p^{-\{2[\Phi^{-1}(0.75)]^2C'^2\min_{1\leq i\leq p}c^2(\sigma_i)-1\}},
\end{equation*}
the robust covariance estimator satisfies
\begin{equation}
\label{EqnCovKendall}
\left\|\hSigma^K - \bSigma^*\right\|_{\infty} \le \left(C(M_\sigma^2 + M_\sigma + 1) + C' (2M_\sigma + 1)\right) \sqrt{\frac{\log p}{n}} + \left(97 M_\sigma^2 + 89M_\sigma + 82\right) \epsilon.
\end{equation}
\end{Theorem}
The proof of Theorem~\ref{ThmRobustCovK} is provided in Section~\ref{SecThmRobustCovK}.

\begin{Remark}
Theorem~\ref{ThmRobustCovK} clearly illustrates the effect of $\epsilon$-contamination on the estimation error of the covariance matrix estimator.  Note that when $\epsilon = 0$, we recover the minimax optimal rate for covariance matrix estimation in $\ell_\infty$-norm \citep{CaiZho12}; although the estimator $\hSigma^K$ is not equal to the sample covariance estimator in the uncontaminated case, the robust covariance estimator nonetheless converges to the true covariance matrix at the optimal rate. On the other hand, cellwise contamination introduces an extra term that is linear in $\epsilon$.

Another way to interpret the bound~\eqref{EqnCovKendall} is that if the level of contamination $\epsilon$ is bounded by a constant times $\sqrt{\frac{\log p}{n}}$, then the robust covariance estimator $\hOmega^K$ will enjoy the same statistical error rate as the optimal covariance estimator in the uncontaminated case. As we will see in Theorems~\ref{ThmRobustCLIME} and~\ref{ThmRobustGLasso} below, the sample size requirements for precision matrix estimation are such that the condition $\epsilon \le C \sqrt{\frac{\log p}{n}}$ still allows for a nonvanishing fraction of contamination. Furthermore, note that although the restriction $\epsilon \le 0.02$ may seem somewhat prohibitive, the proof of Theorem~\ref{ThmRobustCovK} reveals that the specific bound on $\epsilon$ is an artifact of the proof technique, and a more careful analysis would allow for a larger degree of contamination, at the expense of slightly looser constants in the covariance estimation bound~\eqref{EqnCovKendall}, as long as $\epsilon$ is bounded by some constant in $[0,1)$.
\end{Remark}

The following theorem is an analog of Theorem~\ref{ThmRobustCovK}, derived for the robust covariance estimator $\hSigma^S$ based on Spearman's correlation coefficient. We assume that the ranks of variables between samples are distinct; note that this happens almost surely when the contaminating distribution has continuous density.

\begin{Theorem}
\label{ThmRobustCovS}
Under the cellwise contamination model~\eqref{cellContam}, suppose the variable ranks are distinct. Also suppose inequality~\eqref{EqnSigmaBd} is satisfied and $\epsilon = \max_{1\leq i\leq p}\epsilon_i \le 0.01$. Let $C > 8\pi$ and \mbox{$C' > \frac{1}{\Phi^{-1}(0.75) \min_{1 \le i \le p} c(\sigma_i) \sqrt{2}}$}, and suppose
\begin{equation*}
\max\left\{\frac{5C}{2} \sqrt{\frac{\log p}{n}} + 51 \pi \epsilon, \; C' \sqrt{\frac{\log p}{n}} + 7.2 M_\sigma \epsilon\right\} \le 1,
\end{equation*}
and the sample size satisfies $\Phi^{-1}(0.75) C' \sqrt{\frac{\log p}{n}} < 1$ and $n\geq \max\left\{15, \; \frac{16\pi^2}{C^2\log p}\right\}$. Then with probability at least 
\begin{equation*}
1 - 2p^{-\left(\frac{C^2}{32\pi^2}-2\right)} - 6p^{-\{2[\Phi^{-1}(0.75)]^2C'^2\min_{1\leq i\leq p}c^2(\sigma_i)-1\}},
\end{equation*}
the robust covariance estimator satisfies
\begin{equation}
\label{EqnCovSpearman}
\left\|\hSigma^S - \bSigma^*\right\|_\infty \le \left(\frac{5C}{2}(M_\sigma^2 + M_\sigma + 1) + C'(2M_\sigma + 1)\right) \sqrt{\frac{\log p}{n}} + \left(175 M_\sigma^2 + 168 M_\sigma + 161\right) \epsilon.
\end{equation}
\end{Theorem}
The proof of Theorem~\ref{ThmRobustCovS} is provided in Section~\ref{SecThmRobustCovS}.

\begin{Remark}
The conclusion of Theorem~\ref{ThmRobustCovS} is very similar to that of Theorem~\ref{ThmRobustCovK}, except for constants and an additional requirement on the size of $n$. However, note that when $\frac{\log p}{n} = o(1)$, implying the statistical consistency of the robust covariance estimator, the requirement $n \ge \max\left\{15, \frac{16\pi^2}{C^2\log p}\right\}$ is essentially extraneous.
\end{Remark}

%\begin{Theorem}
%\label{ThmRobustCovQ}
%Under the cellwise contamination model~\eqref{cellContam}, suppose \dots. Then with probability at least \dots,
%\begin{equation*}
%\max_{1 \le i, j \le p} \left|\hsigma_i \hsigma_j \sin\left(\frac{\pi}{2} r_{ij}^Q\right) - \bSigma^*_{ij}\right| \le \cdots.
%\end{equation*}
%\end{Theorem}

Although the high-dimensional error bounds derived in Theorems~\ref{ThmRobustCovK} and~\ref{ThmRobustCovS} are substantially different from the canonical measures analyzed in the robust statistics literature, our bounds are somewhat related to the notion of the influence function of an estimator. The influence function \citep{Hampel1974}, defined at the population level, measures the infinitesimal change incurred by the statistical functional associated with an estimator when the underlying distribution is contaminated by a point mass. Thus, an estimator has a bounded influence function if the extent of the deviation in its functional representation due to contamination remains bounded, regardless of the location of the point mass. The error bounds~\eqref{EqnCovKendall} and~\eqref{EqnCovSpearman} also reveal that the extent to which the error deviation between the robust covariance estimator and the true covariance grows is bounded by a constant depending only on $M_\sigma$. The two notions do not match precisely; for instance, our theorems allow contamination by an arbitrary distribution rather than simply a point mass, and we are comparing finite-sample deviations of an estimator from $\bSigma^*$ rather than population-level deviations of a statistical functional under a contaminated distribution. However, note that by sending $n \rightarrow \infty$ in the finite-sample bounds and taking the contaminating distribution to be a point mass, we may conclude that the influence function of the robust covariance estimator is bounded. Furthermore, the arguments in our proofs (cf.\ Lemmas~\ref{exptau} and~\ref{exprho} in Appendix~\ref{AppAux}) may be used to derive the fact that the corresponding correlation estimators have a bounded influence function, the precise forms of which appear in \cite{CroDeh10}. The reverse implication, that a correlation estimator with bounded influence (together with a bounded-influence scale estimator) gives rise to high-dimensional deviation bounds of the form in inequalities~\eqref{EqnCovKendall} and~\eqref{EqnCovSpearman}, seems natural but is not immediate.

Finally, note that although Theorems~\ref{ThmRobustCovK} and~\ref{ThmRobustCovS} have been derived under the assumption that the uncontaminated data are drawn from a normal distribution, the same proof techniques may be applied to analyze settings where the uncontaminated data are drawn from a different underlying distribution, as long as the uncontaminated distribution is suitably well-behaved (e.g., has sub-Gaussian tails). Since our ultimate goal is precision matrix estimation, we have focused only on the scenario where the uncontaminated data are drawn from a Gaussian distribution, in which case the structure of the precision matrix is of great interest in the statistical community.

\paragraph{\textbf{Extensions.}} Similar high-dimensional error bounds could be derived for the robust covariance estimator based on the quadrant correlation estimator, which is given by
\begin{equation*}
r_{ij}^Q = \frac{1}{n} \sum_{k=1}^n \sign\left(X_{ki} - \median_{1\leq\ell\leq n} X_{\ell i}\right)\sign\left(X_{kj} - \median_{1\leq\ell\leq n} X_{\ell j}\right),
\end{equation*}
and also known to have bounded influence \citep{SheVil02}. However, we do not provide the full derivations here, since they follow from similar arguments to the ones used in the case of Kendall's and Spearman's correlations.

We also comment briefly on another pairwise covariance estimator appearing in the robust statistics literature. \cite{Tarr2015} and \cite{Croux2014} propose to use the following estimator based on an idea of \cite{GnaKet72}: Noting that
\begin{equation*}
\Cov(X,Y) = \frac{1}{4\alpha \beta} \left[\Var(\alpha X + \beta Y) - \Var(\alpha X - \beta Y)\right],
\end{equation*}
where $\alpha = 1/\sqrt{\Var(X)}$ and $\beta = 1/\sqrt{\Var(Y)}$, the proposal is to replace the variance estimator by a robust variance estimator (e.g., the square of the MAD estimator). However, the drawback of this estimator in comparison to the covariance estimators based on Kendall's tau and Spearman's rho is that the covariance estimator has a maximal breakdown point of 25\% under cellwise contamination, since the argument in the variance involves a sum of variables, and any robust variance estimator has a maximal breakdown point of 50\%. We remark that from the point of view of statistical consistency, a version of the Gnanadesikan-Kettenring covariance estimator may be analyzed in the same manner as the above estimators. Indeed, if we consider the covariance estimator
\begin{equation}
\label{EqnGKCov}
\frac{1}{4} \left(\hsigma_{(i,j), +}^2 - \hsigma_{(i,j),-}^2\right),
\end{equation}
where $\hsigma_{(i,j),+}$ is the (rescaled) MAD statistic computed from $\{X_{ki} + X_{kj}: 1 \le k\le n\}$, and $\hsigma_{(i,j),-}$ is analogously defined to be the MAD statistic computed from $\{X_{ki} - X_{kj}: 1 \le k\le n\}$, our derivations showing the consistency of the MAD estimator (cf. Lemmas~\ref{MADterm2} and~\ref{MADterm1}, with minor modifications) show that
\begin{equation*}
\max_{1\leq i, j\leq p} |\hsigma_{(i,j),+} - \sigma_{(i,j),+}| \le C_1\sqrt{\frac{\log p}{n}} + C_2\epsilon, \quad \text{and} \quad \max_{1\leq i, j\leq p} |\hsigma_{(i,j),-} - \sigma_{(i,j),-}| \le C_1\sqrt{\frac{\log p}{n}} + C_2\epsilon,
\end{equation*}
for data from the cellwise contamination model, where $\sigma_{(i,j),+}$ and $\sigma_{(i,j),-}$ are the population-level standard deviations of the distributions of $X_{ki} + X_{kj}$ and $X_{ki} - X_{kj}$, respectively. Thus,
\begin{equation*}
\max_{1\leq i, j\leq p} |\hsigma_{(i,j),+}^2 - \sigma_{(i,j),+}^2|, \; \max_{1\leq i, j\leq p} |\hsigma_{(i,j),-}^2 - \sigma_{(i,j),-}^2| \le C'\sqrt{\frac{\log p}{n}} + C''\epsilon,
\end{equation*}
as well, from which we may conclude that the covariance estimator~\eqref{EqnGKCov} deviates from the true covariance $\Cov(X_{ki}, X_{kj})$ by the same margin.

Finally, we remark briefly about another popular robust scale estimator known as the $Q_n$ estimator \citep{RousseeuwCroux1993}, defined as follows:
\begin{equation*}
Q_n = c\{|X_{k} - X_{\ell}|: k<\ell\}_{(k^*)},
\end{equation*}
where $c$ is a constant factor, chosen such that $Q_n$ is Fisher-consistent for the population standard deviation, and $k^* = \lceil\binom{n}{2}/4\rceil$. Since the $Q_n$ estimator is also based on quantiles, essentially the same types of arguments used to derive MAD concentration (cf.\ Appendix~\ref{AppMAD}) may be used to establish concentration bounds for the $Q_n$ estimator similar to those appearing in Lemmas~\ref{MADterm2} and~\ref{MADterm1}, up to constant factors.

\subsection{Precision Matrix Estimation}
\label{SecPrecEst}

Using the novel statistical error bounds derived in the previous section, we now provide statistical error bounds on the precision matrix estimators attained by plugging the robust covariance matrix estimates into the CLIME and GLasso. We provide explicit statements in the case of the covariance estimate based on Kendall's tau; analogous statements hold for Spearman's rho, assuming uniqueness of ranks.

We begin with the CLIME estimator. Consider the following uniformity class of matrices:
\begin{equation}
\label{EqnUniformity}
\cU(q, s_0(p), M) = \bigg\{\bOmega: \bOmega\succ 0, \|\bOmega\|_{L_1}\leq M, \max_{1\leq i\leq p}\sum_{j=1}^n|\omega_{ij}|^q \leq s_0(p)\bigg\},
\end{equation}
for $0\leq q<1$, where $\bOmega := (\omega_{ij}) = (\bomega_1, \ldots, \bomega_p)$. The following result provides an elementwise error bound on the estimation error between the CLIME output and the true precision matrix, provided the true precision matrix lies in the class~\eqref{EqnUniformity} defined above:

\begin{Theorem}
\label{ThmRobustCLIME}
Under the cellwise contamination model~\eqref{cellContam}, suppose inequality~\eqref{EqnSigmaBd} is satisfied, and $\epsilon = \max_{1 \le i \le p} \epsilon_i \le 0.02$. Let $C > \pi\sqrt{2}$ and $C' > \frac{1}{\Phi^{-1}(0.75) \min_{1 \le i \le p} c(\sigma_i) \sqrt{2}}$, and suppose inequality~\eqref{EqnContamScaling} also holds and $\Phi^{-1}(0.75) C' \sqrt{\frac{\log p}{n}} < 1$. If the regularization parameter satisfies
\begin{equation}
\label{EqnLambdaLower}
\lambda \ge M\left(C (M_\sigma^2 + M_\sigma + 1) + C' (2M_\sigma + 1)\right) \sqrt{\frac{\log p}{n}} + M \left(97 M_\sigma^2 + 89M_\sigma + 82\right) \epsilon,
\end{equation}
then with probability at least
\begin{equation*}
1 - 2p^{-\left(\frac{C^2}{\pi^2} - 2\right)} - 6 p^{-\left\{2[\Phi^{-1}(0.75)]^2 C'^2 \min_{1 \le i \le p} c^2(\sigma_i) - 1\right\}},
\end{equation*}
the CLIME estimator~\eqref{CLIME} satisfies
\begin{equation*}
\|\hOmega - \bOmega^*\|_\infty \le 4 \|\bOmega^*\|_{L_1} \lambda.
\end{equation*}
\end{Theorem}
The proof of Theorem~\ref{ThmRobustCLIME} is contained in Section~\ref{SecThmRobustCLIME}.

\begin{Remark}
Clearly, the optimal choice of $\lambda$ to minimize the estimation error bound in Theorem~\ref{ThmRobustCLIME} is $\lambda = C_1 \sqrt{\frac{\log p}{n}} + C_2 \epsilon$, where $C_1$ and $C_2$ are the constant prefactors appearing on the right-hand side of inequality~\eqref{EqnLambdaLower}. In this case, the estimation error bound takes the form
\begin{equation*}
\|\hOmega - \bOmega^*\|_\infty \le 4\|\bOmega^*\|_{L_1} \left(C_1 \sqrt{\frac{\log p}{n}} + C_2 \epsilon\right) \le 4M \left(C_1 \sqrt{\frac{\log p}{n}} + C_2 \epsilon\right).
\end{equation*}
\end{Remark}

Turning to the GLasso, we focus on the class of precision matrices satisfying the following incoherence assumption:
\begin{Assumption}
\label{AssumpIncoh}
There exists some $0 < \alpha \le 1$ such that
\begin{equation}
\label{EqnIncoh}
\max_{e \in S^c} \|\bGamma^*_{eS} (\bGamma^*_{SS})^{-1}\|_{L_1} \le 1 - \alpha,
\end{equation}
where $\bGamma^* := \bSigma^* \otimes \bSigma^*$ and $S = \supp(\bOmega^*)$ is the true edge set.
\end{Assumption}
We then have the following result, which is stated in terms of the population-level quantities
\begin{equation*}
\kappa_{\Sigma^*} = \|\bSigma^*\|_{L_1}, \qquad \text{and} \qquad \kappa_{\Gamma^*} = \|(\bGamma^*_{SS})^{-1}\|_{L_1},
\end{equation*}
as well as $k$, the maximum number of nonzero elements in each row of $\bOmega^*$. The theorem also involves constants $C_0, C_1$, and $C_2$, which are independent of $\epsilon$ and the problem instances $n$, $p$, and $k$.

\begin{Theorem}
\label{ThmRobustGLasso}
Under the cellwise contamination model~\eqref{cellContam}, suppose inequality~\eqref{EqnSigmaBd} is satisfied, and $\epsilon = \max_{1 \le i \le p} \epsilon_i \le 0.02$. Also suppose the sample size satisfies the scaling
\begin{equation}
\label{EqnGLassoScaling}
n \ge C_2 \tau \log p \cdot \left(\frac{1}{6(1+8/\alpha) k \max\{\kappa_{\Sigma^*} \kappa_{\Gamma^*}, \kappa_{\Sigma^*}^3 \kappa_{\Gamma^*}^2\}} - C_0 \epsilon\right)^{-2},
\end{equation}
and suppose Assumption~\ref{AssumpIncoh} holds. Suppose $\lambda = \frac{8}{\alpha}\left(C_0\epsilon + C_1 \sqrt{\frac{\tau \log p}{n}}\right)$. Then with probability at least $1 - p^{2-\tau}$, the GLasso estimator~\eqref{GLasso} satisfies $\supp(\hOmega) \subseteq \supp(\bOmega^*)$, and
\begin{equation*}
\|\hOmega - \bOmega^*\|_\infty \le 2 \|(\bGamma^*_{SS})^{-1}\|_{L_1} \left(1 + \frac{8}{\alpha}\right) \left(C_0 \epsilon + C_1 \sqrt{\frac{\tau \log p}{n}}\right).
\end{equation*}
\end{Theorem}
The proof of Theorem~\ref{ThmRobustGLasso} is contained in Section~\ref{SecThmRobustGLasso}. Note that Theorem~\ref{ThmRobustGLasso} implicitly assumes that $\epsilon \le \frac{C}{k}$, so that the expression in parentheses on the right-hand side of inequality~\eqref{EqnGLassoScaling} is positive.

\begin{Remark}
Comparing the results of Theorems~\ref{ThmRobustCLIME} and~\ref{ThmRobustGLasso}, we see that as in the traditional uncontaminated setting, the GLasso delivers slightly stronger guarantees, at the expense of more stringent assumptions. In particular, the GLasso requires the sample size to scale as $n \ge Ck^2 \log p$, whereas the CLIME requires the scaling $n \ge C' \|\bOmega^*\|_{L_1}^2 \log p$ in order to achieve consistency. When the parameter $M$ defining the precision matrix class scales more slowly than $k^2$, the CLIME thus requires a weaker scaling. In addition, the GLasso result supposes Assumption~\ref{AssumpIncoh}, which posits an incoherence bound on submatrices of $\bGamma^*$. On the other hand, Theorem~\ref{ThmRobustGLasso} establishes that the $\supp(\hOmega) \subseteq \supp(\bOmega^*)$ for the GLasso estimator, whereas Theorem~\ref{ThmRobustCLIME} only guarantees consistency for the CLIME estimator in terms of $\ell_\infty$-norm, so the estimated support might contain extraneous terms. In the case of the CLIME estimator, however, the true support of $\bOmega^*$ may be obtained via thresholding, assuming the nonzero elements of $\bOmega^*$ are of the order $\Omega\left(\sqrt{\frac{\log p}{n}}\right)$.
\end{Remark}

Focusing on the level of contamination $\epsilon$ in relation to the problem dimensions, note that Theorems~\ref{ThmRobustCLIME} and~\ref{ThmRobustGLasso} both imply an $O\left(\sqrt{\frac{\log p}{n}}\right) + O(\epsilon)$ error bound on the precision matrix estimator, under the corresponding assumptions. Hence, when $\epsilon \le C \sqrt{\frac{\log p}{n}}$, the estimation error matches the error of the optimal precision matrix estimator in the uncontaminated case, up to a constant factor \citep{RenEtal15}. Further note that when $\epsilon \le C \sqrt{\frac{\log p}{n}}$, the condition $\epsilon = O\left(\frac{1}{k}\right)$ required by the condition~\eqref{EqnGLassoScaling} in Theorem~\ref{ThmRobustGLasso} clearly holds when the sample size satisfies $n \ge Ck^2 \log p$. Note that although the level of contamination tolerated by the estimator decreases as the level of sparsity increases, it is \emph{not} required to decrease as $n$ and $p$ increase, as long as the ratio $\sqrt{\frac{\log p}{n}}$ remains fixed. Thus, the conclusions of Theorems~\ref{ThmRobustCLIME} and~\ref{ThmRobustGLasso} are truly high-dimensional. As in the case of the robust covariance matrix estimators, another nice feature is that when the data are uncontaminated ($\epsilon = 0$), the rate of convergence of the robust precision matrix estimator to the true precision matrix agrees with the optimal rate.

Lastly, note that since the inverse of the correlation matrix has the same support as the precision matrix, we could also estimate $\supp(\bOmega^*)$ using the Kendall's or Spearman's correlation matrices $\hrho^K, \hrho^S$, defined by
\begin{equation*}
\hrho^K_{ij} = \sin\left(\frac{\pi}{2} \boldr^K_{ij}\right), \qquad \text{and} \qquad \hrho^S_{ij} = 2 \sin\left(\frac{\pi}{6} \boldr^S_{ij}\right),
\end{equation*}
respectively, as inputs to the CLIME~\eqref{CLIME} or GLasso~\eqref{GLasso}. Then the same derivations as in Theorems~\ref{ThmRobustCLIME} and~\ref{ThmRobustGLasso}, omitting the concentration bounds on the MAD estimates of scale, would show convergence of $\hrho^K$ and $\hrho^S$ to the population correlation matrix $\brho^*$ in $\ell_\infty$-norm. Note, however, that the conditions for support recovery would then need to hold for the correlation matrix $\brho^*$, rather than for the precision matrix $\bOmega^*$. In particular, a minimum signal strength requirement on $\brho^*$ is stronger than the same requirement imposed on $\bOmega^*$, since the latter can scale inversely with the standard deviations of individual variables in the joint distribution. Therefore, we have chosen to focus our attention in this paper on the output of the CLIME and GLasso when applied to an estimate of the covariance matrix rather than the correlation matrix.

%%%%%%%%%%%%%%%%%%%%%%%%%%%%%%%%%%%%%%%%%%%%%%%%%%%%%%%%%%%%%%%%%%%%%%%%%%%%%%
\subsection{Consequences for Robust Estimation}
\label{SecConsequences}

We now interpret the conclusions of our theorems in some concrete settings, where the data matrix is contaminated according to several different mechanisms.

\paragraph{\textbf{Constant fraction of outliers:}} We first briefly discuss the most basic setting of cellwise contamination, to emphasize the generality of our results. Following the model~\eqref{cellContam}, suppose each entry of the data matrix $\X$ is contaminated independently with probability $\epsilon$. Furthermore, either all contaminated entries may be drawn independently from a fixed contaminating distribution, or the contaminated entries in each row may be drawn jointly from a fixed contaminating distribution. In each case, Theorems~\ref{ThmRobustCovK} and~\ref{ThmRobustCovS} provide elementwise error bounds on the robust covariance estimators, and Theorems~\ref{ThmRobustCLIME} and~\ref{ThmRobustGLasso} provide elementwise error bounds on the robust precision matrix estimators constructed from the CLIME and GLasso. The strength of the theorems lies in the fact that we do not make any side assumptions about the outlier distribution; in particular, it may be heavy-tailed and/or contain point masses. Hence, whereas statistics such as the sample covariance and sample correlation will have slower rates of convergence due to a constant fraction of outliers drawn from an ill-behaved distribution, their robust counterparts are agnostic to the outlier distribution.

It is also important to note that the statistical error bounds given in the theorems of Sections~\ref{SecCovEst} and~\ref{SecPrecEst} continue to hold when $\epsilon > C \sqrt{\frac{\log p}{n}}$. The difference is that in such scenarios, the statistical error will be of the order $O(\epsilon)$ rather than $O\left(\sqrt{\frac{\log p}{n}}\right)$. However, the effect of an $\epsilon$ fraction of outliers nonetheless grows only linearly as a function of $\epsilon$. This emphasizes the robustness properties of the covariance and precision matrix estimators studied in our paper.

\paragraph{\textbf{Missing data:}} Turning to a somewhat different setting, note that missing data may also be seen as an instance of cellwise contamination. In this model, data are missing completely at random (MCAR), meaning that the probability of missingness is independent of the location of the unobserved entry of the data matrix \citep{LitRub86}. In other words, if we observe the matrix $\X^{\text{mis}}$ with missing entries, where the probability that an entry in column $i$ is missing is equal to $\epsilon_i$, we have
\begin{equation}
\label{EqnMCAR}
X_{ki}^{\text{mis}} =
\begin{cases}
Y_{ki}, & \text{ with probability } 1-\epsilon_i, \\
\text{missing}, & \text{ with probability } \epsilon_i,
\end{cases}
\end{equation}
where $\Y$ is the fully-observed matrix. Note that if we zero-fill the missing entries of $\X^{\text{mis}}$, the resulting matrix $\X$ exactly follows the cellwise contamination model~\eqref{cellContam}, with $\Z_k = \0$ for all $k$. The following result is an immediate consequence of our theorems:
\begin{Corollary}
\label{CorMCAR}
Suppose data are drawn from the missing data model~\eqref{EqnMCAR}, and the matrix $\X$ is the zero-filled data matrix. Let $\epsilon = \max_{1 \le i \le p} \epsilon_i$. Under the same conditions as in Theorem~\ref{ThmRobustCLIME}, we have
\begin{equation*}
\|\hOmega - \bOmega^*\|_\infty \le 4 \|\bOmega^*\|_{L_1} \lambda,
\end{equation*}
for the robust CLIME estimator constructed from $\X$. Under the same conditions as in Theorem~\ref{ThmRobustGLasso}, we have $\supp(\hOmega) \subseteq \supp(\bOmega^*)$ and
\begin{equation*}
\|\hOmega - \bOmega^*\|_\infty \le 2 \|(\bGamma^*_{SS})^{-1}\|_{L_1} \left(1 + \frac{8}{\alpha}\right) \left(C_0 \epsilon + C_1 \sqrt{\frac{\tau \log p}{n}}\right),
\end{equation*}
for the robust GLasso estimator constructed from $\X$.
\end{Corollary}

Note that the conclusion of Corollary~\ref{CorMCAR} does not actually require the matrix $\X$ to be zero-filled for missing values; in fact, we could fill the missing entries with samples generated according to any distribution (as long as the distribution remains the same across rows). This is because the missing entries are essentially taken as outliers. Of course, our bounds should only be interpreted up to constant factors, and filling missing entries in a strategic way, e.g., filling entries in column $i$ with the mean $E(X_{ki})$, could lead to smaller estimation error in practice.

\paragraph{\textbf{Rowwise contamination:}} Although we have thus far assumed that data are contaminated according to a cellwise mechanism, we now show that the same results apply for rowwise contamination, as well. Recall that each row in the data matrix for the rowwise contamination model with contamination level $\epsilon$ is given by
\begin{equation}
\label{rowContam}
\X_k = (1-B_k) \Y_k + B_k \Z_k, \qquad \forall 1 \le k \le n,
\end{equation}
where $\Y_k$ is the uncontaminated row vector, $\Z_k$ is the contamination vector, and $B_k \sim\text{Bernoulli}(\epsilon)$.

Although model~\eqref{rowContam} differs from model~\eqref{cellContam}, a simple inspection of the proofs of Theorems~\ref{ThmRobustCLIME} and~\ref{ThmRobustGLasso} shows that only Lemma~\ref{kendalltauconsistency} needs to be modified. Furthermore, the equation~\eqref{EqnPairMarginal}, giving the distribution of pairwise entries in a row, simply needs to be replaced by the equation
\begin{equation}
\label{EqnRowPair}
(X_{ki}, X_{kj}) \stackrel{\text{i.i.d.}}\sim F_{ij} = (1-\epsilon) \Phi_{\bmu_{\{i, j\}}, \bSigma_{\{i, j\}}} + \epsilon H_{ij}, \qquad \forall 1 \le k \le n,
\end{equation}
in the proof of Lemma~\ref{kendalltauconsistency}. Equation~\eqref{EqnRowPair} comes from the fact that the pair is either drawn jointly from a normal distribution with probability $1-\epsilon$, or from the contaminating distribution with probability $\epsilon$. Then the remainder of the argument follows as before, implying that the same conclusion of Lemma~\ref{kendalltauconsistency} applies. (We could obtain a smaller prefactor for $\epsilon$ in the bound~\eqref{tauconsistency}, since $2\epsilon$ is replaced by $\epsilon$, but we are not concerned about optimizing constants here.)

We therefore arrive at the following result:
\begin{Corollary}
\label{CorRowContam}
Under the rowwise contamination model~\eqref{rowContam}, the same conclusions as in Corollary~\ref{CorMCAR} hold for the CLIME and GLasso estimators constructed from $\X$.
\end{Corollary}
We emphasize that the rowwise contamination model~\eqref{rowContam} is \emph{not} in general a special case of the cellwise contamination model~\eqref{cellContam}; rather, the proof techniques for analyzing the cellwise model may be used to handle the rowwise model, as well.

%{\red Discuss alternative methods for rowwise contamination. Note that methods specifically designed for rowwise contamination might perform better than our algorithm.

%\begin{itemize}
%\item Comparison with other robust estimators of covariance and/or inverse covariance matrix
%\item Situations where Chen, Gao \& Ren (2015), Tyler/Maronna $M$-estimators fail, but ours succeeds?
%\item Drton's inverse covariance estimator?
%\end{itemize}
%}

%%%%%%%

\section{Breakdown Point}
\label{SecBreakdown}

We now turn to a brief discussion of the breakdown point of the estimators studied in this paper. As is discussed in \cite{DonohoHuber1983} and \cite{HamEtal11}, breakdown analysis concerns the \emph{global} behavior of a procedure, under large departures from an assumed situation.  On the other hand, the theoretical analysis of statistical consistency and efficiency are related to notions of infinitesimal robustness, and quantifies the \emph{local} behavior of a procedure at or near the assumed situation. The analogy is made in \cite{DonohoHuber1983} between the fields of material science and statistics, where the notions of stiffness (resistance of a material to displacements caused by a small load) and breaking strength (the amount of load required to make the material fracture) parallel those of the influence function and the breakdown point. Ideally, a procedure should perform well both locally and globally; optimizing either measure alone is unwise. Our key result of this section shows that although the GLasso and CLIME estimators both enjoy roughly the same statistical rate of estimation, the CLIME does \emph{not} perform as well as the GLasso when the breakdown point is used to quantify the degree of robustness.

Our analysis of the GLasso estimator closely follows that of \cite{Croux2014}; however, since the specific precision matrix estimators analyzed in our paper differ slightly from those of \cite{Croux2014}, we include the full argument for the sake of completeness. We define the finite-sample breakdown point of the precision matrix estimator under cellwise contamination to be
\begin{equation}
\label{EqnPrecisionBreak}
\epsilon_n(\hOmega, \X) := \min_{1 \le m \le n} \left\{\frac{m}{n}: \sup_{\X^m} D(\hOmega(\X), \hOmega(\X^m)) = \infty\right\},
\end{equation}
where
\begin{equation*}
D(\bA,\B) := \max\left\{|\lambda_1(\bA) - \lambda_1(\B)|, |\lambda_p^{-1}(\bA) - \lambda_p^{-1}(\B)|\right\},
\end{equation*}
and $\X^m$ is a data matrix obtained from $\X$ by replacing at most $m$ entries in each column by arbitrary elements. We also define the explosion finite sample breakdown point of a covariance matrix estimator as follows:
\begin{equation}
\label{EqnExplode}
\epsilon_n^+(\bS, \X) := \min_{1 \le m \le n} \left\{\frac{m}{n}: \sup_{\X^m} |\lambda_1(\bS(\X)) - \lambda_1(\bS(\X^m))| = \infty\right\}
\end{equation}
\citep{MaronnaZamar2002}. Note that the explosion breakdown point only accounts for maximum eigenvalues, whereas the overall covariance matrix estimator breaks down under explosion or \emph{implosion} (i.e., arbitrarily small minimum eigenvalues).  Also, the breakdown point under cellwise contamination is less than or equal to the breakdown point under rowwise contamination, since the supremum in the latter case is only taken over $\X^m$ with at most $m$ rows replaced. 
%according to this definition,
%\begin{equation*}
%\epsilon_n(\hOmega, \X) \le \epsilon_n^+(\hOmega, \X),
%\end{equation*}

We will consider the breakdown behavior of a slightly tweaked version of the GLasso presented earlier, using a positive semidefinite matrix as the input to the optimization problem. Consider the matrix
\begin{equation}
\label{EqnCovProj}
\cSigma(\X) := \argmin_{\M\succeq 0} \|\hSigma - \M\|_\infty,
\end{equation}
where $\hSigma = \hSigma(\X)$ is the robust covariance matrix estimator constructed from the data matrix $\X$. Let
\begin{equation}
\label{EqnGLasso}
\cOmega(\X) := \argmin_{\bOmega\succ 0} \big\{\tr(\cSigma \bOmega) - \log\det(\bOmega) + \lambda\|\bOmega\|_{1, \text{off}}\big\}
\end{equation}
be the corresponding GLasso estimator. Note that from a computational standpoint, the projection step~\eqref{EqnCovProj} is important so that fast solvers for the GLasso program~\eqref{EqnGLasso} may be applied (e.g., \cite{FriEtal08}). Furthermore, we note that the projection step~\eqref{EqnCovProj} constitutes a convex program, so the additional computational time is negligible compared to the computation required for running the GLasso. We have the following result:
\begin{Theorem}
\label{ThmGLassoBreak}
Consider the positive semidefinite version of the robust GLasso estimator~\eqref{EqnGLasso}. Then under the same conditions as in Theorem~\ref{ThmRobustGLasso}, we have $\supp(\cOmega) \subseteq \supp(\bOmega^*)$ and
\begin{equation}
\label{EqnProjGLassoBound}
\|\cOmega - \bOmega^*\|_\infty \le 2 \|(\bGamma^*_{SS})^{-1}\|_{L_1} \left(1 + \frac{8}{\alpha}\right) \left(C_0' \epsilon + C_1' \sqrt{\frac{\tau \log p}{n}}\right).
\end{equation}
Furthermore, for any data matrix $\X \in \mathbb{R}^{n \times p}$, the breakdown point satisfies $\epsilon_n(\cOmega, \X) = 50\%$.
\end{Theorem}

The proof of Theorem~\ref{ThmGLassoBreak} is provided in Section~\ref{SecThmGLassoBreak}.

\begin{Remark}
Note that Theorem~\ref{ThmGLassoBreak} guarantees that the robust GLasso estimator $\cOmega$ obtained from a semidefinite projection of the robust covariance estimator shares the same level of statistical consistency achieved by the robust GLasso estimator $\hOmega$. In addition, the precision matrix estimator $\cOmega$ has a breakdown point of 50\%. Although other authors \citep{Croux2014, Tarr2015} also suggest projecting the robust covariance estimator onto the positive semidefinite cone before applying the GLasso, they advocate a projection in terms of the Frobenius norm rather than the $\ell_\infty$-norm in the optimization program~\eqref{EqnCovProj}. As can be seen in the proof of Theorem~\ref{ThmGLassoBreak}, minimizing the elementwise $\ell_\infty$-norm is much more natural from the point of view of statistical consistency, since it guarantees that the $\ell_\infty$-error between the precision matrix estimate and the true precision matrix grows by at most a factor of two.
\end{Remark}

Turning to the CLIME estimator, we now show that although the CLIME is as robust as the GLasso in terms of statistical consistency under the cellwise contamination model, it has much poorer breakdown behavior. Consider the CLIME estimator based on corrupted data:
\begin{equation}
\label{EqnRobCLIME}
\begin{aligned}
& \text{min} & & \|\bOmega\|_1 \\
& \text{s.t. } & & \|\hSigma(\X^m) \bOmega - I\|_\infty \le \lambda,
\end{aligned}
\end{equation}
where $\hSigma(\X^m)$ is the robust covariance estimator based on a data matrix with at most $m$ arbitrarily corrupted entries per column. Since the CLIME estimator arises as the solution to a constrained linear program, the solution is undefined (infinite) when the problem is infeasible. Indeed, we will show in the following theorem that such a case may arise even by corrupting at most \emph{one} entry in each column of the data matrix.
\begin{Theorem}
\label{ThmCLIMEBreakdown}
In the case when $p = 2$, there exists $\X \in \mathbb{R}^{n \times 2}$ such that $\epsilon_n(\hOmega, \X) = \frac{1}{n}$, where $\hOmega$ denotes the CLIME estimator.
\end{Theorem}

The proof of Theorem~\ref{ThmCLIMEBreakdown}, supplied in Section~\ref{SecThmCLIMEBreakdown}, provides the construction of a data matrix $\X \in \R^{n \times 2}$ where the CLIME estimator becomes infeasible after perturbing a single entry in each column. This is in stark contrast to the result in Theorem~\ref{ThmGLassoBreak}, which establishes that the breakdown point of the robust GLasso estimator is 50\%, for \emph{any} realization of the data matrix $\X$.

\begin{Remark}
Although Theorem~\ref{ThmCLIMEBreakdown} is stated for the case $p=2$, the argument used to prove the theorem is readily generalizable to higher dimensions, as well, in which case we would also have a matrix $\X \in \mathbb{R}^{n \times p}$ satisfying $\epsilon_n(\bOmega, \X) = \frac{1}{n}$. For instance, we could construct an $n \times p$ matrix $\X^1$ such that $\bSigma(\X^1)$ is a block matrix with upper-left block equal to the matrix constructed in the proof of Theorem~\ref{ThmCLIMEBreakdown}, lower-left block equal to the identity, and off-diagonal blocks equal to zero.
\end{Remark}
%{\red Do we need $n>p$?} 

The conclusion of Theorem~\ref{ThmCLIMEBreakdown} underscores the fact that consistency and breakdown point under cellwise contamination are in some sense orthogonal measures of robustness. As we demonstrated in the previous section, both the CLIME and GLasso lead to estimators that enjoy good rates of statistical consistency when the contamination fraction $\epsilon$ is sufficiently small relative to the problem parameters. On the other hand, the results of this section show that the CLIME is extremely non-robust in terms of its breakdown point. Similarly, procedures such as the Gnanadesikan-Kettenring estimator~\eqref{EqnGKCov} may be shown to be statistically consistent under cellwise contamination, but as discussed in \cite{Croux2014}, the breakdown point of the covariance estimator $\hSigma$ is at most 25\%, which leads to error propagation in $\hOmega$.
%if the covariance estimator is used as an input for the GLasso. {\red If $\hSigma\succeq 0$, then the breakdown point of $\hOmega$ is at least as large as the breakdown point of $\hSigma$?}

Finally, we note that the notion of breakdown point that we consider in equation~\eqref{EqnPrecisionBreak} is defined with respect to a finite sample, without recourse to probability distributions. Other notions of breakdown point, defined with respect to an $\epsilon$-contaminated distribution, have also been studied in the literature \citep{HamEtal11}. For some alternative measures of breakdown robustness, the CLIME estimator may have a more controlled breakdown behavior, but we have not explored them here.

\section{Proofs}
\label{SecProofs}

In this section, we provide an outline of the proofs of the main theorems in the paper. Proofs of the more technical supporting lemmas are contained in the supplementary Appendix.

\subsection{Proof of Theorem~\ref{ThmRobustCovK}}
\label{SecThmRobustCovK}

The proof is based on Lemma~\ref{kendalltauconsistency}, which gives an error bound for the pairwise terms $\sin(\frac{\pi}{2}\boldr^K_{ij})$, and Lemma~\ref{MADconsistency}, which gives an error bound for the scale estimates $\hsigma_i$. Note that we require the bound $\epsilon \le 0.02$ on the level of contamination in Lemma~\ref{kendalltauconsistency}, but the requirement could be relaxed with a more refined proof technique. The proofs of Lemmas~\ref{kendalltauconsistency} and~\ref{MADconsistency} are provided in Appendices~\ref{AppLemKendallTau} and~\ref{AppLemMADConsist}.

\begin{Lemma}\label{kendalltauconsistency}
Under model \eqref{cellContam}, let $\epsilon = \max_{1\leq i\leq p}\epsilon_i \le 0.02$. For any constant $C > \pi \sqrt{2}$, we have
\begin{equation}
\max_{1\leq i, j\leq p}\bigg|\sin\Big(\frac{\pi}{2}\boldr^K_{ij}\Big) - \brho_{ij}\bigg|\leq C\sqrt{\frac{\log p}{n}} + 26\pi \epsilon, \label{tauconsistency}
\end{equation}
with probability at least $1 - 2p^{-(C^2/\pi^2-2)}$.
%where $C$ is a constant independent of $n, p$, and $\epsilon$.
%{\\\blue 
%Under model \eqref{cellContam}, let $\epsilon = \max_{1\leq i\leq p}\epsilon_i<1$.  Suppose that $6\epsilon+17\epsilon^2\leq\sqrt{(\log p)/n}$. Then with probability at least $1 - 2p^{-2}$, 
%\begin{equation}
%\max_{1\leq i, j\leq p}\bigg|\sin\Big(\frac{\pi}{2}\htauij\Big) - \rho_{ij}^K\bigg|\leq 4\pi\sqrt{\frac{\log p}{n}}.
%\end{equation}
%}
\end{Lemma}

\begin{Lemma}\label{MADconsistency}
Under model \eqref{cellContam}, suppose $0<\min_{1\leq i\leq p}\sigma_i \leq \max_{1\leq i\leq p}\sigma_i \leq M_\sigma$, and the maximum contamination error satisfies $\epsilon = \max_{1\leq i\leq p}\epsilon_i \leq\frac{1}{16}$. Let  $c(\sigma_i) = \frac{15}{64\sqrt{2\pi}\sigma_i}\exp\left(-\frac{(1.1\sigma_i+0.5)^2}{2\sigma_i^2}\right)$, and suppose $C' > \frac{1}{\Phi^{-1}(0.75) \min_{1 \le i \le p} c(\sigma_i) \sqrt{2}}$. Also suppose $\Phi^{-1}(0.75)C'\sqrt{\frac{\log p}{n}}<1$. Then with probability at least $1 - 6p^{-\{2[\Phi^{-1}(0.75)]^2 C'^2 \min_{1 \le i \le p} c^2(\sigma_i) - 1\}}$, we have
\begin{equation*}
\max_{1\leq i\leq p}|\hsigma_i - \sigma_i|\leq C'\sqrt{\frac{\log p}{n}} + 7.2M_\sigma\epsilon.
\end{equation*}
\end{Lemma}

Using the fact that
\[ABC - abc = (A-a)(B-b)(C-c) + aC(B-b) + Ab(C-c) + Bc(A-a),\]
we can decompose $|\hsigma_i\hsigma_j\sin(\frac{\pi}{2}\boldr^K_{ij}) - \bSigma_{ij}^*| = |\hsigma_i\hsigma_j\sin(\frac{\pi}{2}\boldr^K_{ij}) - \sigma_i\sigma_j\brho_{ij}|$ by the triangle inequality, as follows:
\begin{align*}
\bigg|\hsigma_i\hsigma_j\sin\Big(\frac{\pi}{2}\boldr^K_{ij}\Big) - \sigma_i\sigma_j\brho_{ij}\bigg| & \le |\hsigma_i-\sigma_i| |\hsigma_j-\sigma_j| \bigg|\sin\Big(\frac{\pi}{2}\boldr^K_{ij}\Big)-\brho_{ij}\bigg| + \bigg|\sigma_i\sin\Big(\frac{\pi}{2}\htauij\Big)\bigg| |\hsigma_j-\sigma_j| \\
&\qquad+ |\hsigma_i\sigma_j| \bigg|\sin\Big(\frac{\pi}{2}\htauij\Big)-\brho_{ij}\bigg| + |\hsigma_j\rho_{ij}| |\hsigma_i-\sigma_i| \\
& \stackrel{(i)}{\leq} |\hsigma_i-\sigma_i| |\hsigma_j-\sigma_j| \bigg|\sin\Big(\frac{\pi}{2}\htauij\Big)-\brho_{ij}\bigg| + \sigma_i |\hsigma_j-\sigma_j| \\
&\qquad+ |\hsigma_i\sigma_j| \bigg|\sin\Big(\frac{\pi}{2}\htauij\Big)-\brho_{ij}\bigg| + \hsigma_j |\hsigma_i-\sigma_i| \\
&\leq |\hsigma_i-\sigma_i| |\hsigma_j-\sigma_j| \bigg|\sin\Big(\frac{\pi}{2}\htauij\Big)-\brho_{ij}\bigg| + \sigma_i |\hsigma_j-\sigma_j| \\
&\qquad+ (|\hsigma_i-\sigma_i|+\sigma_i)\sigma_j \bigg|\sin\Big(\frac{\pi}{2}\htauij\Big)-\brho_{ij}\bigg| + (|\hsigma_j-\sigma_j| + \sigma_j) |\hsigma_i-\sigma_i|,
\end{align*}
where $(i)$ uses the facts that $|\sin(x)| \le 1$ for all $x$, and $|\brho_{ij}| \le 1$, since it is a correlation coefficient. Using Lemmas~\ref{kendalltauconsistency} and~\ref{MADconsistency} and the assumption~\eqref{EqnContamScaling}, we obtain the overall upper bound
\begin{align*}
& \left(C \sqrt{\frac{\log p}{n}} + 26\pi \epsilon\right) \left(C' \sqrt{\frac{\log p}{n}} + 7.2 M_\sigma \epsilon\right)^2 + M_\sigma \left(C' \sqrt{\frac{\log p}{n}} + 7.2 M_\sigma \epsilon\right) \\
& \qquad + \left(M_\sigma + C' \sqrt{\frac{\log p}{n}} + 7.2 M_\sigma \epsilon \right)\left\{ \left(C \sqrt{\frac{\log p}{n}} + 26\pi \epsilon\right)M_\sigma + \left(C' \sqrt{\frac{\log p}{n}} + 7.2 M_\sigma \epsilon\right)\right\} \\
&\le \left(M_\sigma(M_\sigma + 1) + 1\right)\left(C \sqrt{\frac{\log p}{n}} + 26\pi\epsilon\right) + \left(2M_\sigma + 1\right)\left(C' \sqrt{\frac{\log p}{n}} + 7.2M_\sigma \epsilon\right),
\end{align*}
implying inequality~\eqref{EqnCovKendall}.

\subsection{Proof of Theorem~\ref{ThmRobustCovS}}
\label{SecThmRobustCovS}

The proof is based on Lemma~\ref{spearmanrhoconsistency}, which gives an error bound for $2\sin(\frac{\pi}{6}\hrhosij)$, and Lemma~\ref{MADconsistency}, which gives an error bound for $\hsigma_i$. Note that we require the bound $\epsilon \le 0.01$ on the level of contamination in Lemma~\ref{spearmanrhoconsistency}, but the requirement could again be relaxed with a more refined proof technique. The proof of Lemma~\ref{spearmanrhoconsistency} is contained in Appendix~\ref{AppLemSpearman}.

\begin{Lemma}\label{spearmanrhoconsistency}
Under model \eqref{cellContam}, let $\epsilon = \max_{1\leq i\leq p}\epsilon_i \le 0.01$. Suppose $C > 8\pi$ and the sample size satisfies $n\geq \max\left\{15, \; \frac{16\pi^2}{C^2\log p}\right\}$. Then
\begin{equation}
\max_{1\leq i, j\leq p}\bigg|2\sin\Big(\frac{\pi}{6}\hrhosij\Big) - \brho_{ij}^S\bigg|\leq \frac{5C}{2}\sqrt{\frac{\log p}{n}} + 51\pi \epsilon, \label{rhoconsistency}
\end{equation}
with probability at least $1 - 2p^{-\left\{\frac{C^2}{32\pi^2}-2\right\}}$.
%{\\\blue Under model \eqref{cellContam}, let $\epsilon = \max_{1\leq i\leq p}\epsilon_i < 1$.  Suppose that $16\epsilon + 86\epsilon^2 + 118\epsilon^3\leq 5\sqrt{(\log p)/n}$,  then for $n\geq 1/(6\log p)$ and $n$ divisible by 3, with probability at least $1 - 2p^{-2}$, 
%\begin{equation}
%\max_{1\leq i, j\leq p}\bigg|2\sin\Big(\frac{\pi}{6}\hrhosij\Big) - \rho_{ij}^S\bigg|\leq 30\pi\sqrt{\frac{\log p}{n}}.
%\end{equation}
%}
\end{Lemma}
Using a similar decomposition as in the proof of Theorem~\ref{ThmRobustCovK}, we have
\begin{align*}
&\bigg|2\hsigma_i\hsigma_j\sin\Big(\frac{\pi}{6}\hrhosij\Big) - \sigma_i\sigma_j\brho_{ij}\bigg| \\
&\leq |\hsigma_i-\sigma_i| |\hsigma_j-\sigma_j| \bigg|2\sin\Big(\frac{\pi}{6}\hrhosij\Big)-\brho_{ij}\bigg| + \sigma_i |\hsigma_j-\sigma_j| \\
&\qquad+ (|\hsigma_i-\sigma_i|+\sigma_i)\sigma_j \bigg|2\sin\Big(\frac{\pi}{6}\hrhosij\Big)-\brho_{ij}\bigg| + (|\hsigma_j-\sigma_j| + \sigma_j) |\hsigma_i-\sigma_i|.
\end{align*}
Using Lemmas~\ref{MADconsistency} and \ref{spearmanrhoconsistency}, we then obtain the overall upper bound
\begin{align*}
& \left(\frac{5C}{2}\sqrt{\frac{\log p}{n}} + 51\pi \epsilon\right) \left(C' \sqrt{\frac{\log p}{n}} + 7.2 M_\sigma \epsilon\right)^2 + M_\sigma \left(C' \sqrt{\frac{\log p}{n}} + 7.2 M_\sigma \epsilon\right) \\
& \qquad + \left(M_\sigma + C' \sqrt{\frac{\log p}{n}} + 7.2 M_\sigma \epsilon \right) \left\{ M_\sigma\left(\frac{5C}{2}\sqrt{\frac{\log p}{n}} + 51\pi \epsilon\right) + \left(C' \sqrt{\frac{\log p}{n}} + 7.2 M_\sigma \epsilon\right)\right\} \\
&\le (M_\sigma(M_\sigma + 1) + 1) \left(\frac{5C}{2} \sqrt{\frac{\log p}{n}} + 51 \pi \epsilon\right) + (2M_\sigma + 1) \left(C' \sqrt{\frac{\log p}{n}} + 7.2 M_\sigma \epsilon\right),
\end{align*}
which is easily simplified to obtain the prescribed bound.

%\subsection{Proof of Theorem~\ref{ThmRobustCovQ}}

%We have the following lemma:

%\begin{Lemma}
%\label{LemQuadConsistency}
%Under model~\eqref{cellContam}, w.h.p.,
%\begin{equation*}
%\max_{1 \le i,j \le p} \left|\sin\left(\frac{\pi}{2} r_{ij}^Q\right) - \rho_{ij}^Q\right| \le C \sqrt{\frac{\log p}{n}}.
%\end{equation*}
%\end{Lemma}

\subsection{Proof of Theorem~\ref{ThmRobustCLIME}}
\label{SecThmRobustCLIME}

Clearly, it suffices to prove the elementwise deviation bound for the unsymmetrized matrix $\hOmega$. We begin with the following general lemma, relating deviation bounds in the covariance matrix estimator $\hSigma$ to the error of the CLIME estimator. A version of the following result appears in \cite{Cai2011}, but we include the relatively short proof for the sake of completeness.
%{\red Do we still want to include the following lemma and its proof, since it is given in Cai 2011?}
\begin{Lemma}
\label{LemCLIMEerror}
Suppose $\bOmega^* \in\cU(q, s_0(p), M)$.  If $\hOmega$ is the output of the CLIME estimator~\eqref{CLIME}, where the regularization parameter satisfies $\lambda \geq M \|\hSigma-\bSigma^*\|_\infty$, then $\|\hOmega-\bOmega^*\|_\infty \leq 4\|\bOmega^*\|_{L_1}\lambda$. 
\end{Lemma}
\begin{proof}
We have
\begin{align}
\|\I - \hSigma\bOmega^*\|_\infty = \|(\hSigma-\bSigma^*)\bOmega^*\|_\infty \leq \|\bOmega^*\|_{L_1}\|\hSigma-\bSigma^*\|_\infty \leq\lambda,
\end{align}
the first inequality is due to $\|\bA\B\|_\infty \leq\|\bA\|_\infty\|\B\|_{L_1}$, and the second inequality follows by assumption.  Then
\[\|\hSigma(\hOmega-\bOmega^*)\|_\infty \leq \|\hSigma\hOmega-\I\|_\infty + \|\I - \hSigma\bOmega^*\|_\infty \leq 2\lambda.\]

For $1\leq i\leq p$, let $\e_i$ be the canonical vector with 1 in the $i^{th}$ coordinate and 0 in all other coordinates, and let $\hbeta_i$ be the solution of the following convex optimization problem:
\[\min_{\bbeta\in\R^p} \|\bbeta\|_1 \qquad\text{subject to}\qquad\|\hSigma\bbeta - \e_i\|_\infty\leq\lambda.\]
Note that $\hOmega = (\hbeta_1, \ldots, \hbeta_p)$ (cf. Lemma 1 in \cite{Cai2011}). It follows that $\|\hbeta_i\|_1 \leq \|\bOmega^*\|_{L_1}$, for $1\leq i\leq p$, so $\|\hOmega\|_{L_1} \leq \|\bOmega^*\|_{L_1}$. Hence,
\begin{align*}
\|\bSigma^*(\hOmega-\bOmega^*)\|_\infty &\leq \|\hSigma(\hOmega-\bOmega^*)\|_\infty + \|(\hSigma-\bSigma^*)(\hOmega-\bOmega^*)\|_\infty \\
&\leq 2\lambda + \|\hOmega-\bOmega^*\|_{L_1}\|\hSigma-\bSigma^*\|_\infty \\
&\leq 2\lambda + \|\hOmega\|_{L_1}\|\hSigma-\bSigma^*\|_\infty + \|\bOmega^*\|_{L_1}\|\hSigma-\bSigma^*\|_\infty\\
&\leq 4\lambda.
\end{align*}
Finally,
\[\|\hOmega-\bOmega^*\|_\infty = \|\bOmega^*\bSigma^*(\hOmega-\bOmega^*)\|_\infty \leq \|\bOmega^*\|_{L_1}\|\bSigma^*(\hOmega-\bOmega^*)\|_\infty \leq 4\|\bOmega^*\|_{L_1}\lambda.\]
\end{proof}

Combining Lemma~\ref{LemCLIMEerror} with the result of Theorem~\ref{ThmRobustCovK}, we obtain the desired result.

%%%%%

\subsection{Proof of Theorem~\ref{ThmRobustGLasso}}
\label{SecThmRobustGLasso}

Our proof is based on the following result:
\begin{Lemma} [Theorem 1 in \cite{RavEtal11}]
\label{LemGLassoBd}
Suppose $\bOmega^*$ satisfies the incoherence condition~\eqref{EqnIncoh}, and that for all $1 \le i, j \le p$, the tail condition
\begin{equation}
\label{EqnTail}
P\left(|\hSigma_{ij} - \bSigma^*_{ij}| \ge \delta\right) \le \frac{1}{f(n,\delta)}, \qquad \forall \delta > 0,
\end{equation}
holds, for some function $f$ that is monotonically increasing in $n$. Also suppose the sample size satisfies
\begin{equation*}
n > \bar{n}_f\left(\frac{1}{6(1+8/\alpha) k \max\{\kappa_{\Sigma^*} \kappa_{\Gamma^*}, \kappa^3_{\Sigma^*} \kappa^2_{\Gamma^*}\}}, \; p^{\tau}\right),
\end{equation*} 
where
\begin{equation*}
\bar{n}_f(\delta; r) = \argmax\{n: f(n, \delta) \le r\}, \quad \text{and} \quad \bar{\delta}_f(n; r) := \argmax\{\delta: f(n, \delta) \le r\}.
\end{equation*}
Then with probability at least $1 - p^{2 - \tau}$, for the choice $\lambda = \frac{8}{\alpha} \bar{\delta}_f(n, p^\tau)$, the GLasso estimator satisfies
\begin{equation*}
\|\hOmega - \bOmega^*\|_\infty \le 2 \kappa_{\Gamma^*} \left(1 + \frac{8}{\alpha}\right) \bar{\delta}_f(n, p^\tau),
\end{equation*}
and
\begin{equation*}
\supp(\hOmega) \subseteq \supp(\bOmega^*).
\end{equation*}
\end{Lemma}

Inspecting the proofs of the technical lemmas employed in proving Theorem~\ref{ThmRobustCovK}, we may see that inequality~\eqref{EqnTail} holds with the function $f(n, \delta) = c_1 \exp(c_2 n (\delta - c_0\epsilon)^2)$, defined for $\delta > c_0 \epsilon$, where $c_0, c_1$, and $c_2$ are appropriately chosen constants. An easy calculation shows that
\begin{equation*}
\bar{\delta}_f(n, r) = c_0 \epsilon + \sqrt{\frac{1}{c_2 n} \log\left(\frac{r}{c_1}\right)},
\end{equation*}
so
\begin{equation*}
\bar{\delta}_f(n, p^\tau) = c_0 \epsilon + C_1 \sqrt{\frac{\tau \log p}{n}}.
\end{equation*}
Similarly, we may easily verify that
\begin{equation*}
\quad \bar{n}_f(\delta, p^\tau) = C_2 \frac{\tau \log p}{(\delta - c_0\epsilon)^2}.
\end{equation*}
Lemma~\ref{LemGLassoBd} then implies that the desired conclusions.

%%%%

\subsection{Proof of Theorem~\ref{ThmGLassoBreak}}
\label{SecThmGLassoBreak}

Note that $\cSigma$ is the projection of the robust covariance estimator $\hSigma$ onto the positive semidefinite cone, where the distance is measured in the elementwise $\ell_\infty$-norm. Furthermore, note that
\begin{equation*}
\|\cSigma - \hSigma\|_\infty \le \|\bSigma^* - \hSigma\|_\infty,
\end{equation*}
since $\bSigma^* \succeq 0$. Hence,
\begin{equation}
\label{EqnSigmaTriangle}
\|\cSigma - \bSigma^*\|_\infty \le \|\cSigma - \hSigma\|_\infty + \|\hSigma - \bSigma^*\|_\infty \le 2 \|\hSigma - \bSigma^*\|_\infty.
\end{equation}
This implies that the bound~\eqref{EqnTail} in Lemma~\ref{LemGLassoBd} holds with $\hSigma$ replaced by $\cSigma$, and $f(n,\delta)$ replaced by $f(n, \delta/2)$. Proceeding as in the proof of Theorem~\ref{ThmRobustGLasso} with these minor modifications, we arrive at the bound~\eqref{EqnProjGLassoBound}.

Turning to the derivation of the breakdown point, note that by Theorem 1 of \cite{Croux2014}, we have
\begin{equation}
\label{EqnOmegaSigma}
\epsilon_n(\cOmega(\X), \X) \ge \epsilon_n^+(\cSigma(\X), \X).
\end{equation}
We first show that
\begin{equation}
\label{EqnCheckHat}
\epsilon_n^+(\cSigma(\X), \X) \geq 50\%.
\end{equation}
Consider the estimator $\cSigma(\X^m)$, based on corrupted data. We have
\begin{align}
\label{EqnCheckOne}
\|\cSigma(\X^m) - \bSigma^*\|_\infty & \le 2 \|\hSigma(\X^m) - \bSigma^*\|_\infty \notag \\
& \le 2 \|\hSigma(\X^m)\|_\infty + 2\|\bSigma^*\|_\infty,
%
%& \le 2 \left(\|\hSigma(X^m)\|_2 + \|\bSigma^*\|_2\right),
\end{align}
where the first inequality follows from the bound~\eqref{EqnSigmaTriangle}, and the second inequality comes from the triangle inequality. Furthermore, note that since $\cSigma(\X^m) \succeq 0$ by construction, we have
\begin{equation}
\label{EqnCheckTwo}
\lambda_1(\cSigma(\X^m)) = \|\cSigma(\X^m)\|_2 \le \|\cSigma(\X^m) - \bSigma^*\|_2 + \|\bSigma^*\|_2 \le p \|\cSigma(\X^m) - \bSigma^*\|_\infty + \|\bSigma^*\|_2,
\end{equation}
where we have used the bound
\begin{equation*}
%\label{EqnMatrixSandwich}
\|\bA\|_\infty \le \|\bA\|_2 \le p \|\bA\|_\infty, \qquad \forall \bA \in \mathbb{R}^{p \times p},
\end{equation*}
in the last inequality. Combining inequalities~\eqref{EqnCheckOne} and~\eqref{EqnCheckTwo}, we then obtain
\begin{equation*}
\lambda_1(\cSigma(\X^m)) \le 2p \|\hSigma(\X^m)\|_\infty + 2p \|\bSigma^*\|_\infty + \|\bSigma^*\|_2,
\end{equation*}
so
\begin{equation}
\label{EqnLambdaDiff}
\left|\lambda_1(\cSigma(\X^m)) - \lambda_1(\cSigma(\X))\right| \le \lambda_1(\cSigma(\X)) + \left(2p \|\hSigma(\X^m)\|_\infty + 2p \|\bSigma^*\|_\infty + \|\bSigma^*\|_2\right).
\end{equation}
Finally, since the correlation estimators are bounded in magnitude by 1, we have
\begin{equation}
\label{EqnHatBound}
\|\hSigma(\X^m)\|_\infty \le \max_{1 \le i, j \le p} \hsigma_i(\X^m) \hsigma_j(\X^m),
\end{equation}
where $\left\{\hsigma_i(\X^m)\right\}_{1 \le i \le p}$ are the robust scale estimators based on $\X^m$, given by the MAD estimators calculated from the corresponding columns. Furthermore, the breakdown point of the MAD is 50\% \citep{Huber1981}, meaning the quantity on the right-hand side of inequality~\eqref{EqnHatBound} is finite when $\frac{m}{n} < 50\%$. Then by inequality~\eqref{EqnLambdaDiff} and the definition of the explosion breakdown point, we conclude that the bound~\eqref{EqnCheckHat} holds. By inequality~\eqref{EqnOmegaSigma}, we therefore have $\epsilon_n(\cOmega(\X), \X) \ge 50\%$, as well.

We now establish that $\epsilon_n(\cOmega(\X), \X) = 50\%$. Note that if we are allowed to corrupt more than 50\% of the entries in each column of the data matrix, the columnwise MAD estimates may be made arbitrarily small (say, smaller than some value $a$); indeed, we may simply replace more than half of the entries in each column by values in $(0, a)$. Consequently, the overall covariance estimator $\hSigma(\X^m)$ will have all entries bounded in magnitude by $[\Phi^{-1}(0.75)]^{-2} a^2$. We claim that the diagonal elements of $\cSigma(\X^m)$ must therefore be bounded in magnitude by $2[\Phi^{-1} (0.75)]^{-2} a^2$. Indeed, note that the matrix $\diag(\hSigma(\X^m))$ is feasible for the projection~\eqref{EqnCovProj}. Hence, we must have
\begin{equation*}
\|\hSigma(\X^m) - \cSigma(\X^m)\|_\infty \le \|\hSigma(\X^m) - \diag(\hSigma(\X^m))\|_\infty \le [\Phi^{-1}(0.75)]^{-2}a^2,
\end{equation*}
implying in particular that
\begin{equation*}
\|\diag(\cSigma(\X^m))\|_\infty \le \|\diag(\hSigma(\X^m))\|_\infty + \|\diag(\hSigma(\X^m)) - \diag(\cSigma(\X^m))\|_\infty \le 2 [\Phi^{-1}(0.75)]^{-2} a^2,
\end{equation*}
as claimed. Now note that the first-order optimality condition for the GLasso is given by
\begin{equation*}
\cSigma(\X^m) - \left(\cOmega(\X^m)\right)^{-1} + \lambda \cdot \sign\{\cOmega(\X^m) - \diag(\cOmega(\X^m))\} = 0,
\end{equation*}
where the $\sign$ function is computed entrywise, omitting the diagonal elements of $\cOmega(\X^m)$. In particular, this implies that the $\diag(\cSigma(\X^m)) = \diag\left\{\left(\cOmega(\X^m)\right)^{-1}\right\}$, so the diagonal elements of $\left(\cOmega(\X^m)\right)^{-1}$ are also bounded in magnitude by $2 [\Phi^{-1}(0.75)]^{-2} a^2$. Hence,
\begin{align*}
\lambda_p\left(\left(\cOmega(\X^m)\right)^{-1}\right) & = \min_{\|\bv\|_2 = 1} \bv^T \left(\left(\cOmega(\X^m)\right)^{-1}\right) \bv \\
& \le \min_{1 \le j \le p} \e_j^T \left(\left(\cOmega(\X^m)\right)^{-1}\right) \e_j \\
& \le \left\|\diag\left\{\left(\cOmega(\X^m)\right)^{-1}\right\} \right\|_\infty \\
& \le 2 [\Phi^{-1}(0.75)]^{-2} a^2,
\end{align*}
where the $\e_j$'s are the canonical basis vectors, and we have used the variational representation of eigenvalues of a Hermitian matrix to show that the minimum eigenvalue is bounded by the minimum diagonal entry.
%\begin{equation*}
%\lambda_p\left(\left(\cOmega(X^m)\right)^{-1}\right) \le a.
%\end{equation*}
This allows us to conclude that
\begin{equation*}
1 = \lambda_p\left(\cOmega(\X^m) \cdot \left(\cOmega(\X^m)\right)^{-1}\right) \le \lambda_1\left(\cOmega(\X^m)\right) \cdot \lambda_p\left(\left(\cOmega(\X^m)\right)^{-1}\right) \le \lambda_1\left(\cOmega(\X^m)\right) \cdot 2 [\Phi^{-1}(0.75)]^{-2} a^2,
\end{equation*}
where we have used the inequality $\lambda_p(\bA \bB) \le \lambda_1(\bA) \lambda_p(\bB)$, for $\bA, \bB \succeq 0$, in the first inequality \citep{Zha11}. Hence,
\begin{equation*}
\lambda_1\left(\cOmega(\X^m)\right) \ge \frac{[\Phi^{-1}(0.75)]^2}{2a^2}.
\end{equation*}
However, we may choose $a$ to be arbitrarily close to 0, implying that the maximum eigenvalue of $\cOmega(\X^m)$ may be made arbitrarily large, and the estimator breaks down. This concludes the proof.

\subsection{Proof of Theorem~\ref{ThmCLIMEBreakdown}}
\label{SecThmCLIMEBreakdown}

Clearly, $\epsilon_n(\hOmega, \X) \geq \frac{1}{n}$ for any $\X$, by the definition of the breakdown point. To show equality, we now provide a data matrix $X$ and a corrupted data matrix $\X^1$, where $\X^1$ differs from $\X$ in at most one element per column, and the CLIME problem is feasible for $\hSigma(\X)$ but infeasible for $\hSigma(\X^1)$. Consider the $n \times 2$ matrix $\X^1$, constructed as follows:
\begin{equation*}
\X^1 = \left(
\begin{array}{cc}
a_1 & -a_1 \\
a_2 & -a_2 \\
\vdots & \vdots \\
a_n & -a_n
\end{array}
\right),
\end{equation*}
where the $a_k$'s are all distinct. Note that the columns of $\X^1$ are perfectly negatively correlated; hence, the correlation matrix (computed from either Kendall's tau or Spearman's rho, for instance) is
\begin{equation*}
\left(
\begin{array}{cc}
1 & -1 \\
-1 & 1
\end{array}
\right).
\end{equation*}
Furthermore, we have $\hsigma_1 = \hsigma_2 := \hsigma$, since the data in the two columns are negatives of each other. It follows that
\begin{equation*}
\hSigma(\X^1) = \hsigma^2 \left(
\begin{array}{cc}
1 & -1 \\
-1 & 1
\end{array}
\right).
\end{equation*}
Clearly, the problem
\begin{equation*}
\beta_1: \left\|\hSigma(\X^1) \beta_1 - \left(\begin{array}{c}1 \\ 0 \end{array}\right)\right\|_\infty \le \lambda
\end{equation*}
is infeasible for $\lambda < \frac{1}{2}$. Hence, the CLIME estimator based on $\hSigma(\X^1)$ is infeasible.

On the other hand, we may construct an initial data matrix $\X$ such that the CLIME program based on $\hSigma(\X)$ is feasible, simply by altering the last row of $\X^1$. Suppose we change the last row of $\X^1$ to $(a_n, a_n)$. Then the columns are no longer perfectly negatively correlated, and it is easy to check that the correlation matrix of $\X$ will take the form
\begin{equation*}
\left(
\begin{array}{cc}
1 & a \\
a & 1
\end{array}
\right),
\end{equation*}
for some $|a| < 1$. Denoting the corresponding estimates of scale as $\hsigma_1$ and $\hsigma_2$, we then have
\begin{equation*}
\hSigma(\X) = \left(
\begin{array}{cc}
\hsigma_1^2 & a\hsigma_1\hsigma_2 \\
a\hsigma_1\hsigma_2 & \hsigma_2^2
\end{array}
\right).
\end{equation*}
Note that $\det\{\hSigma(\X)\} = \hsigma_1^2 \hsigma_2^2 (1 - a^2) > 0$. It follows that $\hSigma(\X)$ is invertible. In particular, the matrix $\left(\hSigma(\X)\right)^{-1}$ is always a feasible point for the CLIME program based on $\hSigma(\X)$.

Hence, we conclude that the CLIME program breaks down when even one corruption per column is allowed. It follows that $\epsilon_n(\hOmega, \X) = \frac{1}{n}$ for the constructed value of $\X$.

\section{Simulations}
\label{SecSimulations}

In this section, we perform simulation studies to examine the performance of the two robust covariance matrix estimators introduced in Section~\ref{SecBackground}, and also the robust precision matrix estimators obtained using the GLasso. We will refer to the two type of estimators as \texttt{Kendall} and \texttt{Spearman}, respectively.

%{\red In order to employ existing software for optimizing the GLasso program, we perform the projection~\eqref{EqnCovProj} of the robust covariance estimators onto the positive semidefinite cone. An efficient iterative algorithm for implementing this projection step may be found in \cite{HanEtal15}.}

For comparison, we also compute the following robust covariance matrix estimators, which are similarly plugged into the GLasso to obtain robust precision matrix estimators:
\begin{itemize}
\item \texttt{SpearmanU}: The pairwise covariance matrix estimator proposed in \cite{Croux2014}, where the MAD estimator is combined with Spearman's rho (without transformation): 
\[\hSigma_{ij} = \hsigma_i\hsigma_j\boldr^S_{ij}, \qquad\text{where }\hsigma_i = [\Phi^{-1}(0.75)]^{-1}\hat{d}_i.\]
\item \texttt{OGK}: The OGK estimator proposed in \cite{MaronnaZamar2002}, with scale estimator $Q_n$.
\item \texttt{NPD}: The pairwise covariance matrix estimator considered in \cite{Tarr2015}, where 
\[\tSigma_{ij} = \frac{1}{4} \left(\hsigma_{(i,j), +}^2 - \hsigma_{(i,j),-}^2\right),\]
$\hsigma_{(i,j),+}$ is the $Q_n$ statistic computed from $\{X_{ki} + X_{kj}: 1 \le k \le n\}$, and $\hsigma_{(i,j),-}$ is the $Q_n$ statistic computed from $\{X_{ki} - X_{kj}: 1 \le k \le n\}$.  An NPD projection is applied to $\tSigma$ to obtain the final positive semidefinite covariance matrix estimator:
\[\hSigma = \min_{\M\succeq 0} \|\tSigma - \M\|_F.\]
\end{itemize}
Further details for the orthogonalized Gnanedesikan-Kettenring (OGK) and nearest positive definite (NPD) procedures may be found in \cite{MaronnaZamar2002} and \cite{Higham2002}, respectively. The nonrobust GLasso, which takes the sample covariance matrix estimator as an input (\texttt{SampleCov}), as well as the inverse sample covariance matrix estimator (\texttt{InvCov}), applicable in the case $p<n$, are used as points of reference.  

%We select the tuning parameter $\lambda$ in the GLasso estimator using the modified BIC criterion, as is done in \cite{Croux2014}:
%\[\text{BIC}(\lambda) = \tr(\hOmega_\lambda\hSigma) -\log\det(\hOmega_\lambda) + \frac{\log n}{n}\sum_{i\leq j}\hat{e}_{ij}(\lambda),\]
%where $\hSigma$ is the (robust) covariance matrix estimator used in estimating $\hOmega$, $\hat{e}_{ij}(\lambda) = 1$ if $(\hOmega_\lambda)_{ij}\neq 0$ and $\hat{e}_{ij}(\lambda) = 0$ otherwise.  We compute $\hSigma_\lambda$ over a logarithmic spaced grid of ten values, as is done by default in the \texttt{huge}-package \citep{Zhao2014}.  The final estimator used is the one with the lowest $\text{BIC}(\lambda)$.

An implementation of the GLasso that allows the diagonal entries of the precision matrix estimator to be unpenalized is provided in the widely used \texttt{glasso} package. In this paper, however, we use the GLasso implementation from the \texttt{QUIC} package \citep{HsiEtal11}, since it does not require the input covariance matrix to be positive semidefinite, and speeds up substantially over \texttt{glasso}.  We select the tuning parameter $\lambda$ in GLasso by cross-validation: We first split the data into $K$ groups, or folds, of nearly equal size.  For a given $\lambda$ and $1\leq k\leq K$, we take the $k^{th}$ fold as the test set, and compute the precision matrix estimate $\hOmega_\lambda^{(-k)}$ based on the remaining $K-1$ folds.  We then compute the negative log-likelihood on the test set:
\[L^{(k)}(\lambda) = -\log\det\hOmega_\lambda^{(-k)} + \tr\left(\hSigma^{(k)}\hOmega_\lambda^{(-k)}\right),\]
where $\hSigma^{(k)}$ is the robust covariance estimate obtained from the test set. This is done over a logarithmically spaced grid of 15 values between $\lambda_{\text{max}} = \max_{i\neq j}|\hSigma_{ij}|$ and $\lambda_{\text{min}} = 0.01\lambda_{\text{max}}$, where $\hSigma$ is the robust covariance estimate computed from the whole data set.  The value of $\lambda$ that minimizes
\[\frac{1}{K}\sum_{k=1}^KL^{(k)}(\lambda)\]
is selected as the final tuning parameter.

\paragraph{\textbf{Simulation settings:}} We consider the following four sampling schemes, covering different structures of the true precision matrix $\bOmega^*\in\R^{p\times p}$.  The first three structures come from \cite{Cai2011}.
\begin{itemize}
\item Banded: $\bOmega^*_{ij} = 0.6^{|i-j|}$.
\item Sparse: $\bOmega^* = \B + \delta\I_p$, where $b_{ii} = 0$ and $b_{ij} = b_{ji}$, with $P(b_{ij} = 0.5) = 0.1$ and $P(b_{ij} = 0) = 0.9$, for $i\neq j$.  The parameter $\delta$ is chosen such that the condition number of $\bOmega^*$ equals $p$.  The matrix is then standardized to have unit diagonals.
\item Dense: $\bOmega^*_{ii} = 1$ and $\bOmega^*_{ij} = 0.5$, for $i\neq j$.
\item Diagonal: $\bOmega^* = \I_p$.
\end{itemize}

For each sampling scheme and dimension $p \in \{120, 400\}$, we generate $B = 100$ samples of size $n = 200$ from the multivariate normal distribution $N(\0, (\bOmega^*)^{-1})$.  We then add $5\%$ or $10\%$ of rowwise or cellwise contamination to the data, where the outliers are sampled independently from $N(10, 0.2)$.  We also simulate model deviation by generating all observations from either the multivariate $t$-distribution, $t_3(\0, (\bOmega^*)^{-1})$, or the alternative $t$-distribution, $t_3^*(\0, (\bOmega^*)^{-1})$, each with three degrees of freedom. Recall that $\X\sim t_{\nu}(\0, (\bOmega^*)^{-1})$, where $t_{\nu}(\0, \bOmega^*)^{-1})$ denotes the multivariate $t$-distribution with $\nu$ degrees of freedom, if 
\[\X = \Y/\sqrt{\tau},\]
where $\Y\sim N(\0, (\bOmega^*)^{-1})$ and $\tau\sim\Gamma(\nu/2, \nu/2)$.  The alternative $t$-distribution, denoted by $t_\nu^*$, is proposed in \cite{FinegoldDrton2011} as a generalization of the multivariate $t$-distribution.  We say that $\X\sim t_{\nu}^*(\0, (\bOmega^*)^{-1})$ if 
\[X_i = Y_i/\sqrt{\tau_i}, \qquad \forall 1 \le i \le p,\]
where the $p$ divisors $\tau_i\sim\Gamma(\nu/2, \nu/2)$ are independent.  In this case, the heaviness of the tails are different for different components of $\X$.

\paragraph{\textbf{Performance measures:}} We assess the performance of the covariance and precision matrix estimators via the deviations $\|\hSigma-\bSigma^*\|_\infty$ and $\|\hOmega-\bOmega^*\|_\infty$, respectively. To measure the accuracy of recovering the support of the true precision matrix, we also consider the false positive (FP) and false negative (FN) rates:
\[
\text{FP} = \frac{|\{(i, j): \hOmega_{ij}\neq 0, \bOmega^*_{ij} = 0\}|}{|\{(i, j): \bOmega^*_{ij} = 0\}|}, \quad \text{and} \quad
\text{FN} = \frac{|\{(i, j): \hOmega_{ij} = 0, \bOmega^*_{ij} \neq 0\}|}{|\{(i, j): \bOmega^*_{ij} \neq 0\}|}. 
\]
FP gives the proportion of zero elements in the true precision matrix that are incorrectly estimated to be nonzero, while FN gives the proportion of nonzero elements in the true precision matrix that are incorrectly estimated to be zero. Note that if $\bOmega^*$ has no zero entries, as in the case of the banded and dense structures, the quantity FP is undefined.

\vspace{5mm}
Tables~\ref{Tablep120a} and \ref{Tablep120b} show the results for $n = 200$ and $p = 120$. We summarize the salient points below:
\begin{itemize}
\item When the dataset is clean, \texttt{SampleCov} performs best in terms of both covariance and precision matrix estimation, across all sampling schemes.  Note that even though the data are uncontaminated, \texttt{InvCov} performs poorly, due to the fact that the sample covariance matrix has low precision when $p>n/2$.
\item In the case of rowwise contamination, the nonrobust \texttt{SampleCov} has the largest estimation error for the covariance matrix, as expected. Curiously, the precision matrix estimation error based on \texttt{SampleCov} is the lowest among all estimators. We do not have good explanation for this, but the tuning parameter selected for \texttt{SampleCov} by cross-validation tends to be smaller (as can be seen from its relatively low FN). \texttt{NPD}, \texttt{Kendall}, \texttt{Spearman}, and \texttt{SpearmanU} have similar performance in terms of both covariance and precision matrix estimation.  In all sampling schemes, \texttt{OGK} outperforms these four estimators for covariance estimation, but not consistently so for precision matrix estimation.
\item For covariance and precision matrix estimation under cellwise contamination, the \texttt{Kendall}, \texttt{Spearman}, and \texttt{SpearmanU} estimators perform the best.  \texttt{NPD} performs the worst among all cellwise robust covariance matrix estimators.  Nonetheless, \texttt{NPD} still beats \texttt{OGK}, which is designed to work well under rowwise contamination, and also beats the nonrobust \texttt{SampleCov}.
\item When the data are generated from the multivariate $t$-distribution or alternative $t$-distribution, we again see that \texttt{Kendall}, \texttt{Spearman}, and \texttt{SpearmanU} behave similarly and outperform all other estimators, across all sampling schemes.  
\item When $\bOmega^*$ is either sparse or diagonal, FP is low for all estimators except \texttt{InvCov}, under all contamination mechanisms.
\item Except for \texttt{InvCov}, FN is high when $\bOmega^*$ is banded or dense, under all contamination mechanisms.  This is expected because GLasso implicitly assumes the underlying $\bOmega^*$ to be sparse, which is not true in these cases.  When $\bOmega^*$ is sparse, the FN for \texttt{Kendall}, \texttt{Spearman}, and \texttt{SpearmanU} are relatively low compared to the other estimators.
\end{itemize}

\noindent Tables~\ref{Tablep400a} and \ref{Tablep400b} show the results for $n = 200$ and $p = 400$.  Since $p>n$, the inverse sample covariance matrix cannot be computed, hence is excluded from the analysis.  Overall, we obtain conclusions similar to those obtained in the first set of simulations:
\begin{itemize}
\item When the data are clean, \texttt{SampleCov} perform best in terms of estimation error, across all sampling schemes.  Immediately following are \texttt{OGK} and \texttt{NPD}, and then \texttt{Kendall}, \texttt{Spearman}, and \texttt{SpearmanU} (the last three have nearly the same performance).
\item Under rowwise contamination, \texttt{SampleCov} has the worst covariance estimation error, but also the best precision estimation error, across all sampling schemes.  \texttt{OGK} performs best in terms of covariance estimation, but not precision estimation. \texttt{NPD}, \texttt{Kendall}, \texttt{Spearman}, and \texttt{SpearmanU} have similar performance in nearly all cases.  When $\bOmega^*$ is diagonal and the contamination fraction is 10\%, \texttt{Kendall} turns out to have high precision estimation error, possibly because the selected tuning parameter in GLasso is too small (as can be seen by the high FP).  
\item In terms of estimation error under cellwise contamination, \texttt{OGK} performs nearly as badly as \texttt{SampleCov}.  \texttt{Kendall}, \texttt{Spearman}, and \texttt{SpearmanU} perform equally well, while \texttt{NPD} is slightly worse off.  
\item When the data are generated from the multivariate $t$-distribution or alternative $t$-distribution, \texttt{SampleCov} performs badly.  \texttt{Kendall}, \texttt{Spearman}, and \texttt{SpearmanU} perform similarly and outperform \texttt{OGK} and \texttt{NPD}, across all sampling schemes.  
\item In general, under all contamination mechanisms, when $\bOmega^*$ is either sparse or diagonal, FP is low for all estimators.  On the other hand, when $\bOmega^*$ is banded or dense, FN is high, as expected.  When $\bOmega^*$ is sparse, FN is not as low as desired.
\end{itemize}

In summary, \texttt{SampleCov} performs best for clean data.  Under rowwise contamination, \texttt{OGK} yields the best results in terms of covariance estimation.  Under cellwise contamination, \texttt{Kendall}, \texttt{Spearman}, and \texttt{SpearmanU} equally share the best performance, while \texttt{NPD} is slightly worse off.  \texttt{Kendall}, \texttt{Spearman}, and \texttt{SpearmanU} also perform very well when the data are generated from a multivariate $t$-distribution or the alternative $t$-distribution, although these latter cases are not covered by our theory.

\begin{landscape}
\begin{table}[ht]
\centering
\begin{tabular}{|l|l|rrrr|rrrr|rrrr|}
\hline
& & \multicolumn{4}{c|}{clean} & \multicolumn{4}{c|}{$5\%$ rowwise} & \multicolumn{4}{c|}{$10\%$ rowwise} \\
\hline
& & Cov & Prec & FP & FN & Cov & Prec & FP & FN & Cov & Prec & FP & FN \\ 
\hline
\multirow{7}{*}{Banded} & \texttt{SampleCov} & 1.11 & 0.30 &  & 0.85 & 5.91 & 0.31 &  & 0.60 & 10.44 & 0.31 &  & 0.61 \\ 
& \texttt{OGK} & 1.20 & 0.32 &  & 0.88 & 1.98 & 0.37 &  & 0.90 & 2.91 & 0.41 &  & 0.91 \\ 
& \texttt{NPD} & 1.26 & 0.35 &  & 0.96 & 2.24 & 0.37 &  & 0.72 & 3.39 & 0.39 &  & 0.71 \\ 
& \texttt{Kendall} & 1.73 & 0.33 &  & 0.87 & 2.50 & 0.32 &  & 0.63 & 3.37 & 0.31 &  & 0.63 \\ 
& \texttt{Spearman} & 1.73 & 0.33 &  & 0.87 & 2.50 & 0.33 &  & 0.64 & 3.37 & 0.33 &  & 0.64 \\ 
& \texttt{SpearmanU} & 1.73 & 0.34 &  & 0.88 & 2.50 & 0.34 &  & 0.64 & 3.37 & 0.34 &  & 0.63 \\ 
& \texttt{InvCov} & 1.11 & 1.68 &  & 0.00 & 5.91 & 1.83 &  & 0.00 & 10.44 & 2.09 &  & 0.00 \\ 
\hline
\multirow{7}{*}{Sparse} & \texttt{SampleCov} & 0.70 & 0.34 & 0.19 & 0.11 & 5.57 & 0.35 & 0.36 & 0.30 & 10.09 & 0.32 & 0.36 & 0.32 \\ 
& \texttt{OGK} & 0.79 & 0.39 & 0.18 & 0.15 & 1.62 & 0.51 & 0.18 & 0.20 & 2.39 & 0.59 & 0.17 & 0.24 \\ 
& \texttt{NPD} & 0.82 & 0.47 & 0.09 & 0.32 & 1.63 & 0.55 & 0.21 & 0.66 & 2.58 & 0.61 & 0.20 & 0.76 \\ 
& \texttt{Kendall} & 1.15 & 0.43 & 0.17 & 0.16 & 1.63 & 0.41 & 0.32 & 0.37 & 2.36 & 0.40 & 0.32 & 0.41 \\ 
& \texttt{Spearman} & 1.15 & 0.43 & 0.17 & 0.16 & 1.64 & 0.43 & 0.32 & 0.37 & 2.38 & 0.43 & 0.31 & 0.42 \\ 
& \texttt{SpearmanU} & 1.15 & 0.45 & 0.17 & 0.15 & 1.65 & 0.45 & 0.33 & 0.36 & 2.37 & 0.46 & 0.31 & 0.41 \\ 
& \texttt{InvCov} & 0.70 & 2.83 & 1.00 & 0.00 & 5.57 & 3.14 & 1.00 & 0.00 & 10.09 & 3.54 & 1.00 & 0.00 \\ 
\hline
\multirow{7}{*}{Dense} & \texttt{SampleCov} & 0.60 & 0.60 &  & 0.99 & 5.54 & 0.61 &  & 0.75 & 10.05 & 0.60 &  & 0.75 \\ 
& \texttt{OGK} & 0.63 & 0.61 &  & 0.99 & 1.18 & 0.68 &  & 0.99 & 1.88 & 0.74 &  & 0.99 \\ 
& \texttt{NPD} & 0.67 & 0.62 &  & 0.99 & 1.23 & 0.65 &  & 0.82 & 1.89 & 0.69 &  & 0.79 \\ 
& \texttt{Kendall} & 1.00 & 0.66 &  & 0.99 & 1.37 & 0.64 &  & 0.79 & 1.91 & 0.64 &  & 0.78 \\ 
& \texttt{Spearman} & 1.00 & 0.66 &  & 0.99 & 1.37 & 0.64 &  & 0.79 & 1.91 & 0.64 &  & 0.77 \\ 
& \texttt{SpearmanU} & 0.99 & 0.66 &  & 0.99 & 1.37 & 0.64 &  & 0.78 & 1.91 & 0.65 &  & 0.77 \\ 
& \texttt{InvCov} & 0.60 & 2.63 &  & 0.00 & 5.54 & 1.28 &  & 0.00 & 10.05 & 1.48 &  & 0.00 \\ 
\hline
\multirow{7}{*}{Diagonal} & \texttt{SampleCov} & 0.30 & 0.31 & 0.00 & 0.00 & 5.31 & 0.26 & 0.24 & 0.00 & 9.84 & 0.28 & 0.24 & 0.00 \\ 
& \texttt{OGK} & 0.32 & 0.33 & 0.00 & 0.00 & 0.55 & 0.35 & 0.00 & 0.00 & 0.80 & 0.44 & 0.00 & 0.00 \\ 
& \texttt{NPD} & 0.33 & 0.35 & 0.00 & 0.00 & 0.63 & 0.31 & 0.18 & 0.00 & 0.98 & 0.39 & 0.21 & 0.00 \\ 
& \texttt{Kendall} & 0.51 & 0.62 & 0.00 & 0.00 & 0.68 & 0.51 & 0.20 & 0.00 & 0.96 & 0.46 & 0.21 & 0.00 \\ 
& \texttt{Spearman} & 0.51 & 0.62 & 0.00 & 0.00 & 0.68 & 0.52 & 0.21 & 0.00 & 0.96 & 0.47 & 0.22 & 0.00 \\ 
& \texttt{SpearmanU} & 0.51 & 0.62 & 0.00 & 0.00 & 0.68 & 0.52 & 0.21 & 0.00 & 0.96 & 0.45 & 0.23 & 0.00 \\ 
& \texttt{InvCov} & 0.30 & 2.81 & 1.00 & 0.00 & 5.31 & 3.19 & 1.00 & 0.00 & 9.84 & 3.60 & 1.00 & 0.00 \\ 
\hline
\end{tabular}
\caption{Simulation results for seven estimators and four sampling schemes, when $n = 200$ and $p = 120$. Performance is measured by $\|\hSigma-\bSigma^*\|_\infty$ for covariance matrix estimation (Cov), $\|\hOmega-\bOmega^*\|_\infty$ for precision matrix estimation (Prec), and false positive rate (FP) and false negative rate (FN) for support recovery of the true precision matrix.  The results are averaged over 100 replications.}
\label{Tablep120a}
\end{table}
\end{landscape}

\begin{landscape}
\begin{table}[ht]
\centering
\begin{tabular}{|l|l|rrrr|rrrr|rrrr|rrrr|}
\hline
& & \multicolumn{4}{c|}{$5\%$ cellwise} & \multicolumn{4}{c|}{$10\%$ cellwise} & \multicolumn{4}{c|}{multivariate $t$} & \multicolumn{4}{c|}{alternative $t$} \\
\hline
& & Cov & Prec & FP & FN & Cov & Prec & FP & FN & Cov & Prec & FP & FN & Cov & Prec & FP & FN \\ 
\hline
\multirow{7}{*}{Banded} & \texttt{SampleCov} & 8.33 & 0.51 &  & 0.97 & 13.09 & 0.54 &  & 0.99 & 18.31 & 0.49 &  & 0.87 & 57.72 & 0.57 &  & 0.93 \\ 
& \texttt{OGK} & 8.10 & 0.51 &  & 0.95 & 13.15 & 0.54 &  & 0.99 & 3.85 & 0.43 &  & 0.92 & 12.15 & 0.53 &  & 0.92 \\ 
& \texttt{NPD} & 2.78 & 0.41 &  & 0.95 & 4.70 & 0.46 &  & 0.96 & 4.06 & 0.44 &  & 0.96 & 4.53 & 0.46 &  & 0.96 \\ 
& \texttt{Kendall} & 2.43 & 0.40 &  & 0.92 & 3.67 & 0.45 &  & 0.92 & 3.32 & 0.41 &  & 0.90 & 3.60 & 0.42 &  & 0.90 \\ 
& \texttt{Spearman} & 2.43 & 0.41 &  & 0.92 & 3.67 & 0.45 &  & 0.92 & 3.32 & 0.41 &  & 0.91 & 3.60 & 0.42 &  & 0.90 \\ 
& \texttt{SpearmanU} & 2.43 & 0.41 &  & 0.93 & 3.67 & 0.45 &  & 0.93 & 3.32 & 0.42 &  & 0.91 & 3.60 & 0.43 &  & 0.90 \\ 
& \texttt{InvCov} & 8.33 & 0.41 &  & 0.00 & 13.09 & 0.46 &  & 0.00 & 18.31 & 1.26 &  & 0.00 & 57.72 & 0.53 &  & 0.00 \\ 
\hline
\multirow{7}{*}{Sparse} & \texttt{SampleCov} & 8.39 & 0.90 & 0.05 & 0.81 & 13.25 & 0.93 & 0.01 & 0.91 & 11.47 & 0.77 & 0.14 & 0.43 & 32.95 & 0.94 & 0.12 & 0.44 \\ 
& \texttt{OGK} & 8.18 & 0.90 & 0.06 & 0.77 & 13.71 & 0.94 & 0.01 & 0.90 & 3.38 & 0.65 & 0.16 & 0.23 & 8.67 & 0.86 & 0.16 & 0.34 \\ 
& \texttt{NPD} & 2.15 & 0.61 & 0.06 & 0.45 & 4.04 & 0.73 & 0.05 & 0.59 & 3.17 & 0.69 & 0.08 & 0.45 & 3.31 & 0.71 & 0.07 & 0.49 \\ 
& \texttt{Kendall} & 1.58 & 0.61 & 0.16 & 0.30 & 2.44 & 0.72 & 0.13 & 0.46 & 2.34 & 0.58 & 0.15 & 0.25 & 2.32 & 0.62 & 0.16 & 0.22 \\ 
& \texttt{Spearman} & 1.58 & 0.62 & 0.15 & 0.30 & 2.44 & 0.73 & 0.13 & 0.46 & 2.34 & 0.59 & 0.15 & 0.25 & 2.32 & 0.62 & 0.15 & 0.23 \\ 
& \texttt{SpearmanU} & 1.58 & 0.63 & 0.16 & 0.30 & 2.44 & 0.73 & 0.13 & 0.46 & 2.34 & 0.60 & 0.15 & 0.25 & 2.32 & 0.63 & 0.16 & 0.22 \\ 
& \texttt{InvCov} & 8.39 & 0.77 & 1.00 & 0.00 & 13.25 & 0.85 & 1.00 & 0.00 & 11.47 & 2.10 & 1.00 & 0.00 & 32.95 & 0.87 & 1.00 & 0.00 \\ 
\hline
\multirow{7}{*}{Dense} & \texttt{SampleCov} & 8.39 & 0.90 &  & 0.99 & 13.25 & 0.93 &  & 0.99 & 10.06 & 0.88 &  & 0.98 & 31.24 & 0.95 &  & 0.99 \\ 
& \texttt{OGK} & 8.02 & 0.90 &  & 0.99 & 13.14 & 0.93 &  & 0.99 & 2.14 & 0.76 &  & 0.99 & 6.82 & 0.89 &  & 0.99 \\ 
& \texttt{NPD} & 1.51 & 0.71 &  & 0.99 & 2.64 & 0.78 &  & 0.99 & 2.21 & 0.76 &  & 0.99 & 2.50 & 0.78 &  & 0.99 \\ 
& \texttt{Kendall} & 1.36 & 0.70 &  & 0.99 & 2.00 & 0.75 &  & 0.99 & 1.84 & 0.74 &  & 0.99 & 2.08 & 0.75 &  & 0.99 \\ 
& \texttt{Spearman} & 1.36 & 0.70 &  & 0.99 & 2.00 & 0.75 &  & 0.99 & 1.84 & 0.74 &  & 0.99 & 2.08 & 0.75 &  & 0.99 \\ 
& \texttt{SpearmanU} & 1.36 & 0.70 &  & 0.99 & 2.00 & 0.75 &  & 0.99 & 1.84 & 0.74 &  & 0.99 & 2.08 & 0.75 &  & 0.99 \\ 
& \texttt{InvCov} & 8.39 & 0.78 &  & 0.00 & 13.25 & 0.85 &  & 0.00 & 10.06 & 1.88 &  & 0.00 & 31.24 & 0.88 &  & 0.00 \\ 
\hline
\multirow{7}{*}{Diagonal} & \texttt{SampleCov} & 8.44 & 0.89 & 0.00 & 0.00 & 13.37 & 0.93 & 0.00 & 0.00 & 5.07 & 0.77 & 0.01 & 0.00 & 15.41 & 0.90 & 0.00 & 0.00 \\ 
& \texttt{OGK} & 7.89 & 0.89 & 0.00 & 0.00 & 13.15 & 0.93 & 0.00 & 0.00 & 1.07 & 0.51 & 0.00 & 0.00 & 3.44 & 0.77 & 0.00 & 0.00 \\ 
& \texttt{NPD} & 0.76 & 0.43 & 0.00 & 0.00 & 1.37 & 0.58 & 0.00 & 0.00 & 1.11 & 0.52 & 0.00 & 0.00 & 1.25 & 0.55 & 0.00 & 0.00 \\ 
& \texttt{Kendall} & 0.70 & 0.44 & 0.00 & 0.00 & 1.00 & 0.50 & 0.00 & 0.00 & 0.93 & 0.48 & 0.00 & 0.00 & 1.02 & 0.50 & 0.00 & 0.00 \\ 
& \texttt{Spearman} & 0.70 & 0.44 & 0.00 & 0.00 & 1.00 & 0.50 & 0.00 & 0.00 & 0.93 & 0.48 & 0.00 & 0.00 & 1.02 & 0.50 & 0.00 & 0.00 \\ 
& \texttt{SpearmanU} & 0.70 & 0.44 & 0.00 & 0.00 & 1.00 & 0.50 & 0.00 & 0.00 & 0.93 & 0.48 & 0.00 & 0.00 & 1.02 & 0.50 & 0.00 & 0.00 \\ 
& \texttt{InvCov} & 8.44 & 0.76 & 1.00 & 0.00 & 13.37 & 0.85 & 1.00 & 0.00 & 5.07 & 2.12 & 1.00 & 0.00 & 15.41 & 0.92 & 1.00 & 0.00 \\ 
\hline
\end{tabular}
\caption{Simulation results for seven estimators and four sampling schemes, when $n = 200$ and $p = 120$. Performance is measured by $\|\hSigma-\bSigma^*\|_\infty$ for covariance matrix estimation (Cov), $\|\hOmega-\bOmega^*\|_\infty$ for precision matrix estimation (Prec), and false positive rate (FP) and false negative rate (FN) for support recovery of the true precision matrix.  The results are averaged over 100 replications.}
\label{Tablep120b}
\end{table}
\end{landscape}

\begin{landscape}
\begin{table}[ht]
\centering
\begin{tabular}{|l|l|rrrr|rrrr|rrrr|}
\hline
& & \multicolumn{4}{c|}{clean} & \multicolumn{4}{c|}{$5\%$ rowwise} & \multicolumn{4}{c|}{$10\%$ rowwise} \\
\hline
& & Cov & Prec & FP & FN & Cov & Prec & FP & FN & Cov & Prec & FP & FN \\ 
\hline
\multirow{6}{*}{Banded} & \texttt{SampleCov} & 1.24 & 0.33 &  & 0.96 & 5.98 & 0.34 &  & 0.85 & 10.34 & 0.35 &  & 0.86 \\ 
& \texttt{OGK} & 1.38 & 0.34 &  & 0.96 & 2.20 & 0.38 &  & 0.95 & 3.10 & 0.41 &  & 0.95 \\ 
& \texttt{NPD} & 1.64 & 0.38 &  & 0.99 & 2.75 & 0.40 &  & 0.89 & 3.95 & 0.42 &  & 0.89 \\ 
& \texttt{Kendall} & 2.07 & 0.37 &  & 0.97 & 2.76 & 0.34 &  & 0.85 & 3.73 & 0.35 &  & 0.86 \\ 
& \texttt{Spearman} & 2.07 & 0.37 &  & 0.97 & 2.76 & 0.35 &  & 0.86 & 3.73 & 0.35 &  & 0.86 \\ 
& \texttt{SpearmanU} & 2.07 & 0.37 &  & 0.97 & 2.76 & 0.35 &  & 0.86 & 3.73 & 0.35 &  & 0.86 \\ 
\hline
\multirow{6}{*}{Sparse} & \texttt{SampleCov} & 0.81 & 0.44 & 0.09 & 0.56 & 5.61 & 0.43 & 0.14 & 0.73 & 9.93 & 0.40 & 0.14 & 0.74 \\ 
& \texttt{OGK} & 0.96 & 0.45 & 0.09 & 0.59 & 1.86 & 0.53 & 0.09 & 0.62 & 2.87 & 0.61 & 0.10 & 0.62 \\ 
& \texttt{NPD} & 1.11 & 0.59 & 0.03 & 0.79 & 2.14 & 0.63 & 0.08 & 0.93 & 3.61 & 0.68 & 0.08 & 0.95 \\ 
& \texttt{Kendall} & 1.35 & 0.50 & 0.09 & 0.60 & 1.76 & 0.48 & 0.12 & 0.77 & 2.71 & 0.47 & 0.12 & 0.79 \\ 
& \texttt{Spearman} & 1.35 & 0.50 & 0.08 & 0.60 & 1.77 & 0.49 & 0.12 & 0.77 & 2.72 & 0.49 & 0.12 & 0.79 \\ 
& \texttt{SpearmanU} & 1.35 & 0.51 & 0.09 & 0.60 & 1.78 & 0.51 & 0.13 & 0.77 & 2.72 & 0.51 & 0.12 & 0.79 \\  
\hline
\multirow{6}{*}{Dense} & \texttt{SampleCov} & 0.69 & 0.62 &  & 1.00 & 5.53 & 0.62 &  & 0.91 & 9.90 & 0.60 &  & 0.91 \\ 
& \texttt{OGK} & 0.78 & 0.64 &  & 1.00 & 1.29 & 0.69 &  & 1.00 & 1.92 & 0.74 &  & 1.00 \\ 
& \texttt{NPD} & 0.89 & 0.65 &  & 1.00 & 1.54 & 0.68 &  & 0.93 & 2.24 & 0.72 &  & 0.91 \\ 
& \texttt{Kendall} & 1.17 & 0.68 &  & 1.00 & 1.54 & 0.65 &  & 0.92 & 2.12 & 0.70 &  & 0.91 \\ 
& \texttt{Spearman} & 1.17 & 0.68 &  & 1.00 & 1.54 & 0.65 &  & 0.92 & 2.12 & 0.65 &  & 0.91 \\ 
& \texttt{SpearmanU} & 1.17 & 0.68 &  & 1.00 & 1.54 & 0.66 &  & 0.92 & 2.12 & 0.65 &  & 0.91 \\ 
\hline
\multirow{6}{*}{Diagonal} & \texttt{SampleCov} & 0.34 & 0.37 & 0.00 & 0.00 & 5.28 & 0.26 & 0.09 & 0.00 & 9.64 & 0.32 & 0.09 & 0.00 \\ 
& \texttt{OGK} & 0.38 & 0.38 & 0.00 & 0.00 & 0.58 & 0.36 & 0.00 & 0.00 & 0.78 & 0.44 & 0.00 & 0.00 \\ 
& \texttt{NPD} & 0.45 & 0.32 & 0.00 & 0.00 & 0.78 & 0.37 & 0.07 & 0.00 & 1.15 & 0.44 & 0.09 & 0.00 \\ 
& \texttt{Kendall} & 0.59 & 0.72 & 0.00 & 0.00 & 0.78 & 0.60 & 0.08 & 0.00 & 1.07 & 4.83 & 0.33 & 0.00 \\ 
& \texttt{Spearman} & 0.59 & 0.72 & 0.00 & 0.00 & 0.78 & 0.60 & 0.08 & 0.00 & 1.07 & 0.57 & 0.08 & 0.00 \\ 
& \texttt{SpearmanU} & 0.59 & 0.72 & 0.00 & 0.00 & 0.78 & 0.59 & 0.08 & 0.00 & 1.07 & 0.56 & 0.09 & 0.00 \\  
\hline
\end{tabular}
\caption{Simulation results for six estimators and four sampling schemes, when $n = 200$ and $p = 400$. Performance is measured by $\|\hSigma-\bSigma^*\|_\infty$ for covariance matrix estimation (Cov), $\|\hOmega-\bOmega^*\|_\infty$ for precision matrix estimation (Prec), and false positive rate (FP) and false negative rate (FN) for support recovery of the true precision matrix.  The results are averaged over 100 replications.}
\label{Tablep400a}
\end{table}
\end{landscape}

\begin{landscape}
\begin{table}[ht]
\centering
\begin{tabular}{|l|l|rrrr|rrrr|rrrr|rrrr|}
\hline
& & \multicolumn{4}{c|}{$5\%$ cellwise} & \multicolumn{4}{c|}{$10\%$ cellwise} & \multicolumn{4}{c|}{multivariate $t$} & \multicolumn{4}{c|}{alternative $t$} \\
\hline
& & Cov & Prec & FP & FN & Cov & Prec & FP & FN & Cov & Prec & FP & FN & Cov & Prec & FP & FN \\ 
\hline
\multirow{6}{*}{Banded} & \texttt{SampleCov} & 8.90 & 0.48 &  & 0.69 & 13.70 & 0.46 &  & 0.44 & 22.41 & 0.45 &  & 0.87 & 137.82 & 0.57 &  & 0.86 \\ 
& \texttt{OGK} & 8.79 & 0.48 &  & 0.66 & 13.89 & 0.46 &  & 0.39 & 3.97 & 0.44 &  & 0.95 & 18.41 & 0.51 &  & 0.53 \\ 
& \texttt{NPD} & 4.04 & 0.45 &  & 0.98 & 7.03 & 0.45 &  & 0.78 & 5.03 & 0.46 &  & 0.94 & 5.83 & 0.48 &  & 0.97 \\ 
& \texttt{Kendall} & 2.89 & 0.42 &  & 0.96 & 4.11 & 0.46 &  & 0.98 & 3.69 & 0.42 &  & 0.96 & 3.99 & 0.43 &  & 0.97 \\ 
& \texttt{Spearman} & 2.89 & 0.42 &  & 0.96 & 4.11 & 0.46 &  & 0.98 & 3.69 & 0.42 &  & 0.96 & 3.99 & 0.43 &  & 0.97 \\ 
& \texttt{SpearmanU} & 2.89 & 0.42 &  & 0.96 & 4.11 & 0.46 &  & 0.97 & 3.69 & 0.42 &  & 0.96 & 3.99 & 0.44 &  & 0.97 \\ 
\hline
\multirow{6}{*}{Sparse} & \texttt{SampleCov} & 8.98 & 0.91 & 0.01 & 0.96 & 13.82 & 0.85 & 0.52 & 0.45 & 13.53 & 0.79 & 0.05 & 0.80 & 79.44 & 0.96 & 0.04 & 0.85 \\ 
& \texttt{OGK} & 8.83 & 0.91 & 0.02 & 0.94 & 14.48 & 0.88 & 0.57 & 0.40 & 3.83 & 0.66 & 0.10 & 0.62 & 12.48 & 0.90 & 0.07 & 0.77 \\ 
& \texttt{NPD} & 3.10 & 0.72 & 0.03 & 0.83 & 6.15 & 0.82 & 0.02 & 0.87 & 4.40 & 0.76 & 0.03 & 0.84 & 4.67 & 0.78 & 0.03 & 0.86 \\ 
& \texttt{Kendall} & 1.80 & 0.64 & 0.06 & 0.74 & 2.94 & 0.74 & 0.05 & 0.82 & 2.61 & 0.63 & 0.07 & 0.69 & 2.71 & 0.66 & 0.07 & 0.67 \\ 
& \texttt{Spearman} & 1.80 & 0.65 & 0.06 & 0.74 & 2.94 & 0.74 & 0.05 & 0.82 & 2.61 & 0.64 & 0.07 & 0.69 & 2.71 & 0.66 & 0.07 & 0.67 \\ 
& \texttt{SpearmanU} & 1.80 & 0.65 & 0.07 & 0.73 & 2.94 & 0.75 & 0.05 & 0.82 & 2.61 & 0.64 & 0.07 & 0.68 & 2.71 & 0.66 & 0.07 & 0.67 \\ 
\hline
\multirow{6}{*}{Dense} & \texttt{SampleCov} & 8.96 & 0.90 &  & 0.96 & 13.81 & 0.85 &  & 0.46 & 12.64 & 0.88 &  & 0.99 & 79.01 & 0.98 &  & 1.00 \\ 
& \texttt{OGK} & 8.62 & 0.90 &  & 0.93 & 13.64 & 0.85 &  & 0.38 & 2.24 & 0.76 &  & 1.00 & 10.33 & 0.92 &  & 1.00 \\ 
& \texttt{NPD} & 2.35 & 0.77 &  & 1.00 & 4.28 & 0.84 &  & 1.00 & 2.82 & 0.79 &  & 1.00 & 3.22 & 0.81 &  & 1.00 \\ 
& \texttt{Kendall} & 1.64 & 0.72 &  & 1.00 & 2.29 & 0.77 &  & 1.00 & 2.12 & 0.75 &  & 1.00 & 2.25 & 0.76 &  & 1.00 \\ 
& \texttt{Spearman} & 1.64 & 0.72 &  & 1.00 & 2.29 & 0.77 &  & 1.00 & 2.12 & 0.75 &  & 1.00 & 2.25 & 0.76 &  & 1.00 \\ 
& \texttt{SpearmanU} & 1.64 & 0.72 &  & 1.00 & 2.29 & 0.77 &  & 1.00 & 2.12 & 0.75 &  & 1.00 & 2.25 & 0.76 &  & 1.00 \\ 
\hline
\multirow{6}{*}{Diagonal} & \texttt{SampleCov} & 9.03 & 0.90 & 0.00 & 0.00 & 13.93 & 0.87 & 0.47 & 0.00 & 6.33 & 0.77 & 0.01 & 0.00 & 39.73 & 0.95 & 0.00 & 0.00 \\ 
& \texttt{OGK} & 8.60 & 0.90 & 0.00 & 0.00 & 13.74 & 0.87 & 0.54 & 0.00 & 1.11 & 0.52 & 0.00 & 0.00 & 5.17 & 0.84 & 0.00 & 0.00 \\ 
& \texttt{NPD} & 1.20 & 0.54 & 0.00 & 0.00 & 2.19 & 0.69 & 0.00 & 0.00 & 1.42 & 0.58 & 0.00 & 0.00 & 1.62 & 0.62 & 0.00 & 0.00 \\ 
& \texttt{Kendall} & 0.81 & 0.52 & 0.00 & 0.00 & 1.15 & 0.54 & 0.00 & 0.00 & 1.06 & 0.52 & 0.00 & 0.00 & 1.14 & 0.54 & 0.00 & 0.00 \\ 
& \texttt{Spearman} & 0.81 & 0.52 & 0.00 & 0.00 & 1.15 & 0.54 & 0.00 & 0.00 & 1.06 & 0.52 & 0.00 & 0.00 & 1.14 & 0.54 & 0.00 & 0.00 \\ 
& \texttt{SpearmanU} & 0.81 & 0.52 & 0.00 & 0.00 & 1.15 & 0.54 & 0.00 & 0.00 & 1.06 & 0.52 & 0.00 & 0.00 & 1.14 & 0.54 & 0.00 & 0.00 \\  
\hline
\end{tabular}
\caption{Simulation results for six estimators and four sampling schemes, when $n = 200$ and $p = 400$. Performance is measured by $\|\hSigma-\bSigma^*\|_\infty$ for covariance matrix estimation (Cov), $\|\hOmega-\bOmega^*\|_\infty$ for precision matrix estimation (Prec), and false positive rate (FP) and false negative rate (FN) for support recovery of the true precision matrix.  The results are averaged over 100 replications.}
\label{Tablep400b}
\end{table}
\end{landscape}

%%%%%%%%%%%%%%%%%%%%%%%%%%%%%%%%%%%%%%%%%%%%%%%%%%%%%%%%%%%%%%%%%%%%%%%%%%%%%%

\section{Discussion}
\label{SecDiscussion}

In this paper, we have derived statistical error bounds for high-dimensional robust precision matrix estimators, when data are drawn from a multivariate normal distribution and then observed subject to cellwise contamination.  We show that in such settings, the precision matrix estimators that are obtained by plugging in pairwise robust covariance estimators to the GLasso or CLIME routine, as suggested by \cite{Croux2014} and \cite{Tarr2015}, have error bounds that match standard high-dimensional bounds for uncontaminated precision matrix estimation, up to an additive factor involving a constant multiple of the contamination fraction $\epsilon$.  Our results for precision matrix estimators are derived via estimation error bounds for robust covariance matrix estimators, which have similar deviation properties.

The results of our paper naturally suggest several venues for future work. As discussed earlier, our results seem to indicate that covariance estimators based on bounded-influence estimators of correlation and scale give rise to statistical error bounds of the form derived in our paper, and it would be interesting to rigorize this notion, as a further attempt to connect the fields of robust and high-dimensional statistics. It would also be interesting to relate the nonasymptotic statistical error bounds to the behavior of the sensitivity curve of the robust covariance estimator, which is the finite-sample analog of the influence function. We have also left open the question of calculating the breakdown point for the CLIME estimator with respect to more general data matrices, as well as the breakdown behavior of CLIME and GLasso under different notions of breakdown point. Although our results imply the superiority of the GLasso over the CLIME estimator from the perspective of the finite-sample breakdown point, this may only be part of the story.

Lastly, it would be interesting to generalize our study to other classes of distributions. In one direction, it would be possible to study contaminated versions of other distributions besides the multivariate Gaussian, for which the precision matrix encodes information about the underlying graphical model (e.g., Ising models on trees). A harder question to tackle would be the problem of robust graphical model estimation in settings where the structure of the graph is not encoded in the precision matrix alone. Finally, one could consider robust estimation of scatter matrices, when the uncontaminated data are drawn from an elliptical distribution. In that case, the proposed Kendall's tau and Spearman's rho correlation coefficients would still be Fisher consistent upon taking the respective sine transformations, so similar error bounds should hold.  As demonstrated in our simulation results, the pairwise covariance estimators based on Kendall's tau and Spearman's rho perform reasonably well when data are generated from either the multivariate $t$-distribution or the alternative $t$-distribution. This motivates studying the convergence rates of the same covariance matrix estimators under heavy-tailed or elliptical distributions.

The problem of estimating high-dimensional covariance matrices under various structural assumptions has also been widely studied. Various families of structured covariance matrices have been introduced, including bandable matrices \citep{CaiZhangZhou2010}, Toeplitz matrices \citep{CaiRenZhou2013}, and sparse matrices \citep{BickelLevina2008, CaiZho12}.  The proposed covariance matrix estimators involve regularizing the sample covariance matrix in accordance to structural assumptions.  It would be interesting to study robust versions of these structured covariance matrix estimators under a model such as cellwise contamination. Besides graphical models, covariance matrix estimation is also useful for statistical methods such as linear discriminant analysis and principal component analysis. Several high-dimensional procedures have been proposed with proven theoretical guarantees when data are uncontaminated \citep{CaiLiu2011, Vu2013}, and it would be interesting to study robust adaptations of these procedures, as well.

\bibliographystyle{chicago}
\bibliography{ref.bib}

%%%%%%%%%%%%%%%%%%%%%%%%%%%%%%%%%%%%%%%%%%%%%%%%%%%%%%%%%%%%%%%%%%%%%%%%%%%%%
\newpage
\appendix

\section{Proofs of supporting lemmas}

In this Appendix, we provide the proofs of the technical lemmas used to establish Theorems~\ref{ThmRobustCovK} and~\ref{ThmRobustCovS} in Section~\ref{SecCovEst}.

\subsection{Proof of Lemma~\ref{kendalltauconsistency}}
\label{AppLemKendallTau}

When $i = j$, we have
\begin{align*}
\boldr_{ii}^K & = \frac{2}{n(n-1)}\sum_{k<\ell} \sign^2(X_{ki}-X_{\ell i}) \\
& = \frac{2}{n(n-1)}\sum_{k<\ell} (1-\1(X_{ki}=X_{\ell i})) \\
& = 1 - \frac{2}{n(n-1)}\sum_{k<\ell}\1(X_{ki}=X_{\ell i}).
\end{align*}
Hence,
\begin{align*}
\bigg|\sin\Big(\frac{\pi}{2}\boldr_{ii}^K\Big) - \brho_{ii}\bigg| &= \bigg|\sin\Big(\frac{\pi}{2} - \frac{\pi}{n(n-1)}\sum_{k<\ell}\1(X_{ki}=X_{\ell i})\Big) - 1\bigg| \nonumber\\
&= \bigg|\cos\Big(\frac{\pi}{n(n-1)}\sum_{k<\ell}\1(X_{ki}=X_{\ell i})\Big) - \cos(0)\bigg| \nonumber\\
&\le \frac{\pi}{2}q_i,
\end{align*}
where 
\[q_i = \frac{2}{n(n-1)} \sum_{k < \ell} \1(X_{ki} = X_{\ell i})\]
is a $U$-statistic, and the last inequality follows from the fact that $\cos(x)$ is 1-Lipschitz.  By Hoeffding's inequality for $U$-statistics, we have
\begin{equation}
P\bigg(\bigg|\sin\Big(\frac{\pi}{2}\boldr_{ii}^K\Big) - \brho_{ii}\bigg| \ge t\bigg) \leq P\bigg(q_i\ge \frac{2t}{\pi}\bigg) \leq \exp\bigg(-\frac{4nt^2}{\pi^2}\bigg). \label{ii}
\end{equation}

Now, consider the case where $i\neq j$.  Note that
\begin{align}
\label{EqnChoco}
\bigg|\sin\Big(\frac{\pi}{2}\htauij\Big) - \brho_{ij}\bigg| &\leq \bigg|\sin\Big(\frac{\pi}{2}\htauij\Big) - \sin\Big(\frac{\pi}{2}\brho^K_{ij}\Big)\bigg| + \bigg|\sin\Big(\frac{\pi}{2}\brho^K_{ij}\Big) - \brho_{ij}\bigg|,
\end{align}
where $\brho^K_{ij} = E(\htauij)$ and the expectation is with respect to the distribution under model \eqref{cellContam}.  Since $\htauij$ is a $U$-statistic with kernel bounded between $-1$ and 1, Hoeffding's inequality and the fact that $\sin(x)$ is 1-Lipschitz implies that the first term on the right-hand side of inequality~\eqref{EqnChoco} satisfies
\begin{equation}
P\bigg(\bigg|\sin\Big(\frac{\pi}{2}\htauij\Big) - \sin\Big(\frac{\pi}{2}\brho^K_{ij}\Big)\bigg| \ge t\bigg) \leq P\bigg(|\htauij - \brho^K_{ij}| \ge \frac{2}{\pi}t\bigg) \leq 2\exp\bigg(-\frac{nt^2}{\pi^2}\bigg). \label{ij}
\end{equation}
Combining inequalities~\eqref{ii} and~\eqref{ij} and taking $t = C\sqrt{\frac{\log p}{n}}$, we conclude that
%\[P\bigg(\left\{\max_{i\neq j}\bigg|\sin\Big(\frac{\pi}{2}\htauij\Big) - \sin\Big(\frac{\pi}{2}\rho^K_{ij}\Big)\bigg| \ge t\right\} \cup \left\{\max_{1\leq i\leq p} \bigg|\sin\Big(\frac{\pi}{2}r_{ii}^K\Big) - \rho_{ii}\bigg| \geq t\right\}\bigg) \leq 2p^2\exp\bigg(-\frac{nt^2}{\pi^2}\bigg).\]
with probability at least $1 - 2p^{-(C^2/\pi^2-2)}$,
\begin{subequations}
\label{EqnMatrixElts}
\begin{equation}
\max_{1\leq i\leq p} \bigg|\sin\Big(\frac{\pi}{2}\boldr_{ii}^K\Big) - \brho_{ii}\bigg| \leq C\sqrt{\frac{\log p}{n}}, \qquad \text{and}
\end{equation}
\begin{equation}
\max_{i\neq j}\bigg|\sin\Big(\frac{\pi}{2}\htauij\Big) - \sin\Big(\frac{\pi}{2}\brho^K_{ij}\Big)\bigg| \leq C\sqrt{\frac{\log p}{n}}.
\end{equation}
\end{subequations}

For the second term on the right-hand side of equation~\eqref{EqnChoco}, we have under model~\eqref{cellContam} that for any pair $i\neq j$,
\begin{equation}
\label{EqnPairMarginal}
(X_{ki}, X_{kj}) \stackrel{\text{i.i.d.}}\sim F_{ij} = (1-\gamma_{ij}) \Phi_{\bmu_{\{i, j\}}, \bSigma_{\{i, j\}}} + \gamma_{ij}H_{ij}, \qquad \forall 1 \le k \le n,
\end{equation}
where $\Phi_{\bmu_{\{i, j\}}, \bSigma_{\{i, j\}}} = N(\bmu_{\{i, j\}}, \bSigma_{\{i, j\}})$ is the marginal distribution of $(Y_{ki}, Y_{kj})$, $H_{ij}$ is a mixture of the distributions of $Y_{ki}, Y_{kj}, Z_{ki}$, and $Z_{kj}$, and $1 - \gamma_{ij} = (1-\epsilon_i)(1-\epsilon_j)$.

By Lemma~\ref{exptau}, we have $\brho^K_{ij} = \frac{2}{\pi}\sin^{-1}\brho_{ij} + R_{ij}$, where $|R_{ij}| \leq 12\gamma_{ij} + 17\gamma_{ij}^2$. Setting $R'_{ij} = \frac{\pi}{2}R_{ij}$, we then have
\begin{align*}
\left|\sin\Big(\frac{\pi}{2}\brho^K_{ij}\Big) - \brho_{ij}\right| & = \left|\sin\big(\sin^{-1}(\brho_{ij}) + R'_{ij}\big) - \brho_{ij}\right| \\
&= \left|\sin(\sin^{-1}(\brho_{ij}))\cos(R'_{ij}) + \cos(\sin^{-1}(\brho_{ij}))\sin(R'_{ij}) - \brho_{ij}\right| \\
&= \left|\brho_{ij}\cos(R'_{ij}) + \sqrt{1-\brho_{ij}^2}\sin(R'_{ij}) - \brho_{ij}\right| \\
& \le \left|\brho_{ij} \left(1-\cos(R'_{ij})\right)\right| + \left|\sqrt{1-\brho_{ij}^2} \sin(R'_{ij})\right| \\
&\leq \big[1-\cos(R'_{ij})\big] + \big|\sin(R'_{ij})\big|.
\end{align*}
Note that $\gamma_{ij} = \epsilon_i + \epsilon_j - \epsilon_i \epsilon_j \le 2\epsilon$, so
\begin{equation*}
|R'_{ij}| \leq \frac{\pi}{2}(12 \gamma_{ij} + 17 \gamma_{ij}^2) \leq \frac{\pi}{2} \left(12 \cdot 2\epsilon + 17 (2\epsilon)^2\right) = 12\pi \epsilon + 34\pi \epsilon^2.
\end{equation*}
In particular, this bound is less than 1 when $\epsilon \le 0.02$. Then using the fact that $|\sin(x) - x| \le \frac{|x|^3}{3!}$ and $|1-\cos(x)| \le \frac{x^2}{2!}$ for $|x| \le 1$, we conclude that
\begin{equation}
\label{EqnRDev}
\max_{1 \le i,j \le p} \left|\sin\left(\frac{\pi}{2} \brho^K_{ij}\right) - \brho_{ij}\right| \leq \max_{1 \le i,j \le p} \left[|R'_{ij}| + \frac{(R'_{ij})^2}{2} + \frac{|R'_{ij}|^3}{6}\right] \le 2\max_{1 \le i,j \le p} |R'_{ij}| \leq 26\pi \epsilon.
\end{equation}

Combining inequalities~\eqref{EqnMatrixElts} and~\eqref{EqnRDev} then proves the desired result.

%{\blue 
%Here is an alternative proof of Lemma~\ref{kendalltauconsistency}: \\
%By Lemma~\ref{exptau}, we have $\rho_{ij} = \sin(\frac{\pi}{2}(\rho^K_{ij}-R_{ij}))$, where $|R_{ij}|\leq 12\gamma_{ij} + 17\gamma_{ij}^2$.  Let $t = 4\pi\sqrt{(\log p)/n}$.  Since $\gamma_{ij} = \epsilon_i + \epsilon_j - \epsilon_i \epsilon_j \le 2\epsilon$, by assumption we have $|R_{ij}|\leq 12\gamma_{ij} + 17\gamma_{ij}^2 \leq 12(2\epsilon) + 17(2\epsilon)^2\leq t/\pi$.  Hence,
%\begin{align*}
%P\bigg(\bigg|\sin\Big(\frac{\pi}{2}\htauij\Big) - \rho_{ij}\bigg| > t\bigg) &= P\bigg(\bigg|\sin\Big(\frac{\pi}{2}\htauij\Big) - \sin\left(\frac{\pi}{2}(\rho^K_{ij}-R_{ij})\right)\bigg| > t\bigg) \\
%&\leq P\bigg(|\htauij-\rho^K_{ij} + R_{ij}| > \frac{2}{\pi}t\bigg) \\
%&\leq P\bigg(|\htauij-\rho^K_{ij}| > \frac{2}{\pi}t-|R_{ij}|\bigg) \\
%&\leq P\bigg(|\htauij-\rho^K_{ij}| > \frac{t}{\pi}\bigg) \\
%&\leq 2\exp\left(-\frac{nt^2}{4\pi^2}\right).
%\end{align*}
%It follows that
%\[P\bigg(\max_{1\leq i, j\leq p}\bigg|\sin\Big(\frac{\pi}{2}\htauij\Big) - \rho_{ij}\bigg| > t\bigg) \leq 2p^2\exp\left(-\frac{nt^2}{4\pi^2}\right).\]
%Plugging in $t = 4\pi\sqrt{(\log p)/n}$, we see that with probability at least $1 - 2p^{-2}$,
%\[\max_{1\leq i, j\leq p}\bigg|\sin\Big(\frac{\pi}{2}\htauij\Big) - \rho_{ij}\bigg| \leq 4\pi\sqrt{\frac{\log p}{n}}.\]
%}

\subsection{Proof of Lemma~\ref{MADconsistency}}
\label{AppLemMADConsist}

Under model~\eqref{cellContam}, we have the marginal distributions
\[X_{ki} \stackrel{\text{i.i.d.}}\sim F_i = (1-\epsilon_i)\Phi_{\mu_i, \sigma_i} + \epsilon_i H_i, \quad \forall 1 \le k \le n,\]
for each $1 \le i \le p$, where $\Phi_{\mu_i, \sigma_i} = N(\mu_i, \sigma_i^2)$ is the marginal distribution of $Y_{ki}$ and $H_i$ is the marginal distribution of $Z_{ki}$.

Let $d(F_i)$ and $d(\Phi_{\mu_i, \sigma_i})$ denote the population MADs corresponding to $F_i$ and $\Phi_{\mu_i, \sigma_i}$, respectively.  Since $\hsigma_i = [\Phi^{-1}(0.75)]^{-1}\hat{d}_i$ and $\sigma_i = [\Phi^{-1}(0.75)]^{-1}d(\Phi_{\mu_i, \sigma_i})$, with $\hat{d}_i$ defined as in equation~\eqref{EqnSmiley}, it suffices to bound the term $|\hat{d}_i - d(\Phi_{\mu_i, \sigma_i})|$, which we decompose as follows:
\[|\hat{d}_i - d(\Phi_{\mu_i, \sigma_i})| \leq |\hat{d}_i - d(F_i)| + |d(F_i) - d(\Phi_{\mu_i, \sigma_i})|.\]
By Lemma~\ref{MADterm1}, for $0<t<1$,
\begin{align*}
P\big(\max_{1\leq i\leq p} |\hat{d}_i - d(F_i)| > t\big) &\leq \sum_{i=1}^pP\big(|\hat{d}_i - d(F_i)| > t\big) \\
&\leq 6p\max_{1\leq i\leq p} \left\{\exp(-2nc^2(\sigma_i)t^2)\right\} \\
&= 6p\exp\left(- 2n\min_{1\leq i\leq p}c^2(\sigma_i)t^2\right).
\end{align*}
Let $t = \Phi^{-1}(0.75)C'\sqrt{\frac{\log p}{n}}<1$. With probability at least $1 - 6p^{-\{2[\Phi^{-1}(0.75)]^2C'^2\min_{1\leq i\leq p}c^2(\sigma_i)-1\}}$, we then have
\[\max_{1\leq i\leq p} |\hat{d}_i - d(F_i)| \leq \Phi^{-1}(0.75)C'\sqrt{\frac{\log(p)}{n}}.\]

On the other hand, by Lemma~\ref{MADterm2}, we have
\[\max_{1\leq i\leq p} |d(F_i) - d(\Phi_{\mu_i, \sigma_i})| \leq 4.8\max_{1\leq i\leq p}\sigma_i\epsilon_i \leq 4.8M_\sigma \epsilon.\]
Thus, with probability at least $1 - 6p^{-\{2[\Phi^{-1}(0.75)]^2C'^2\min_{1\leq i\leq p}c^2(\sigma_i)-1\}}$, 
\[\max_{1\leq i\leq p} |\hat{d}_i - d(\Phi_{\mu_i, \sigma_i})| \leq \Phi^{-1}(0.75)C'\sqrt{\frac{\log(p)}{n}} + 4.8M_\sigma\epsilon.\]
It follows that with the same probability,
\begin{align*}
\max_{1\leq i\leq p}|\hsigma_i - \sigma_i| &= [\Phi^{-1}(0.75)]^{-1}\max_{1\leq i\leq p} |\hat{d}_i - d(\Phi_{\mu_i, \sigma_i})| \leq C'\sqrt{\frac{\log(p)}{n}} + 7.2M_\sigma\epsilon.
\end{align*}

%%%%%%

\subsection{Proof of Lemma~\ref{spearmanrhoconsistency}}
\label{AppLemSpearman}

When $i = j$, we have $2\sin(\frac{\pi}{6}r^S_{ii}) = \rho_{ii} = 1$; hence, we only need to consider the case when $i\neq j$. First, note that
\begin{align}
\label{EqnChoco2}
\bigg|2\sin\Big(\frac{\pi}{6}\hrhosij\Big) - \brho_{ij}\bigg| &\leq 2\bigg|\sin\Big(\frac{\pi}{6}\hrhosij\Big) - \sin\Big(\frac{\pi}{6}E(\hrhosij)\Big)\bigg| + \bigg|2\sin\Big(\frac{\pi}{6}E(\hrhosij)\Big) - \brho_{ij}\bigg|,
\end{align}
where the expectation is taken with respect to the distribution under model \eqref{cellContam}.  By Lemma~\ref{rhodecomp}, we have $\hrhosij = \frac{n-2}{n+1}U_{ij} + \frac{3}{n+1}\htauij$, where $U_{ij}$ is a $U$-statistic with kernel bounded between $-3$ and 3, and $\htauij$ is the Kendall's tau correlation. Using the fact that $\sin(x)$ is 1-Lipschitz, we then have
\begin{align*}
P\bigg(2\bigg|\sin\Big(\frac{\pi}{6}\hrhosij\Big) - \sin\Big(\frac{\pi}{6}E(\hrhosij)\Big)\bigg| \ge t\bigg) &\leq P\bigg(|\hrhosij - E(\hrhosij)| \ge \frac{3t}{\pi}\bigg)\\
&= P\left(\left|\frac{n-2}{n+1}(U_{ij}-E(U_{ij})) + \frac{3}{n+1}(\htauij-\brho^K_{ij})\right| \geq \frac{3t}{\pi}\right) \\
&\leq P\left(|U_{ij}-E(U_{ij}))| + \frac{6}{n+1} \geq \frac{3t}{\pi}\right) \\
&\leq P\bigg(|U_{ij}-E(U_{ij})| \ge \frac{3t}{2\pi}\bigg), 
\end{align*}
where the last inequality follows from the choice $t = C \sqrt{\frac{\log p}{n}}$ and the fact that $\frac{6}{n+1} \le \frac{3t}{2\pi}$ when $n \ge \frac{16\pi^2}{C^2 \log p}$. Furthermore, Hoeffding's inequality implies
\begin{align*}
P\bigg(|U_{ij}-E(U_{ij})| \ge \frac{3t}{2\pi}\bigg) \leq 2\exp\left(-2 \left\lfloor \frac{n}{3}\right\rfloor \left(\frac{3t}{2\pi}\right)^2 \frac{1}{6^2}\right) \le 2 \exp\left(-\frac{nt^2}{32\pi^2}\right).
\end{align*}
Plugging in $t = C\sqrt{\frac{\log p}{n}}$ and using a union bound, we then have
\begin{equation}
\label{EqnRSbd}
P\bigg(\max_{1\leq i, j\leq p}2\bigg|\sin\Big(\frac{\pi}{6}\hrhosij\Big) - \sin\Big(\frac{\pi}{6}E(\hrhosij)\Big)\bigg| \ge C \sqrt{\frac{\log p}{n}}\bigg) \leq 2p^2\exp\bigg(-\frac{C^2 \log p}{32\pi^2}\bigg) = 2p^{-\left\{\frac{C^2}{32\pi^2}-2\right\}}.
\end{equation}
For the second term on the right-hand side of equation~\eqref{EqnChoco}, we have under model~\eqref{cellContam} that for any pair $i\neq j$, 
\[(X_{ki}, X_{kj}) \stackrel{\text{i.i.d.}}\sim F_{ij} = (1-\gamma_{ij}) \Phi_{\bmu_{\{i, j\}}, \bSigma_{\{i, j\}}} + \gamma_{ij}H_{ij}, \qquad \forall 1 \le k \le n,\]
where $\Phi_{\bmu_{\{i, j\}}, \bSigma_{\{i, j\}}} = N(\bmu_{\{i, j\}}, \bSigma_{\{i, j\}})$ is the marginal distribution of $(Y_{ki}, Y_{kj})$, $H_{ij}$ is a mixture of the distributions of $Y_{ki}, Y_{kj}, Z_{ki}$, and $Z_{kj}$, and $1 - \gamma_{ij} = (1-\epsilon_i)(1-\epsilon_j)$.

By Lemma~\ref{exprho}, we have $E(\hrhosij) = \frac{6}{\pi}\sin^{-1}\left(\frac{\rho_{ij}}{2}\right) + R_{ij}$, where $|R_{ij}| \leq 48\gamma_{ij} + 129\gamma_{ij}^2 + 88\gamma_{ij}^3 + \frac{12}{n+1}$. Setting $R'_{ij} = \frac{\pi}{6}R_{ij}$, we then have
\begin{align*}
\bigg|2\sin\Big(\frac{\pi}{6}E(\hrhosij)\Big) - \brho_{ij}\bigg| & = \left|2\sin\big(\sin^{-1}(\brho_{ij}/2) + R'_{ij}\big) - \brho_{ij}\right| \\
&= \left|2\sin(\sin^{-1}(\brho_{ij}/2))\cos(R'_{ij}) + 2\cos(\sin^{-1}(\brho_{ij}/2))\sin(R'_{ij}) - \brho_{ij}\right| \\
&= \left|\brho_{ij}\cos(R'_{ij}) + 2\sqrt{1-\brho_{ij}^2/4} \cdot \sin(R'_{ij}) - \brho_{ij}\right| \\
& \le \left|\brho_{ij} \left(1-\cos(R'_{ij})\right)\right| + 2\left|\sqrt{1-\brho_{ij}^2/4} \cdot \sin(R'_{ij})\right| \\
&\leq \big[1-\cos(R'_{ij})\big] + 2\big|\sin(R'_{ij})\big|.
\end{align*}
Note that $\gamma_{ij} = \epsilon_i + \epsilon_j - \epsilon_i \epsilon_j \le 2\epsilon$, so
\begin{align*}
|R'_{ij}| &\leq \frac{\pi}{6}\left(48\gamma_{ij} + 129\gamma_{ij}^2 + 88\gamma_{ij}^3 + \frac{12}{n+1}\right) \\
&\leq \frac{\pi}{6} \left(48 \cdot 2\epsilon + 129 (2\epsilon)^2 + 88(2\epsilon)^3 + \frac{12}{n+1}\right) \\
&\leq 16\pi \epsilon + 86\pi \epsilon^2 + 118\pi\epsilon^3 + \frac{2\pi}{n+1}.
\end{align*}
In particular, this bound is less than 1 when $\epsilon \le 0.01$ and $n\geq 15$. Then using the fact that $|\sin(x) - x| \le \frac{|x|^3}{3!}$ and $|\cos(x) - 1| \le \frac{x^2}{2!}$ for $|x| \le 1$, we conclude that
\begin{align*}
\max_{1 \le i,j \le p} \left|2\sin\left(\frac{\pi}{6} E(\hrhosij)\right) - \brho_{ij}\right| &\leq \max_{1 \le i,j \le p} \left[2|R'_{ij}| + \frac{(R'_{ij})^2}{2} + \frac{|R'_{ij}|^3}{3}\right] \\
&\le 3\max_{1 \le i,j \le p} |R'_{ij}| \\
&\leq 48\pi \epsilon + 258\pi \epsilon^2 + 354\pi \epsilon^3 + \frac{6\pi}{n+1} \\
& \le 51 \pi \epsilon + \frac{3C}{2} \sqrt{\frac{\log p}{n}},
\end{align*}
where the final inequality uses the assumption $n \ge \frac{16 \pi^2}{C^2 \log p}$ once more. Combining this bound with inequality~\eqref{EqnRSbd} implies the desired result.

\section{Lemmas for MAD concentration}
\label{AppMAD}

In this Appendix, we prove several lemmas that are needed in deriving consistency of the MAD estimator. We begin with some results concerning the concentration of sample medians from an arbitrary distribution. A version of Lemmas~\ref{medFact2} and~\ref{MADterm1.0} is also contained in \cite{Serfling2009}.

\begin{Lemma}\label{medFact}
Let $X_1, \ldots, X_n$ be a random sample from a distribution with cdf $F$, and let $\hat{m}$ be the sample median.  If $\hat{m}<c$, then $|\{X_i: X_i\leq c\}|\geq\frac{n}{2}$.  If $\hat{m}>c$, then $|\{X_i: X_i\leq c\}|\leq\frac{n}{2}$.
\end{Lemma}
\begin{proof}
This result follows easily from the definition of the sample median.
\end{proof}

\begin{Lemma}\label{medFact2}
Let $X_1, \ldots, X_n$ be a random sample from a distribution $F$. Let $m$ be the population median and let $\hat{m}$ be the sample median.  Then 
\[P\bigg(|\hat{m}-m|>\frac{t}{2}\bigg) \leq 2\exp(-2nb^2(t)),\]
where $b(t) = \min\big\{F(m+\frac{t}{2})-\half, \half-F(m-\frac{t}{2})\big\}$.
\end{Lemma}
\begin{proof}
By Lemma~\ref{medFact},
\begin{align}
\label{leftm}
P\bigg(\hat{m}>m+\frac{t}{2}\bigg) & \le P\bigg(\Big|\Big\{X_i: X_i\leq m+\frac{t}{2}\Big\}\Big|\leq\frac{n}{2}\bigg) \notag \\
&= P\bigg(\sum_{i=1}^n\1\left\{X_i\leq m+\frac{t}{2}\right\}\leq\frac{n}{2}\bigg) \notag \\
&= P\bigg(\sum_{i=1}^n(Y_i-EY_i)\leq\frac{n}{2}-np_1\bigg) \notag \\
&= \exp\bigg[-2n\Big(p_1-\frac{1}{2}\Big)^2\bigg],
\end{align}
where $Y_i = \1\left\{X_i\leq m+\frac{t}{2}\right\}$ and $p_1 = F(m+\frac{t}{2})$, and the last inequality follows from Hoeffding's inequality. Similarly, we have
\begin{align}
\label{rightm}
P\bigg(\hat{m}<m-\frac{t}{2}\bigg) & \le P\bigg(\Big|\Big\{X_i: X_i\leq m-\frac{t}{2}\Big\}\Big|\geq\frac{n}{2}\bigg) \nonumber\\
&= P\bigg(\sum_{i=1}^n\1(X_i\leq m-\frac{t}{2})\geq\frac{n}{2}\bigg) \nonumber\\
&= P\bigg(\sum_{i=1}^n(Z_i-EZ_i)\geq\frac{n}{2}-np_2\bigg) \nonumber\\
&\leq \exp\bigg[-2n\Big(p_2-\frac{1}{2}\Big)^2\bigg],
\end{align}
where $Z_i =  \1\left\{X_i\leq m-\frac{t}{2}\right\}$ and $p_2 = F(m-\frac{t}{2})$. Combining expressions~\eqref{leftm} and \eqref{rightm}, we then obtain 
\[P\bigg(|\hat{m}-m|>\frac{t}{2}\bigg) \leq \exp\bigg[-2n\Big(p_1-\frac{1}{2}\Big)^2\bigg] + \exp\bigg[-2n\Big(p_2-\frac{1}{2}\Big)^2\bigg] \leq 2\exp(-2nb^2(t)).\]
\end{proof}

\begin{Lemma}\label{MADterm1.0}
Let $X_1, \ldots, X_n$ be a random sample from a distribution with cdf $F$.  Let $m$ and $d$ denote the population median and MAD, respectively, and let $\hat{m}$ and $\hat{d}$ denote the sample median and MAD. Let $G$ be the distribution of $|X_i-m|$.  Then
\begin{equation}
P(|\hat{d}-d|>t) \leq 6\exp(-2na^2(t)), \label{MADexptail}
\end{equation}
where
\[a(t) = \min\Bigg\{F\left(m+\frac{t}{2}\right)-\half, \; \half-F\left(m-\frac{t}{2}\right), \; G\left(d+\frac{t}{2}\right)-\half, \; \half-G\left(d-\frac{t}{2}\right)\Bigg\}.\]
\end{Lemma}

\begin{proof}
Let $W_i = |X_i-\hat{m}|$. By the definition of the sample MAD, Lemma~\ref{medFact} gives
\begin{align*}
P(\hat{d} > d + t) & \le P\bigg(|\{W_i: W_i\leq d+t\}|\leq\frac{n}{2}\bigg) \\
&= P\bigg(|\{X_i: |X_i-\hat{m}| \leq d+t\}|\leq\frac{n}{2}\bigg) \\
&\leq P\bigg(|\{X_i: |X_i-\hat{m}| \leq d+t\}|\leq\frac{n}{2}, \text{ and } |\hat{m}-m|\leq\frac{t}{2}\bigg) + P\bigg(|\hat{m}-m|>\frac{t}{2}\bigg) \\
&\leq P\bigg(\Big|\Big\{X_i: |X_i-m| \leq d+\frac{t}{2}\Big\}\Big|\leq\frac{n}{2}\bigg) + P\bigg(|\hat{m}-m|>\frac{t}{2}\bigg) \\
&= P\bigg(\sum_{i=1}^n\1\left\{|X_i-m| \leq d+\frac{t}{2}\right\}\leq\frac{n}{2}\bigg) + P\bigg(|\hat{m}-m|>\frac{t}{2}\bigg) \\
&= P\bigg(\sum_{i=1}^n(Y_i-EY_i)\leq\frac{n}{2}-np_3\bigg) + P\bigg(|\hat{m}-m|>\frac{t}{2}\bigg),
\end{align*}
where $Y_i = \1\left\{|X_i-m| \leq d+\frac{t}{2}\right\}$ and $p_3 = G(d+\frac{t}{2})$. Then by Hoeffding's inequality and Lemma~\ref{medFact2}, the last quantity is bounded by
\begin{equation}
\label{leftd}
\exp\bigg[-2n\Big(p_3-\frac{1}{2}\Big)^2\bigg] + 2\exp(-2nb^2(t)).
\end{equation}
Similarly,
\begin{align*}
P(\hat{d} < d - t) & \le P\bigg(|\{W_i: W_i\leq d-t\}|\geq\frac{n}{2}\bigg) \\
&= P\bigg(|\{X_i: |X_i-\hat{m}| \leq d-t\}|\geq\frac{n}{2}\bigg) \\
&\leq P\bigg(|\{X_i: |X_i-\hat{m}| \leq d-t\}|\geq\frac{n}{2}, \text{ and } |\hat{m}-m|\leq\frac{t}{2}\bigg) + P\bigg(|\hat{m}-m|>\frac{t}{2}\bigg) \\
&\leq P\bigg(\Big|\Big\{X_i: |X_i-m| \leq d-\frac{t}{2}\Big\}\Big|\geq\frac{n}{2}\bigg) + P\bigg(|\hat{m}-m|>\frac{t}{2}\bigg) \\
&= P\bigg(\sum_{i=1}^n\1\left\{|X_i-m| \leq d-\frac{t}{2}\right\}\geq\frac{n}{2}\bigg) + P\bigg(|\hat{m}-m|>\frac{t}{2}\bigg) \\
&= P\bigg(\sum_{i=1}^n(Z_i-EZ_i)\geq\frac{n}{2}-np_4\bigg) + P\bigg(|\hat{m}-m|>\frac{t}{2}\bigg),
\end{align*}
where $Z_i = \1\left\{|X_i-m| \leq d-\frac{t}{2}\right\}$ and $p_4 = G(d-\frac{t}{2})$. By Hoeffding's inequality and Lemma~\ref{medFact2}, the last quantity is upper-bounded by
\begin{equation}
\label{rightd}
\exp\bigg[-2n\Big(p_4-\frac{1}{2}\Big)^2\bigg] + 2\exp(-2nb^2(t)).
\end{equation}
Combining expressions~\eqref{leftd} and~\eqref{rightd} then yields
\[P(|\hat{d}-d|>t) \leq 4\exp(-2nb^2(t)) + \exp\bigg[-2n\Big(p_3-\frac{1}{2}\Big)^2\bigg] + \exp\bigg[-2n\Big(p_4-\frac{1}{2}\Big)^2\bigg] \leq 6\exp(-2na^2(t)).\]
\end{proof}

Next, we prove two population-level lemmas for the $\epsilon$-contamination model. As remarked in the introduction, we use the notation $F^{-1}(c) = \inf\{x: F(x) \ge c\}$, which is defined even if the cdf $F$ is not surjective on the interval $[0,1]$. Note that Lemmas~\ref{epsContamQuantilebound} and~\ref{MADterm2} do not impose any conditions on the contaminating distribution $H$.

\begin{Lemma}\label{epsContamQuantilebound}
Let $F = (1-\epsilon)\Phi_{\mu, \sigma} + \epsilon H$, where $\Phi_{\mu, \sigma}$ denotes the $N(\mu, \sigma^2)$ distribution and $H$ is an arbitrary distribution. Let $\Phi := \Phi_{0, 1}$ be the standard normal cdf and suppose that $0\leq\epsilon<1$. Then
\begin{equation}
\mu + \Phi^{-1}\Big(\frac{c-\epsilon}{1-\epsilon}\Big)\sigma = \Phi_{\mu, \sigma}^{-1}\left(\frac{c-\epsilon}{1-\epsilon}\right) \leq F^{-1}(c) \leq \Phi_{\mu, \sigma}^{-1} \left(\frac{c}{1-\epsilon}\right) = \mu + \Phi^{-1}\Big(\frac{c}{1-\epsilon}\Big)\sigma. \label{bound}
\end{equation}
\end{Lemma}
\begin{proof}
Let $F = (1-\epsilon)\Phi_{\mu, \sigma} + \epsilon H$.  Then
\begin{align}
F\Big(\Phi_{\mu, \sigma}^{-1}\Big(\frac{c}{1-\epsilon}\Big)\Big) &= (1-\epsilon)\Phi_{\mu, \sigma}\Big(\Phi_{\mu, \sigma}^{-1}\Big(\frac{c}{1-\epsilon}\Big)\Big) + \epsilon H\Big(\Phi_{\mu, \sigma}^{-1}\Big(\frac{c}{1-\epsilon}\Big)\Big) \nonumber\\
&\geq (1-\epsilon) \cdot \frac{c}{1-\epsilon} = c, \label{side1}
\end{align}
where by a slight abuse of notation, we use $F$ and $H$ to denote the cdfs of the corresponding distributions. In addition,
\begin{align}
1-F\Big(\Phi_{\mu, \sigma}^{-1}\Big(\frac{c-\epsilon}{1-\epsilon}\Big)\Big) &= (1-\epsilon)\bigg[1-\Phi_{\mu, \sigma}\Big(\Phi_{\mu, \sigma}^{-1}\Big(\frac{c-\epsilon}{1-\epsilon}\Big)\Big)\bigg] + \epsilon\bigg[1- H\Big(\Phi_{\mu, \sigma}^{-1}\Big(\frac{c-\epsilon}{1-\epsilon}\Big)\Big)\bigg] \nonumber\\
&\geq (1-\epsilon) \left(1 - \frac{c-\epsilon}{1-\epsilon}\right) = 1-c. \label{side2}
\end{align}
Combining equations~\eqref{side1} and~\eqref{side2}, and using the facts that $F$ is monotonically increasing, we then obtain the desired bound~\eqref{bound}. Note that the outer equalities hold since $\Phi_{\mu, \sigma}^{-1}(x) = \mu + \Phi^{-1}(x)\sigma$.
\end{proof}

\begin{Lemma}\label{MADterm2}
Let $F = (1-\epsilon)\Phi_{\mu, \sigma} + \epsilon H$, where $\Phi_{\mu, \sigma}$ denotes the $N(\mu, \sigma^2)$ distribution and $H$ is an arbitrary distribution. Suppose $0\leq\epsilon\leq\frac{1}{16}$. Let $d(F)$ and $d(\Phi_{\mu, \sigma})$ denote the population MADs corresponding to $F$ and $\Phi_{\mu, \sigma}$, respectively.  Then
\begin{align*}
|d(F)-d(\Phi_{\mu, \sigma})| \leq 4.8\sigma\epsilon.
\end{align*}
\end{Lemma}

\begin{proof}

By an abuse of notation, we also use $F$ to denote the cdf of the contaminated distribution. Then $F^{-1}$ is the quantile function. Note in particular that the following statements hold, where $X \sim F$, as an easy consequence of the definition of $F^{-1}$:
\begin{itemize}
\item[(i)] $d(F) \le a$ if $P(|X - F^{-1}(0.5)| \le a) \ge 0.5$,
\item[(ii)] $d(F) > a$ if $P(|X - F^{-1}(0.5)| \le a) < 0.5$.
\end{itemize}
Furthermore, we may write
\begin{align*}
P(|X - F^{-1}(0.5)| \le a) &\ge (1-\epsilon) \cdot P(|Z - F^{-1}(0.5)| \le a) \\
&= (1-\epsilon) \left\{\Phi_{\mu, \sigma} \left(F^{-1}(0.5) + a\right) - \Phi_{\mu, \sigma} \left(F^{-1}(0.5) - a\right)\right\},
\end{align*}
where $Z \sim N(\mu, \sigma^2)$. By Lemma~\ref{epsContamQuantilebound}, the last expression is further lower-bounded by
\begin{equation*}
(1-\epsilon)\left\{\Phi_{\mu, \sigma} \left(\Phi_{\mu, \sigma}^{-1}\left(\frac{0.5-\epsilon}{1-\epsilon}\right) + a\right) - \Phi_{\mu, \sigma} \left(\Phi_{\mu, \sigma}^{-1}\left(\frac{0.5}{1-\epsilon}\right) - a\right)\right\}.
\end{equation*}
We will take
\begin{equation*}
a = \Phi_{\mu, \sigma}^{-1}\left(\frac{0.75}{1-\epsilon}\right) - \Phi_{\mu, \sigma}^{-1}\left(\frac{0.5 - \epsilon}{1-\epsilon}\right) = \Phi_{\mu, \sigma}^{-1}\left(\frac{0.5}{1-\epsilon}\right) - \Phi_{\mu, \sigma}^{-1}\left(\frac{0.25 - \epsilon}{1-\epsilon}\right),
\end{equation*}
where the second inequality comes from the fact that $\Phi_{\mu, \sigma}^{-1}(b) = -\Phi_{\mu, \sigma}^{-1}(1-b)$. Then the lower bound becomes
\begin{equation*}
(1-\epsilon) \left(\frac{0.75}{1-\epsilon} - \frac{0.25 - \epsilon}{1-\epsilon} \right) \ge 0.5.
\end{equation*}
Putting the bounds together, we have
\begin{equation*}
P(|X - F^{-1}(0.5)| \le a) \ge 0.5,
\end{equation*}
so by the implication (i) above, it follows that
\begin{equation}
\label{EqnDupper}
d(F) \le \Phi_{\mu, \sigma}^{-1}\left(\frac{0.75}{1-\epsilon}\right) - \Phi_{\mu, \sigma}^{-1}\left(\frac{0.5 - \epsilon}{1-\epsilon}\right).
\end{equation}
Similarly, we may derive a lower bound on $d(F)$ by writing
\begin{equation*}
P(|X - F^{-1}(0.5)| > a) \ge (1-\epsilon) \cdot P(|Z - F^{-1}(0.5)| > a),
\end{equation*}
where $Z \sim N(\mu, \sigma^2)$. Furthermore,
\begin{align*}
P(|Z - F^{-1}(0.5)| \le a) & = \Phi_{\mu, \sigma}\left(F^{-1}(0.5) + a\right) - \Phi_{\mu, \sigma}\left(F^{-1}(0.5) - a\right) \\
& \le \Phi_{\mu, \sigma}\left(\Phi_{\mu, \sigma}^{-1}\left(\frac{0.5}{1-\epsilon}\right) + a\right) - \Phi_{\mu, \sigma}\left(\Phi_{\mu, \sigma}^{-1}\left(\frac{0.5 - \epsilon}{1-\epsilon}\right) - a\right),
\end{align*}
using Lemma~\ref{epsContamQuantilebound}. Taking
\begin{equation*}
a = \Phi_{\mu, \sigma}^{-1}\left(\frac{0.75 - 2\epsilon}{1-2\epsilon}\right) - \Phi_{\mu, \sigma}^{-1}\left(\frac{0.5}{1-\epsilon}\right) = \Phi_{\mu, \sigma}^{-1}\left(\frac{0.5 - \epsilon}{1-\epsilon}\right) - \Phi_{\mu, \sigma}^{-1}\left(\frac{0.25}{1-2\epsilon}\right),
\end{equation*}
we then have the bound
\begin{equation*}
P(|Z - F^{-1}(0.5)| \le a) \le \frac{0.75 - 2\epsilon}{1-2\epsilon} - \frac{0.25}{1-2\epsilon} = \frac{0.5 - 2\epsilon}{1-2\epsilon},
\end{equation*}
implying that
\begin{equation*}
P(|X - F^{-1}(0.5)| > a) \ge (1-\epsilon) \cdot \left(1 - \frac{0.5 - 2\epsilon}{1-2\epsilon}\right) > 0.5.
\end{equation*}
It follows that
\begin{equation*}
P(|X - F^{-1}(0.5)| \le a) < 0.5,
\end{equation*}
so by implication (ii) above,
\begin{equation}
\label{EqnDlower}
d(F) > \Phi^{-1}_{\mu, \sigma}\left(\frac{0.75 - 2\epsilon}{1-2\epsilon}\right) - \Phi_{\mu, \sigma}^{-1}\left(\frac{0.5}{1-\epsilon}\right).
\end{equation}

%Note that we always have {\red Note: check if the following still holds if $H$ is not continuous.}
%\begin{align*}
%d(F) &\geq\min\{F^{-1}(0.75)-F^{-1}(0.5), F^{-1}(0.5)-F^{-1}(0.25)\}, \quad \text{and} \\
%d(F) &\leq \max\{F^{-1}(0.75)-F^{-1}(0.5), F^{-1}(0.5)-F^{-1}(0.25)\}.
%\end{align*}
%For instance, the first inequality follows from the fact that $d(F)$ is the median of the distribution of $|X_i - F^{-1}(0.5)|$, and we know that $\frac{1}{4}$ of the mass of the original distribution lies in each of the intervals $[F^{-1}(0.25), F^{-1}(0.5)]$ and $[F^{-1}(0.5), F^{-1}(0.75)]$. {\red not necessarily have $1/4$ mass on both sides if $F$ is asymmetric?}

%By Lemma~\ref{epsContamQuantilebound}, we then have
%\begin{align}
%d(F) &\geq \min\{F^{-1}(0.75)-F^{-1}(0.5), F^{-1}(0.5)-F^{-1}(0.25)\} \nonumber\\
%&\geq \sigma\min\bigg\{\Phi^{-1}\bigg(\frac{0.75-\epsilon}{1-\epsilon}\bigg)-\Phi^{-1}\bigg(\frac{0.5}{1-\epsilon}\bigg), \Phi^{-1}\bigg(\frac{0.5-\epsilon}{1-\epsilon}\bigg)-\Phi^{-1}\bigg(\frac{0.25}{1-\epsilon}\bigg)\bigg\} \label{MADterm2left}
%\end{align}
%and
%\begin{align}
%d(F) &\leq \max\{F^{-1}(0.75)-F^{-1}(0.5), F^{-1}(0.5)-F^{-1}(0.25)\} \nonumber\\
%&\leq \sigma\max\bigg\{\Phi^{-1}\bigg(\frac{0.75}{1-\epsilon}\bigg)-\Phi^{-1}\bigg(\frac{0.5-\epsilon}{1-\epsilon}\bigg), \Phi^{-1}\bigg(\frac{0.5}{1-\epsilon}\bigg)-\Phi^{-1}\bigg(\frac{0.25-\epsilon}{1-\epsilon}\bigg)\bigg\}. \label{MADterm2right}
%\end{align}

Using the fact that $d(\Phi_{\mu, \sigma}) = \Phi^{-1}_{\mu, \sigma}(0.75)$ and $\Phi_{\mu, \sigma}^{-1}(0.5) = 0$, inequality~\eqref{EqnDupper} implies that
\begin{align*}
d(F) - d(\Phi_{\mu, \sigma}) & \le \left\{\Phi_{\mu, \sigma}^{-1}\left(\frac{0.75}{1-\epsilon}\right) - \Phi_{\mu, \sigma}^{-1}(0.75)\right\} + \left\{\Phi_{\mu, \sigma}^{-1}(0.5) - \Phi_{\mu, \sigma}^{-1}\left(\frac{0.5 - \epsilon}{1-\epsilon}\right)\right\} \\
& \le 3.6 \sigma \left\{\left(\frac{0.75}{1-\epsilon} - 0.75\right) + \left(0.5 - \frac{0.5 - \epsilon}{1-\epsilon}\right)\right\} \\
& = 3.6 \sigma \cdot \frac{1.25\epsilon}{1-\epsilon} \\
& \le 4.8 \sigma \epsilon, 
\end{align*}
where the second inequality comes from Lemma~\ref{quantile} and the observation $\Phi^{-1}_{\mu, \sigma}(x) = \mu + \sigma \Phi^{-1}_{0, 1}(x)$, along with the assumption $\epsilon \le \frac{1}{16}$. Similarly, inequality~\eqref{EqnDlower} implies that
\begin{align*}
d(F) - d(\Phi_{\mu, \sigma}) & \ge \left\{\Phi_{\mu, \sigma}^{-1}\left(\frac{0.75 - 2\epsilon}{1-2\epsilon}\right) - \Phi_{\mu, \sigma}^{-1}(0.75)\right\} + \left\{\Phi_{\mu, \sigma}^{-1}(0.5) - \Phi_{\mu, \sigma}^{-1}\left(\frac{0.5}{1-\epsilon}\right) \right\} \\
& \ge -3.6 \sigma \left\{\left(0.75 - \frac{0.75 - 2\epsilon}{1-2\epsilon}\right) + \left(\frac{0.5}{1-\epsilon} - 0.5\right)\right\} \\
& = -3.6 \sigma \left(\frac{0.5\epsilon}{1-2\epsilon} + \frac{0.5\epsilon}{1-\epsilon}\right) \\
& \ge -3.98 \sigma \epsilon.
\end{align*}
Thus, we have the desired result.
\end{proof}

We conclude with the main lemma of this section, which establishes the consistency of the sample MAD to its population-level version.

\begin{Lemma}\label{MADterm1}
Let $X_1, \ldots, X_n$ be a random sample from $F = (1-\epsilon)\Phi_{\mu, \sigma}+\epsilon H$, where $0\leq\epsilon\leq\frac{1}{16}$, $\Phi_{\mu, \sigma}$ denotes the $N(\mu, \sigma^2)$ distribution, and $H$ is an arbitrary distribution.  Let $d := d(F)$ be the population MAD corresponding to $F$, and let $\hat{d}$ be the sample MAD.  Then for $0<t<1$, we have
\begin{equation}
P(|\hat{d} - d|>t) \leq 6\exp(-2nc^2(\sigma)t^2), \label{MADtail}
\end{equation}
where $c(\sigma) = \frac{15}{64\sqrt{2\pi}\sigma}\exp\left(-\frac{(1.1\sigma+0.5)^2}{2\sigma^2}\right)$.
\end{Lemma}

\begin{proof}
By Lemma~\ref{MADterm1.0}, it suffices to show that
\[a(t) \geq c(\sigma)t,\]
for the $\epsilon$-contaminated distribution, with $a(t)$ as defined in the lemma. With an abuse of notation, let $F, \Phi_{\mu, \sigma}$, and $H$ denote the cdfs of the respective distributions. Let
\begin{equation*}
G(c) = P(|X_i-m| \leq c),
\end{equation*}
where $m$ denotes the median of the contaminated distribution. Note that by the definition of the median, we have $F(m) \ge \frac{1}{2}$ and $G(d) \ge \frac{1}{2}$. Define
\begin{align*}
b_1 &= F\Big(m+\frac{t}{2}\Big)-\half \ge F\Big(m+\frac{t}{2}\Big) - F(m), \\
b_2 &= \half-F\Big(m-\frac{t}{2}\Big) \ge F\left(m - \frac{t}{4}\right)-F\Big(m-\frac{t}{2}\Big), \\
b_3 &= G\Big(d+\frac{t}{2}\Big)-\half \ge G\Big(d+\frac{t}{2}\Big) - G(d), \qquad \text{and} \\
b_4 &= \half-G\Big(d-\frac{t}{2}\Big) \ge G\left(d - \frac{t}{4}\right) - G\Big(d-\frac{t}{2}\Big),
\end{align*}
where we have used the fact that $F\left(m - \frac{t}{4}\right) < \frac{1}{2}$ and $G\left(d - \frac{t}{4}\right) < \frac{1}{2}$ in the second and fourth inequalities. Then $a(t) = \min\{b_1, b_2, b_3, b_4\}$.

Note that
\begin{align*}
b_1 &\ge (1-\epsilon)\left(\Phi_{\mu, \sigma}\left(m + \frac{t}{2}\right) - \Phi_{\mu, \sigma}(m)\right) + \epsilon\left(H\left(m + \frac{t}{2}\right) - H(m)\right) \\
&\ge (1-\epsilon)\left(\Phi_{\mu, \sigma}\left(m + \frac{t}{2}\right) - \Phi_{\mu, \sigma}(m)\right).
\end{align*}
Similarly, we can check that
\begin{align*}
b_2 & \ge (1-\epsilon)\left(\Phi_{\mu, \sigma}\left(m - \frac{t}{4}\right) - \Phi_{\mu, \sigma}\left(m - \frac{t}{2}\right)\right), \\
b_3 & \ge (1-\epsilon) \left(G_\Phi\left(d + \frac{t}{2}\right) - G_\Phi(d)\right), \qquad \text{and} \\
b_4 & \ge (1-\epsilon) \left(G_\Phi\left(d - \frac{t}{4}\right) - G_\Phi\left(d + \frac{t}{2}\right)\right),
\end{align*}
where $G_\Phi(c) := \Phi_{\mu, \sigma}(m+c) - \Phi_{\mu, \sigma}(m-c)$. By the mean value theorem, we have $c_1, c_2, c_3$, and $c_4$ such that
\begin{align*}
b_1 & \ge (1-\epsilon) \Phi_{\mu, \sigma}'(c_1)\frac{t}{2}, & m \le c_1 \le m+\frac{t}{2}, \\
b_2 & \ge (1-\epsilon) \Phi_{\mu, \sigma}'(c_2)\frac{t}{4}, & m-\frac{t}{2} \le c_2 \le m - \frac{t}{4}, \\
b_3 & \ge (1-\epsilon) G_\Phi'(c_3)\frac{t}{2} = (1-\epsilon) \left(\Phi_{\mu, \sigma}'(m+c_3) + \Phi_{\mu, \sigma}'(m-c_3)\right) \frac{t}{2}, & d \le c_3 \le d+\frac{t}{2}, \\
b_4 & \ge (1-\epsilon) G_\Phi'(c_4)\frac{t}{4} = (1-\epsilon) \left(\Phi_{\mu, \sigma}'(m+c_4) + \Phi_{\mu, \sigma}'(m-c_4)\right)\frac{t}{4}, &d-\frac{t}{2} \le c_4 \le d - \frac{t}{4}.
\end{align*}
Note in particular that
\begin{equation*}
c_1, \; c_2, \; m+c_3, \; m-c_3, \; m+c_4, \; m-c_4 \in \left[m - d - \frac{t}{2}, \; m + d + \frac{t}{2}\right].
\end{equation*}
%\begin{multline}
%\min\big\{\Phi_\sigma'(c_1), \; \Phi_\sigma'(c_2), \; \Phi_\sigma'(m+c_3) + \Phi_\sigma'(m-c_3), \; \Phi_\sigma'(m+c_4) + \Phi_\sigma'(m-c_4)\big\} \\
%\geq \min\bigg\{\Phi_\sigma'(c): m-d-\frac{t}{2}\leq c\leq m+d+\frac{t}{2}\bigg\}. \label{lb}
%\end{multline}
%Thus, it suffices to establish a lower bound for the right-hand expression in inequality~\eqref{lb}.

Let $d(\Phi_{\mu, \sigma}) = \Phi^{-1}(0.75)\sigma$ be the MAD estimator corresponding to $\Phi_{\mu, \sigma}$.  By Lemma~\ref{epsContamQuantilebound}, for $0\leq\epsilon\leq\frac{1}{16}$, the median $m = F^{-1}(0.5)$ satisfies
\[\mu + \Phi^{-1}\bigg(\frac{7}{15}\bigg)\sigma \leq \mu + \Phi^{-1}\Big(\frac{1-2\epsilon}{2-2\epsilon}\Big)\sigma \leq m \leq \mu + \Phi^{-1}\Big(\frac{1}{2-2\epsilon}\Big)\sigma \leq \mu + \Phi^{-1}\bigg(\frac{8}{15}\bigg)\sigma.\]
In addition, Lemma~\ref{MADterm2} implies that for $0\leq\epsilon\leq\frac{1}{16}$, we have
\[d \leq d(\Phi_{\mu, \sigma}) + 4.8\sigma\epsilon \leq \Phi^{-1}(0.75)\sigma + 0.3\sigma \leq\sigma .\]
Therefore, for $c\in[m-d-\frac{t}{2}, m+d+\frac{t}{2}]$ and $0<t<1$, we have
\begin{align*}
c &\geq m - d - \frac{t}{2} \geq \mu + \Phi^{-1}\bigg(\frac{7}{15}\bigg)\sigma - \sigma - 0.5 \geq \mu -1.1\sigma - 0.5, \quad \text{and} \\
c &\leq m + d + \frac{t}{2} \leq \mu + \Phi^{-1}\bigg(\frac{8}{15}\bigg)\sigma + \sigma + 0.5 \leq \mu + 1.1\sigma + 0.5.
\end{align*}
Hence,
\begin{align*}
\min\bigg\{\Phi_{\mu, \sigma}'(c): m-d-\frac{t}{2}\leq c\leq m+d+\frac{t}{2}\bigg\} &\geq \min\{\Phi_{\mu, \sigma}'(c): |c-\mu| \leq 1.1\sigma + 0.5\} \\
%&= \min\{(1-\epsilon)\Phi_\sigma'(c) + \epsilon H'(c): |c| \leq 1.1\sigma + 0.5\} \\
%&\geq \frac{15}{16}\min\{\Phi_\sigma'(c): |c| \leq 1.1\sigma + 0.5\} \\
&= \frac{1}{\sqrt{2\pi}\sigma}\exp\bigg(-\frac{(1.1\sigma+0.5)^2}{2\sigma^2}\bigg). 
\end{align*}
It follows that
\begin{align*}
a(t) = \min\{b_1, b_2, b_3, b_4\} &\geq (1-\epsilon) \cdot \frac{1}{\sqrt{2\pi}\sigma}\exp\bigg(-\frac{(1.1\sigma+0.5)^2}{2\sigma^2}\bigg)\frac{t}{4} \\
&\geq \frac{15}{16 \sqrt{2\pi} \sigma} \exp\bigg(-\frac{(1.1\sigma+0.5)^2}{2\sigma^2}\bigg)\frac{t}{4} = c(\sigma)t.
\end{align*}
\end{proof}

%%%%%%%%%%%%%%%%%%%%%%%%%%%%%%%%%%%%%%%%%%%%%%%%%%%%%%%%%%%%%%%%%%%%%%%%%%%%%%
\section{Auxiliary lemmas}
\label{AppAux}

We begin with a lemma describing the behavior of the mean of the Kendall's tau statistic under a contaminated normal distribution. Note that the statement of the lemma does not depend on the variances of the uncontaminated marginals, or the contaminating distribution $H$.

\begin{Lemma}\label{exptau}
Let $(X_{k1}, X_{k2})$, for $k = 1, \ldots, n$, be a random sample from
\[F = (1-\gamma)\Phi_\rho + \gamma H,\]
where $\Phi_\rho$ is a bivariate normal distribution with correlation $\rho$ and $H$ is an arbitrary bivariate distribution.  Let $\rho^K = E_F(r^K)$, where $r^K$ is Kendall's tau statistic.  Then
\begin{equation*}
\rho^K = \frac{2}{\pi}\sin^{-1}(\rho) + R, \label{exptauequal}
\end{equation*}
where $|R| \leq 12\gamma + 17\gamma^2$.
\end{Lemma}
\begin{proof}
Define $a(X) = \1(X>0)$, and let $\signbar(X) = 2a(X) - 1$. In particular,
\begin{equation*}
\sign(X) = \1(X>0) - \1(X<0) = 2a(X) - 1 - \1(X=0) = \signbar(X) - \1(X=0).
\end{equation*}
We may rewrite $\rho^K$ as
\begin{align*}
\rho^K &= E\left[\sign(X_{11}-X_{21})\sign(X_{12}-X_{22})\right] \\
&= E\left[\signbar(X_{11} - X_{21}) \signbar(X_{12} - X_{22})\right] - E[\1(X_{11} = X_{21})\signbar(X_{12} - X_{22})] \\
&\qquad - E\left[\signbar(X_{11} - X_{21}) \1(X_{12} = X_{22})\right] + E\left[\1(X_{11} = X_{21})\1(X_{12} = X_{22})\right] \\
& \defn A + B + C + D.
\end{align*}
In particular,
\begin{equation}
\label{Bbound}
|B| = \left|E[\1(X_{11} = X_{21}) \signbar(X_{12} - X_{22})]\right| \le E[\1(X_{11} = X_{21})] = P(X_{11} = X_{21}),
\end{equation}
using the fact that $|\signbar(X)| = 1$. Furthermore, we have
\begin{equation*}
P(X_{11} = X_{21}) \le \gamma^2,
\end{equation*}
since the normal distribution is absolutely continuous, so we can only have $P(X_{11} = X_{21})$ with positive probability when both $X_1$ and $X_2$ are drawn from the contaminating distribution. Similarly,
\begin{equation}
\label{Cbound}
|C| = \left| E[\signbar(X_{11} - X_{21}) \1(X_{12} = X_{22})] \right| \le E[\1(X_{12} = X_{22})] = P(X_{12} = X_{22}) \le \gamma^2.
\end{equation}
We also have
\begin{equation}
\label{Dbound}
|D| = \left| E[\1(X_{11} =X_{21}) \1(X_{12} = X_{22})] \right| \leq \left(E[\1(X_{11} =X_{21})]\right)^{1/2} \left(E[\1(X_{12} = X_{22})]\right)^{1/2} \leq \gamma^2.
\end{equation}
Turning to the final term, we have
\begin{align*}
A &= E\left[\signbar(X_{11} - X_{21}) \signbar(X_{12} - X_{22})\right] \\
&= E\big[(2a(X_{11}-X_{21})-1)(2a(X_{12}-X_{22})-1)\big] \\
&= 4E[a(X_{11}-X_{21})a(X_{12}-X_{22})] - 2E[a(X_{11}-X_{21})] - 2E[a(X_{12}-X_{22})] + 1 \\
&= \big(4E[a(X_{11}-X_{21})a(X_{12}-X_{22})]-1\big) + 2\big(1-E[a(X_{11}-X_{21})]-E[a(X_{12}-X_{22})]\big) \\
&\defn A_1 + A_2.
\end{align*}
Here, the expectation is with respect to the joint distribution of $(X_{11}, X_{12}, X_{21}, X_{22})$, with density
\begin{align}
f &= [(1-\gamma)\phi_1 + \gamma h_1][(1-\gamma)\phi_2 + \gamma h_2] \nonumber\\
&= (1-\gamma)^2\phi_1\phi_2 + \gamma(1-\gamma)\phi_1h_2 + \gamma(1-\gamma)\phi_2h_1 + \gamma^2h_1h_2.
\label{jointf}
\end{align}
This follows from the fact that the pairs $(X_{11}, X_{12})$ and $(X_{21}, X_{22})$ are independently drawn from the mixture distribution, where $\phi$ is the joint density of $(X_{k1}, X_{k2})$ under $\Phi_\rho$, and $h$ is the joint density of $(X_{k1}, X_{k2})$ under $H$.  Now, let $U = X_{11}-X_{21}$ and $V = X_{12}-X_{22}$. Under the product distribution $\phi_1\phi_2$, the distribution of $(U, V)$ is bivariate normal with mean $\0$ and correlation $\rho$. Hence, 
\begin{equation}
E_{\phi_1\phi_2}[a(U)] = E_{\phi_1\phi_2}[a(V)] = \half, \label{greaterthan}
\end{equation}
and by Lemma~\ref{orthant},
\begin{equation}
E_{\phi_1\phi_2}[a(U)a(V)] = \frac{1}{4}\bigg[1+\frac{2}{\pi}\sin^{-1}(\rho)\bigg]. \label{orthantprob}
\end{equation}

Combining equations~\eqref{jointf} and~\eqref{greaterthan}, we then have
\begin{align*}
&E_f[a(U)] \\
&= (1-\gamma)^2E_{\phi_1\phi_2}[a(U)] + \gamma(1-\gamma)E_{\phi_1h_2}[a(U)] + \gamma(1-\gamma)E_{\phi_2h_1}[a(U)] + \gamma^2E_{h_1h_2}[a(U)] \\
&= \half -\gamma + \half\gamma^2 + \gamma(1-\gamma)E_{\phi_1h_2}[a(U)] + \gamma(1-\gamma)E_{\phi_2h_1}[a(U)] + \gamma^2E_{h_1h_2}[a(U)] \\
&= \half + \left\{-1 + E_{\phi_1h_2}[a(U)] + E_{\phi_2h_1}[a(U)]\right\}\gamma + \left\{\half - E_{\phi_1h_2}[a(U)] - E_{\phi_2h_1}[a(U)] + E_{h_1h_2}[a(U)]\right\}\gamma^2.
\end{align*}
Noting that $E_{\phi_1h_2}[a(U)]$, $E_{\phi_2h_1}[a(U)]$ and $E_{h_1h_2}[a(U)]$ are between 0 and 1, we have
\[\left|E_f[a(U)] - \half\right| \leq \gamma + \frac{3}{2}\gamma^2, \qquad \text{and} \qquad \left|E_f[a(V)] - \half\right| \leq \gamma + \frac{3}{2}\gamma^2.\]
It follows that
\begin{equation}
\label{A2bound}
|A_2| = 2|1 - E_f[a(U)] - E_f[a(V)]| \leq 4\gamma + 6\gamma^2. 
\end{equation}

On the other hand, combining equations~\eqref{jointf} and~\eqref{orthantprob}, we have
\begin{align*}
A_1 &= 4E_f[a(U)a(V)] - 1 \\
&= 4\Big\{(1-\gamma)^2E_{\phi_1\phi_2}[a(U)a(V)] + \gamma(1-\gamma)E_{\phi_1h_2}[a(U)a(V)] \\
&\qquad\qquad\qquad + \gamma(1-\gamma)E_{\phi_2h_1}[a(U)a(V)] + \gamma^2E_{h_1h_2}[a(U)a(V)]\Big\} -1 \\
&= (1-\gamma)^2\bigg[1+\frac{2}{\pi}\sin^{-1}(\rho)\bigg] -1 \\
&\qquad + 4\Big\{\gamma(1-\gamma)E_{\phi_1h_2}[a(U)a(V)] + \gamma(1-\gamma)E_{\phi_2h_1}[a(U)a(V)] + \gamma^2E_{h_1h_2}[a(U)a(V)]\Big\} \\
&= \frac{2}{\pi}\sin^{-1}(\rho) + (-2\gamma+\gamma^2)\bigg[1+\frac{2}{\pi}\sin^{-1}(\rho)\bigg] \\
&\qquad + 4\Big\{\gamma(1-\gamma)E_{\phi_1h_2}[a(U)a(V)] + \gamma(1-\gamma)E_{\phi_2h_1}[a(U)a(V)] + \gamma^2E_{h_1h_2}[a(U)a(V)]\Big\} \\
&= \frac{2}{\pi}\sin^{-1}(\rho) + \bigg\{-2 - \frac{4}{\pi}\sin^{-1}(\rho) + 4E_{\phi_1h_2}[a(U)a(V)] + 4E_{\phi_2h_1}[a(U)a(V)]\bigg\}\gamma \\
&\qquad + \bigg\{1+\frac{2}{\pi}\sin^{-1}(\rho) - 4E_{\phi_1h_2}[a(U)a(V)] -4E_{\phi_2h_1}[a(U)a(V)] + 4E_{h_1h_2}[a(U)a(V)]\bigg\}\gamma^2.
\end{align*}
Noting that the quantities
\begin{align*}
-2 - \frac{4}{\pi}\sin^{-1}(\rho) + 4E_{\phi_1h_2}[a(U)a(V)] + 4E_{\phi_2h_1}[a(U)a(V)]
\end{align*}
and
\begin{align*}
1+\frac{2}{\pi}\sin^{-1}(\rho) - 4E_{\phi_1h_2}[a(U)a(V)] -4E_{\phi_2h_1}[a(U)a(V)] + 4E_{h_1h_2}[a(U)a(V)]
\end{align*}
are both bounded in magnitude by 8, we obtain
\begin{equation}
\label{A1bound}
\left|A_1 - \frac{2}{\pi}\sin^{-1}(\rho)\right| \leq 8\gamma + 8\gamma^2. 
\end{equation}

Combining inequalities \eqref{Bbound}, \eqref{Cbound}, \eqref{Dbound}, \eqref{A2bound} and \eqref{A1bound} then gives
\begin{align*}
\left|\rho^K - \frac{2}{\pi}\sin^{-1}(\rho)\right| &= \left|A_1 + A_2 + B + C + D - \frac{2}{\pi}\sin^{-1}(\rho)\right| \\
&\leq \left|A_1 - \frac{2}{\pi}\sin^{-1}(\rho)\right| + |A_2| + |B| + |C| + |D| \\
&\leq 12\gamma + 17\gamma^2. 
\end{align*}

\iffalse
\begin{align*}
\bigg|\rho^K - \frac{2}{\pi}\sin^{-1}(\rho)\bigg| &\leq |-2\gamma+\gamma^2|\bigg|1+\frac{2}{\pi}\sin^{-1}(\rho)\bigg| \\
&\quad + 4\Big|\gamma(1-\gamma)E_{\phi_1h_2}[a(U)a(V)] + \gamma(1-\gamma)E_{\phi_2h_1}[a(U)a(V)] + \gamma^2E_{h_1h_2}[a(U)a(V)]\Big| \\
&\leq 4\gamma + 4(3\gamma) \\
&\leq 16\gamma.
\end{align*}
\fi
\end{proof}

The second lemma provides an analogous result to Lemma~\ref{exprho}, this time for the Spearman's rho statistic.

\begin{Lemma}\label{exprho}
Let $(X_{k1}, X_{k2})$, for $k = 1, \ldots, n$, be a random sample from
\[F = (1-\gamma)\Phi_\rho + \gamma H,\]
where $\Phi_\rho$ is a bivariate normal distribution with correlation $\rho$, and $H$ is an arbitrary bivariate distribution.  Let $r^S$ be the Spearman's rho statistic, and suppose the samples $\{X_{ki}: k = 1, \ldots, n\}$ are unique. Then
\begin{equation*}
E_F(r^S) = \frac{6}{\pi}\sin^{-1}\left(\frac{\rho}{2}\right) + R, \label{exprhoequal}
\end{equation*}
where $|R| \leq 48\gamma + 129\gamma^2 + 88\gamma^3 + \frac{12}{n+1}$.
\end{Lemma}
\begin{proof}
Let $\rho^K = E_F(r^K)$ be the population version of Kendall's tau correlation.  By Lemma~\ref{rhodecomp}, we have
\begin{align}
E_F(r^S) &= \frac{3(n-2)}{n+1} \cdot E\left[\sign(X_{11}-X_{21})\sign(X_{12}-X_{32})\right] + \frac{3}{n+1}\rho^K \nonumber\\
&= 3 E\left[\sign(X_{11}-X_{21})\sign(X_{12}-X_{32})\right] \nonumber\\
&\qquad + \frac{3}{n+1}(\rho^K-3E\left[\sign(X_{11}-X_{21})\sign(X_{12}-X_{32})\right]). \label{Erhodecomp}
\end{align}
Note that the second term is clearly bounded in magnitude by $\frac{12}{n+1}$. Now define $a(X) = \1(X>0)$, and let $\signbar(X) = 2a(X) - 1$. Then $\sign(X) = \signbar(X) - \1(X=0)$.  It follows that
\begin{align*}
&E\left[\sign(X_{11}-X_{21})\sign(X_{12}-X_{32})\right] \\
&= E\left[\signbar(X_{11} - X_{21}) \signbar(X_{12} - X_{32})\right] - E[\1(X_{11} = X_{21})\signbar(X_{12} - X_{32})] \\
&\qquad - E\left[\signbar(X_{11} - X_{21}) \1(X_{12} = X_{32})\right] + E\left[\1(X_{11} = X_{21})\1(X_{12} = X_{32})\right] \\
& \defn A + B + C + D.
\end{align*}
A similar argument as in the proof of Lemma~\ref{exptau} yields
\begin{equation}
\label{BCDbound}
\max\{|B|, |C|, |D|\} \leq \gamma^2,
\end{equation}
and
\begin{align*}
A &= \big(4E[a(X_{11}-X_{21})a(X_{12}-X_{32})]-1\big) + 2\big(1-E[a(X_{11}-X_{21})]-E[a(X_{12}-X_{32})]\big) \\
&\defn A_1 + A_2.
\end{align*}
Here, the expectation is with respect to the joint distribution of $(X_{11}, X_{12}, X_{21}, X_{22}, X_{31}, X_{32})$, with density
\begin{align}
f &= [(1-\gamma)\phi_1 + \gamma h_1][(1-\gamma)\phi_2 + \gamma h_2][(1-\gamma)\phi_3 + \gamma h_3] \nonumber\\
&= (1-\gamma)^3\phi_1\phi_2\phi_3 + \gamma(1-\gamma)^2[\phi_1\phi_2h_3 + \phi_1\phi_3h_2 + \phi_2\phi_3h_1] \nonumber\\
&\qquad + \gamma^2(1-\gamma)[\phi_1h_2h_3 + \phi_2h_1h_3 + \phi_3h_1h_2] + \gamma^3h_1h_2h_3. \label{jointf2}
\end{align}
Now let $U = X_{11}-X_{21}$ and $V = X_{12}-X_{32}$.  Under the product distribution $\phi_1\phi_2\phi_3$, the distribution of $(U, V)$ is bivariate normal with mean $0$ and correlation $\rho/2$. Hence, 
\begin{equation}
E_{\phi_1\phi_2\phi_3}[a(U)] = E_{\phi_1\phi_2\phi_3}[a(V)] = \half, \label{greaterthan2}
\end{equation}
and by Lemma~\ref{orthant},
\begin{equation}
E_{\phi_1\phi_2\phi_3}[a(U)a(V)] = \frac{1}{4}\bigg[1+\frac{2}{\pi}\sin^{-1}\left(\frac{\rho}{2}\right)\bigg]. \label{orthantprob2}
\end{equation}
Combining equations~\eqref{jointf2} and~\eqref{greaterthan2}, and noting that $E[a(U)]$ is between 0 and 1, we then have
\begin{align*}
&E_f[a(U)] \\
&= (1-\gamma)^3E_{\phi_1\phi_2\phi_3}[a(U)] + \gamma(1-\gamma)^2\big\{E_{\phi_1\phi_2h_3}[a(U)] + E_{\phi_1\phi_3h_2}[a(U)] + E_{\phi_2\phi_3h_1}[a(U)]\big\} \\
&\qquad + \gamma^2(1-\gamma)\big\{E_{\phi_1h_2h_3}[a(U)] + E_{\phi_2h_1h_3}[a(U)] + E_{\phi_3h_1h_2}[a(U)]\big\} + \gamma^3E_{h_1h_2h_3}[a(U)] \\
&= \half - \frac{3}{2}\gamma + \frac{3}{2}\gamma^2 - \half\gamma^3 + \gamma(1-\gamma)^2\big\{E_{\phi_1\phi_2h_3}[a(U)] + E_{\phi_1\phi_3h_2}[a(U)] + E_{\phi_2\phi_3h_1}[a(U)]\big\} \\
&\qquad + \gamma^2(1-\gamma)\big\{E_{\phi_1h_2h_3}[a(U)] + E_{\phi_2h_1h_3}[a(U)] + E_{\phi_3h_1h_2}[a(U)]\big\} + \gamma^3E_{h_1h_2h_3}[a(U)] \\
&= \half + c\gamma + d\gamma^2 + e\gamma^3,
\end{align*}
where $|c|\leq\frac{3}{2}, |d|\leq\frac{9}{2}$, and $|e|\leq\frac{7}{2}$.  It follows that
\[\left|E_f[a(U)] - \half\right| \leq \frac{3}{2}\gamma + \frac{9}{2}\gamma^2 + \frac{7}{2}\gamma^3, \qquad \text{and} \qquad \left|E_f[a(V)] - \half\right| \leq \frac{3}{2}\gamma + \frac{9}{2}\gamma^2 + \frac{7}{2}\gamma^3,\]
so
\begin{equation}
\label{A2bound2}
|A_2| = 2|1 - E_f[a(U)] - E_f[a(V)]| \leq 6\gamma + 18\gamma^2 + 14\gamma^3. 
\end{equation}

Furthermore, combining equations~\eqref{jointf2} and~\eqref{orthantprob2}, we have
\begin{align*}
A_1 &= 4E_f[a(U)a(V)] - 1 \\
&= 4\bigg\{(1-\gamma)^3E_{\phi_1\phi_2\phi_3}[a(U)a(V)] \\
&\qquad + \gamma(1-\gamma)^2\big\{E_{\phi_1\phi_2h_3}[a(U)a(V)] + E_{\phi_1\phi_3h_2}[a(U)a(V)] + E_{\phi_2\phi_3h_1}[a(U)a(V)]\big\}\\
&\qquad + \gamma^2(1-\gamma)\big\{E_{\phi_1h_2h_3}[a(U)a(V)] + E_{\phi_2h_1h_3}[a(U)a(V)] + E_{\phi_3h_1h_2}[a(U)a(V)]\big\} \\
&\qquad + \gamma^3E_{h_1h_2h_3}[a(U)a(V)]\bigg\} -1 \\
&= (1-\gamma)^3\bigg[1+\frac{2}{\pi}\sin^{-1}\left(\frac{\rho}{2}\right)\bigg] -1 \\
&\qquad + 4\bigg\{\gamma(1-\gamma)^2\big\{E_{\phi_1\phi_2h_3}[a(U)a(V)] + E_{\phi_1\phi_3h_2}[a(U)a(V)] + E_{\phi_2\phi_3h_1}[a(U)a(V)]\big\}\\
&\qquad\qquad + \gamma^2(1-\gamma)\big\{E_{\phi_1h_2h_3}[a(U)a(V)] + E_{\phi_2h_1h_3}[a(U)a(V)] + E_{\phi_3h_1h_2}[a(U)a(V)]\big\} \\
&\qquad\qquad + \gamma^3E_{h_1h_2h_3}[a(U)a(V)]\bigg\} \\
&= \frac{2}{\pi}\sin^{-1}\left(\frac{\rho}{2}\right) + (-3\gamma+3\gamma^2-\gamma^3)\bigg[1+\frac{2}{\pi}\sin^{-1}\left(\frac{\rho}{2}\right)\bigg] \\
&\qquad + 4\bigg\{\gamma(1-\gamma)^2\big\{E_{\phi_1\phi_2h_3}[a(U)a(V)] + E_{\phi_1\phi_3h_2}[a(U)a(V)] + E_{\phi_2\phi_3h_1}[a(U)a(V)]\big\}\\
&\qquad\qquad + \gamma^2(1-\gamma)\big\{E_{\phi_1h_2h_3}[a(U)a(V)] + E_{\phi_2h_1h_3}[a(U)a(V)] + E_{\phi_3h_1h_2}[a(U)a(V)]\big\} \\
&\qquad\qquad + \gamma^3E_{h_1h_2h_3}[a(U)a(V)]\bigg\} \\
&= \frac{2}{\pi}\sin^{-1}\left(\frac{\rho}{2}\right) + c'\gamma + d'\gamma^2 + e'\gamma^3,
\end{align*}
where $|c'| \leq 10, |d'|\leq 22$, and $|e'|\leq \frac{46}{3}$.  Hence, we obtain
\begin{equation}
\label{A1bound2}
\left|A_1 - \frac{2}{\pi}\sin^{-1}\left(\frac{\rho}{2}\right)\right| \leq 10\gamma + 22\gamma^2 + \frac{46}{3}\gamma^3. 
\end{equation}

Combining inequalities \eqref{Erhodecomp}, \eqref{BCDbound}, \eqref{A2bound2} and \eqref{A1bound2}, we then obtain
\begin{align*}
\left|E_F(r^S) - \frac{6}{\pi}\sin^{-1}\left(\frac{\rho}{2}\right)\right| &\leq 3\left|A_1 + A_2 + B + C + D - \frac{2}{\pi}\sin^{-1}\left(\frac{\rho}{2}\right)\right| + \frac{12}{n+1}  \\
&\leq 3\bigg\{\left|A_1 - \frac{2}{\pi}\sin^{-1}\left(\frac{\rho}{2}\right)\right| + |A_2| + |B| + |C| + |D|\bigg\} + \frac{12}{n+1} \\
&\leq 48\gamma + 129\gamma^2 + 88\gamma^3 + \frac{12}{n+1}. 
\end{align*}
\end{proof}

The following lemma comes from \cite{Hoeffding1948}:
\begin{Lemma}\label{rhodecomp}
Suppose the samples $\{X_{ki}: k = 1, \ldots, n\}$ are unique, for $i = 1, 2$. The Spearman's rho correlation can be decomposed as 
\[r^S = \frac{n-2}{n+1}U + \frac{3}{n+1}r^K,\]
where $\htau$ is the Kendall's tau correlation, and $U$ is a $U$-statistic of order 3 with corresponding symmetric kernel
\[\psi_U(X_1, X_2, X_3) = \frac{1}{3!}\sum_{(i_1, i_2, i_3)\in\text{perm}(1, 2, 3)}3 \cdot \sign(X_{i_11}-X_{i_21})\; \sign(X_{i_12}-X_{i_32}),\]
and the summation is taken over all possible permutations of the three arguments.
\end{Lemma}

The proof of the following lemma is adapted from an argument in \cite{CroDeh10}.
\begin{Lemma}\label{orthant}
Suppose $(X, Y)$ follows a bivariate normal distribution with mean $0$ and correlation $\rho$.  Then
\[E[a(X)a(Y)] = P(X>0, Y>0) = \frac{1}{4}\bigg[1+\frac{2}{\pi}\sin^{-1}(\rho)\bigg].\]
\end{Lemma}
\begin{proof}
Recall that we may write 
\[Y = \rho X + \sqrt{1-\rho^2}Z,\]
where $(X,Z) \sim N(0, I_2)$. Furthermore, we have the polar coordinate representation
\[(X, Z) = (R\cos\theta, R\sin\theta),\]
where $\theta\sim\text{Uniform}(-\pi, \pi]$, and $R$ follows a Rayleigh distribution. Then
\begin{equation*}
Y = R\left(\rho \cos(\theta) + \sqrt{1-\rho^2} \sin(\theta)\right),
\end{equation*}
which has the convenient representation $Y = R\sin(\alpha + \theta)$, where $\alpha = \sin^{-1}(\rho)$. It follows that
\begin{multline*}
P(X>0, Y>0) = P(\cos\theta>0, \sin(\alpha+\theta)>0) \\
= P\bigg(\theta\in\Big[-\alpha, \frac{\pi}{2}\Big]\bigg) = \frac{\frac{\pi}{2}+\alpha}{2\pi} = \frac{1}{4}\bigg[1+\frac{2}{\pi}\sin^{-1}(\rho)\bigg].
\end{multline*}
\end{proof}

Finally, we have a simple lemma concerning the Lipschitz behavior of the normal quantile function:

\begin{Lemma}\label{quantile}
The standard normal quantile function $\Phi^{-1}: [0, 1] \rightarrow\R$,  when restricted to the domain $[0.2, 0.8]$, is Lipschitz continuous with Lipschitz constant $3.6$; i.e.,
\[|\Phi^{-1}(a)-\Phi^{-1}(b)| \leq 3.6|a-b|, \qquad \forall a, b\in[0.2, 0.8].\]
\end{Lemma}
\begin{proof}
It suffices to check that $|\frac{d}{dy}\Phi^{-1}(y)| \leq 3.6$, for $y\in[0.2, 0.8]$.  Since $[\Phi^{-1}]'(\Phi(x))\cdot\Phi'(x) = \frac{d}{dx}\Phi^{-1}(\Phi(x)) = \frac{d}{dx}x = 1$, we have
\[[\Phi^{-1}]'(\Phi(x)) = \frac{1}{\Phi'(x)}, \qquad\forall x\in\R.\]
For $y = \Phi(x) \in [0.2, 0.8]$, we have $x \in [-0.8416, 0.8416]$, and for such $x$'s,
\[[\Phi^{-1}]'(\Phi(x)) = \frac{1}{\Phi'(x)} = \sqrt{2\pi}\exp\bigg(\frac{1}{2}x^2\bigg) \leq\sqrt{2\pi}\exp\bigg(\frac{1}{2}\cdot 0.8416^2\bigg) \leq 3.6.\]
This concludes the proof.
\end{proof}
\end{document}